\documentclass[10pt]{article}
\usepackage[left=3cm, right=3cm]{geometry}
\usepackage[english]{babel}
\usepackage{amsmath}
\usepackage{amsfonts}
\usepackage{amsthm}
\usepackage{amssymb}
\usepackage{algorithm}
\usepackage{algorithmic}
\usepackage[font=small,labelfont=bf]{caption}

\usepackage{xspace}
\usepackage{euscript}
\usepackage{graphicx}
\usepackage{caption}
\usepackage{subcaption}
\usepackage{amscd}
\usepackage{tabularx}
\usepackage{lipsum}
\usepackage{array}
\usepackage{mwe}
\usepackage{graphbox}
\usepackage{calc}

\usepackage{enumerate}
\usepackage{epsfig} %% add this line and the next only if you have pictures
\usepackage{graphics} %% pictures should be in esp format
\theoremstyle{plain}
\newtheorem{theo}{Theorem}[section]%[Definizioni]
\newtheorem{lem}[theo]{Lemma}%[Definizioni]
\newtheorem{prop}[theo]{Proposition}%[Definizioni]
\theoremstyle{definition}
\newtheorem{definition}[theo]{Definition}%[Definizioni]
\theoremstyle{remark}
\numberwithin{equation}{section}

\usepackage{blindtext}
\title{Scattering Coefficients of Inhomogeneous Objects and Their Application in Target Classification in Wave Imaging\thanks{\footnotesize  
This work was supported by the SNF grant 200021-172483.}}
\author{Lorenzo Baldassari\thanks{Department of Mathematics, ETH Z\"urich, R\"amistrasse 101, CH-8093 Z\"urich, Switzerland (\texttt{lorenzo.baldassari@sam.math.ethz.ch}).}}
\date{}
\begin{document}
\setcounter{section}{0}
\setcounter{secnumdepth}{2}

\maketitle

\begin{abstract}
The aim of this paper is to provide and numerically test in the presence of measurement noise a procedure for target classification in wave imaging based on comparing frequency-dependent distribution descriptors with  precomputed  ones in a dictionary of learned distributions. Distribution descriptors for inhomogeneous objects are obtained from the scattering coefficients. First, we extract the scattering coefficients of the (inhomogeneous) target from the perturbation of the echoes. Then, for a collection of inhomogeneous targets, we build a frequency-dependent dictionary of distribution descriptors and use a matching algorithm in order to identify a target from the dictionary up to some translation, rotation and scaling.

\medskip
	
	
\medskip
	
	
\end{abstract}

\section{Introduction}
There are several geometric and physical quantities associated with target classification such as eigenvalues and capacities \cite{Isoperimetric}. The concept of the scattering coefficients is one of them. The notion appears naturally when we describe the perturbation of sounds emitted by animals such bats and dolphins due to the presence of targets whose material parameters (permeability and permettivity) are different from the ones of the background \cite{echolocation, Bat perception, Bat echolocation}. 

To mathematically introduce the concept of the scattering coefficients, we consider the Helmholtz problem in $\mathbb{R}^2$ for a given fixed frequency $\omega >0$:
\begin{equation}
\begin{cases} \nabla \cdot (\chi (\mathbb{R}^2 \setminus \bar{B}) + \frac{1}{\sigma} \chi (B)) \nabla u + \omega^2 ( \chi (\mathbb{R}^2 \setminus \bar{B}) + \mu \chi (B)) u = 0 & \mbox{ in } \mathbb{R}^2, \\ \left |\frac{\partial (u - U)}{\partial|x|}-i\omega(u-U) \right | \leq \frac{K}{|x|^{\frac{3}{2}}} & \mbox{ if } |x| \to \infty.
\end{cases}
\label{helm-introd}
\end{equation}
Here, $K$ is a positive constant, $B$ is the target embedded in $\mathbb{R}^2$ with Lipschitz boundary, $\chi (B)$ (resp. $\chi (\mathbb{R}^2 \setminus \bar{B})$) is the characteristc function of $B$ (resp. $\mathbb{R}^2 \setminus \bar{B}$), the positive constants $\sigma$ and $\mu$ are the magnetic permeability and electric permettivity of the target, which are supposed to be different from the background permettivity and permeability (the constant function $1$), $U$ is the background solution, i.e., a given solution to $(\Delta + \omega^2)U=0$, and the solution $u$ to \eqref{helm-introd} represents the perturbed wave. The perturbation $u-U$ due to the presence of the permettivity and permeability target $B$ admits the following asymptotic expansion as $|x| \to \infty$, see \cite{echolocation}:
\begin{equation} 
u(x)-U(x)=-\frac{i}{4} \sum_{n \in \mathbb{Z}} H^{(1)}_n(\omega |x|) e^{in \theta_x} \sum_{m \in \mathbb{Z}} W_{n,m}[B, \sigma, \mu, \omega] a_m(U),
\label{asympt0}
\end{equation}
where $H^{(1)}_n$ are the Hankel functions of the first kind of order $n$ and $a_m(U)$ are constants such that $U(x)=\sum_{m\in \mathbb{Z}} a_m(U)J_m(\omega |x|) e^{im \theta_x}$ with $J_m$ being the Bessel function of order $m$. The building blocks $W_{n,m}[B, \sigma, \mu, \omega]$ for the asymptotic expansion \eqref{asympt0} are called the scattering coefficients. Note that the scattering coefficients $W_{n,m}[B, \sigma, \mu, \omega]$ can be reconstructed from the far-field measurements of $u$ by a least-squared method. A stability analysis of the reconstruction is provided in \cite{echolocation}.  

This paper extends the results of \cite{echolocation} to targets with inhomogeneous permettivities and permeabilities. The concept of inhomogeneous scattering coefficients were first introduced in \cite{heterogeneous} and used later in \cite{heterogeneous2} to prove resolution enhancement in high-contrast media. It is the purpose of this paper to extend the notion of scattering coefficients to objects with inhomogeneous permittivities and permeabilities and show their application in target classification. First, we prove important properties of the scattering coefficients such as translation, rotation and scaling formulas. Then, we construct distribution descriptors for multiple frequencies based on scaling, rotation, and translation properties of the scattering coefficients. Finally, we use a target identification algorithm in order to identify an inhomogeneous target from a dictionary of precomputed frequency-dependent distribution descriptors up to some translation, rotation and scaling. For the sake of simplicity, throughout this paper, we focus on two-dimensional models. However, the results can be easily extended to three dimensions. 

The paper is organized as follows. In Section \ref{sec2} we introduce the scattering coefficients for inhomogeneous targets and prove that they are the building blocks of the far-field expansion of the wave perturbation. Section \ref{sec3} is devoted to the derivation of integral representations of the inhomogeneous scattering coefficients. In Section \ref{sec4} we prove important properties for the inhomogeneous scattering coefficients, such as the exponential decay of the scattering coefficients. We also show translation, rotation and scaling property for the scattering coefficients. In Section \ref{sec5} we construct the translation- and rotation-invariant distribution descriptor. We also observe that the inhomogeneous scattering coefficients are nothing else but the Fourier coefficients of the far-field pattern. In Section \ref{sec6} we present numerical results in order to demonstrate the theoretical framework presented in previous sections. In particular, we investigate the identification of a target by the reconstruction of scattering coefficients from the measurements of the multistatic response matrix. A few concluding remarks are given in Section \ref{sec7}.  In Appendix \ref{secappendix}, we provide integral representations for the case of piecewise constant (inhomogeneous) material parameters. In Appendix \ref{secappendix2}, we present results of target identification using a full-view setting with no noise ($\sigma=0\%$).  

\section{Scattering coefficients and asymptotic expansions} \label{sec2}
Let $\frac{1}{\sigma}$ be a  bounded measurable function in $\mathbb{R}^2$ such that $\frac{1}{\sigma} - 1$ is compactly supported and 
$$0<\lambda_1 \leq \frac{1}{\sigma} \leq \lambda_2,$$
where $\lambda_1, \lambda_2$ are constants. Let $\mu$ be a  bounded measurable function in $\mathbb{R}^2$ such that $\mu - 1$ is compactly supported and 
$$0<\lambda_3 \leq \mu \leq \lambda_4,$$
where $\lambda_3, \lambda_4$ are constants. For a given fixed frequency $\omega>0$, we consider the following Helmholtz problem:
\begin{equation}
\begin{cases} \nabla \cdot \frac{1}{\sigma} \nabla u + \omega^2 \mu u = 0 & \mbox{ in } \mathbb{R}^2, \\ \left |\frac{\partial (u - U)}{\partial |x|}-i\omega(u-U) \right | \leq \frac{K}{|x|^{\frac{3}{2}}} & \mbox{ if } |x| \to \infty,
\end{cases}
\label{helm}
\end{equation}
where $U$ is a given solution to
\begin{equation}
\Delta U + \omega^2 U = 0,
\label{background}
\end{equation}
and $| \cdot |$ denotes the Euclidean norm of $\mathbb{R}^2$. In this section, we derive a full far-field expansion of $(u-U)(x)$ as $|x| \to \infty$. In the course of doing so, the notion of inhomogeneous scattering coefficients appears naturally.

Let $B$ a bounded domain in $\mathbb{R}^2$ with Lipschitz boundary. We assume that $B$ is such that
$$\mbox{supp}\left(\frac{1}{\sigma} - 1\right) \subset B,$$
$$\mbox{supp}(\mu - 1) \subset B.$$
Suppose that $B$ contains the origin. Note that \eqref{helm} is equivalent to
\begin{equation}
\begin{cases} \nabla \cdot \frac{1}{\sigma} \nabla u + \omega^2 \mu u = 0 & \mbox{ in } B, \\ \Delta u + \omega^2 u = 0 & \mbox{ in } \mathbb{R}^2 \setminus \bar{B}, \\ u|_{+} - u|_{-} = 0 & \mbox{ on } \partial B, \\ \nu \cdot \nabla u |_{+} - \nu \cdot \frac{1}{\sigma} \nabla u |_{-}=0 & \mbox{ on } \partial B, \\ \left |\frac{\partial (u - U)}{\partial |x|}-i \omega (u-U) \right | \leq \frac{K}{|x|^{\frac{3}{2}}} & \mbox{ if } |x| \to \infty,
\end{cases}
\label{transm}
\end{equation}
where $\nu$ is the outward normal vector at some $x \in \partial B$ and the subscripts $\pm$ indicate the limits from outside and inside $B$, respectively. 

In two dimensions, the fundamental solution $\Gamma_\omega(x)$ to the the Helmholtz equation 
$$(\Delta + \omega^2)\Gamma_\omega(x)=\delta_0(x)$$
subject to the Sommerfeld outgoing radiation condition is given by $$\Gamma_\omega(x)=-\frac{i}{4}H^{(1)}_0(\omega|x|).$$

Assume that $\omega^2$ is not a Neumann eigenvalue of $-\frac{1}{\mu(x)}\nabla \cdot \frac{1}{\sigma(x)} \nabla$ on $B$. Let $N_{\sigma,\mu}$ be the Neumann function of problem 
\begin{equation}
\begin{cases} \nabla \cdot \frac{1}{\sigma} \nabla u + \omega^2 \mu u = 0 & x \in B, \\ \frac{1}{\sigma} \frac{\partial u}{\partial \nu} = g & x \in \partial B,
\end{cases}
\label{ntd}
\end{equation} 
that is, for each fixed $z \in B$, $N_{\sigma,\mu}(z, \cdot)$ is solution to 
\begin{equation}
\begin{cases} \nabla_y \cdot \frac{1}{\sigma(y)} \nabla_y N_{\sigma,\mu}(z,y) + \omega^2 \mu(y) N_{\sigma,\mu}(z,y) = -\delta_z(y) & y \in B, \\ \frac{1}{\sigma(y)} \nabla _yN_{\sigma,\mu}(z,y) \nu_y =0 & y \in \partial B.
\end{cases}
\label{neumann}
\end{equation}
We can prove the following result.

\begin{prop} 
The function defined by 
\begin{equation}
\mathcal{N}_{B,\sigma, \mu}[g](x):=\int_{\partial B} N_{\sigma, \mu}(x,y)g(y) \; dS_y, \; \; x \in B, 
\label{NtD}
\end{equation}
is the solution to \eqref{ntd}. Moreover, the Neumann-to-Dirichlet (NtD) map $\Lambda_{\sigma, \mu}: H^{-\frac{1}{2}}(\partial B) \to H^{\frac{1}{2}}(\partial B)$ is well-defined, invertible and 
$$\Lambda_{\sigma, \mu}[g](x)=\left. \mathcal{N}_{B,\sigma,\mu}[g]\right|_{\partial B}(x)=u|_{\partial B}(x) \mbox{ for } x \in \partial B.$$
\end{prop}

\begin{proof}
By \eqref{ntd}, for each fixed $x \in B$, we have
$$\left(\nabla_y \cdot \frac{1}{\sigma(y)} \nabla_y u(y) + \omega^2 \mu(y) u(y) \right) N_{\sigma,\mu}(x,y)=0.$$
By integrating over $B$ and applying Green's formula, we have
\begin{equation*}
\begin{aligned}
& \int_{\partial B} g(y) N_{\sigma, \mu}(x,y) \; dS_y - \int_{\partial B} u(y) \frac{1}{\sigma(y)} \nabla_y N_{\sigma, \mu}(x,y) \cdot \nu_y \; dS_y \\ & + \int_B  u(y) \left(\nabla_y \cdot \frac{1}{\sigma(y)} \nabla_y N_{\sigma, \mu}(x,y) + \omega^2 \mu(y) N_{\sigma, \mu}(x,y) \right) \; dy = 0.
\end{aligned}
\end{equation*}
From \eqref{neumann}, it follows that
$$u(x)= \int_{\partial B} N_{\sigma, \mu}(x,y)g(y) \; dS_y.$$
The second part of the proposition follows from the fact that $\omega^2$ is not a Neumann eigenvalue of $-\Delta$ in $B$.
\end{proof}

We can prove the following proposition.

\begin{prop}
For each $x \in \mathbb{R}^2 \setminus \bar{B}$, if $u$ is the solution to \eqref{helm} and $U$ is such that \eqref{background} holds, then
\begin{equation}
\begin{aligned}
(u-U)(x) = \int_{\partial  B } g(y) \Gamma_\omega(x-y) \; dS_y - \int_{\partial B} \Lambda_{\sigma,\mu}[g](y) \left. \frac{\partial \Gamma_\omega}{\partial \nu_y}\right|_{+}(x-y) \; dS_y.
\end{aligned}
\label{asympt2}
\end{equation}
\end{prop}

\begin{proof}
Observe that 
$$\Delta(u-U)(y) + \omega^2 (u-U)(y) =0 \mbox{ for } y \in \mathbb{R}^2 \setminus \bar{B}.$$
Hence, for each fixed $x \in \mathbb{R}^2 \setminus \bar{B}$,
$$\int_{\mathbb{R}^2 \setminus \bar{B}} \left[ \Delta(u-U)(y) + \omega^2 (u-U)(y) \right] \Gamma_\omega(x-y) \; dy = 0.$$
Let $R>0$ be such that $B \subset B_R$. By Green's formula, we have 
\begin{equation}
\begin{aligned}
& - \int_{B_R \setminus \bar{B}} \nabla_y (u-U)(y) \nabla_y \Gamma_\omega(x-y) \; dy + \int_{\partial (B_R \setminus \bar{B})} \frac{\partial (u-U)}{\partial \nu_y}(y) \Gamma_\omega(x-y) \; dS_y \\ & + \int_{B_R \setminus \bar{B}} \omega^2 (u-U)(y) \Gamma_\omega (x-y) \; dy= \\ & = \int_{\partial (B_R \setminus \bar{B})} \frac{\partial (u-U)}{\partial \nu_y}(y) \Gamma_\omega(x-y) \; dS_y - \int_{\partial (B_R \setminus \bar{B})} (u-U)(y) \frac{\partial \Gamma_\omega}{\partial \nu_y}(x-y) \; dS_y \\ & + \int_{B_R \setminus \bar{B}} (u-U)(y)[(\Delta + \omega^2)\Gamma_\omega(x-y)] \; dy = \\ & =0.
\end{aligned}
\end{equation}
Given $x \in \mathbb{R}^2 \setminus \bar{B}$, for $R$ large enough, \begin{equation*}
\begin{aligned}
\int_{B_R \setminus \bar{B}} (u-U)(y)[(\Delta + \omega^2)\Gamma_\omega(x-y)] \; dy & = \int_{B_R \setminus \bar{B}} (u-U)(y) \delta_0(x-y) \; dy = \\ & = (u-U)(x).
\end{aligned}
\end{equation*}
Then, we obtain
\begin{equation}
\begin{aligned}
(u-U)(x) =  & - \int_{\partial B} (u-U)(y) \left. \frac{\partial \Gamma_\omega}{\partial \nu_y}\right|_{+}(x-y) \; dS_y +  \int_{\partial  B } \frac{\partial (u-U)}{\partial \nu_y}(y) \Gamma_\omega(x-y) \; dS_y = \\ = & - \int_{\partial B} u(y) \left. \frac{\partial \Gamma_\omega}{\partial \nu_y}\right|_{+}(x-y) \; dS_y +  \int_{\partial  B } \frac{\partial u}{\partial \nu_y}(y) \Gamma_\omega(x-y) \; dS_y,
\end{aligned}
\label{asympt1}
\end{equation}
where the second equality holds from Green's formula and $\Delta U = - \omega^2 U$:
\begin{equation*}
\begin{aligned}
& \int_{\partial B} U(y) \left. \frac{\partial \Gamma_\omega}{\partial \nu_y}\right|_{+}(x-y) \; dS_y -  \int_{\partial  B } \frac{\partial U}{\partial \nu_y}(y) \Gamma_\omega(x-y) \; dS_y = \\ & = \int_B U(y) \Delta \Gamma_\omega (x-y) - \Delta U(y) \Gamma_\omega (x-y) \; dy = \\ & = \int_B U(y) (\delta_0(x-y) - \omega^2 \Gamma_\omega(x-y)) + \omega^2 U(y) \Gamma_\omega(x-y) \; dy = \\ & = \int_B U(y) \delta_0(x-y) \; dy = 0.
\end{aligned}
\end{equation*} 
Thus, from the transmission conditions, it follows that \eqref{asympt1} can be rewritten as
\begin{equation*}
\begin{aligned}
(u-U)(x) = - \int_{\partial B} \left. \Lambda_{\sigma,\mu}[g](y) \frac{\partial \Gamma_\omega}{\partial \nu_y}\right|_{+}(x-y) \; dS_y +  \int_{\partial  B } g(y) \Gamma_\omega(x-y) \; dS_y.
\end{aligned}
\end{equation*}
\end{proof}

For $x \in \mathbb{R}^2 \setminus \bar{B}$, we have
$$\Lambda_{1,1}\left(\frac{\partial \Gamma_\omega}{\partial \nu_y}(x- \cdot)\right)= \Gamma_\omega (x-\cdot) \mbox{ on } \partial B,$$
and hence 
\begin{equation} \int_{\partial B} \left. \frac{\partial \Gamma_\omega}{\partial \nu_y}\right|_{+}(x-y) \Lambda_{\sigma,\mu}[g](y)  \; dS_y = \int_{\partial B} \Gamma_\omega(x-y) \Lambda^{-1}_{1,1} \Lambda_{\sigma, \mu} [g](y) \; dS_y,
\label{asympt3}
\end{equation}
which is a consequence of the fact that $\Lambda_{1,1}$ is invertible and self-adjoint. The following result holds.

\begin{prop}
For each $u,v \in L^2(\partial B)$ solutions of \eqref{helm},
$$\left\langle \Lambda_{\sigma,\mu} \left[\frac{1}{\sigma} \frac{\partial u}{\partial \nu} \right], \frac{1}{\sigma} \frac{\partial v}{\partial \nu} \right\rangle_{L^2(\partial B)} = \left\langle \frac{1}{\sigma} \frac{\partial u}{\partial \nu}, \Lambda_{\sigma,\mu}\left[ \frac{1}{\sigma} \frac{\partial v}{\partial \nu}\right]\right\rangle_{L^2(\partial B)}.$$ 
\end{prop}
\begin{proof}
For the sake of simplicity, let us prove the case in which $\sigma = \mu = 1$. By subtracting $\Delta u \bar{v}+ \omega^2 u \bar{v} =0$ to $\Delta \bar{v} u + \omega^2 \bar{v}u =0$, we get
$$\Delta u \bar{v}- \Delta \bar{v} u =0.$$
By Green's formula,
\begin{equation*} 
\begin{aligned} 
\left\langle \Lambda_{1,1} \left[\frac{\partial u}{\partial \nu} \right], \frac{\partial v}{\partial \nu} \right\rangle_{L^2(\partial B)} & = \int_{\partial B} u \frac{\partial \bar{v}}{\partial \nu} = \\ & = \int_B \Delta \bar v u - \Delta u \bar v + \int_{\partial B} \bar{v} \frac{\partial u}{\partial \nu} = \\ & = \left\langle \frac{\partial u}{\partial \nu}, \Lambda_{1,1}\left[\frac{\partial v}{\partial \nu}\right]\right\rangle_{L^2(\partial B)}.
\end{aligned}
\end{equation*}
\end{proof} 

By \eqref{asympt3}, for $x \in \mathbb{R}^2 \setminus \bar{B}$, formula \eqref{asympt2} becomes
\begin{equation} (u-U)(x)=\int_{\partial B} \Gamma _\omega(x-y) \Lambda_{1,1}^{-1} (\Lambda_{1,1} - \Lambda_{\sigma,\mu})[g](y)  \; dS_y.
\label{asympt4}
\end{equation}
 For $|x|>|y|$, by Graf's addition formula \cite{Graf formula}, 
$$\Gamma_\omega(x-y)= -\frac{i}{4} \sum_{n \in \mathbb{Z}} H^{(1)}_n(\omega |x|) e^{in\theta_x} J_n(\omega |y|) e^{-in\theta_y},$$
where $x=(|x|, \theta_x)$, $y=(|y|, \theta_y)$, $H^{(1)}_n$ is the Hankel function of the first kind of order $n$ and $J_n$ is the Bessel function of the first kind of order $n$. In the following, we use $\mathcal{C}_m$ to denote the cylindrical wave of index $m \in \mathbb{Z}$ and of wave number $\omega$, which is defined by
\begin{equation}
\mathcal{C}_m(x) = \mathcal{C}_{m,\omega} (y) := J_m(\omega |y|) e^{im\theta_y}.
\label{cylindrical}
\end{equation}
Hence, \eqref{asympt4} becomes:
\begin{equation} (u-U)(x)= -\frac{i}{4} \sum_{n \in \mathbb{Z}} H^{(1)}_n(\omega |x|) e^{in\theta_x} \int_{\partial B}  \overline{\mathcal{C}_{n}(y)} \Lambda_{1,1}^{-1} (\Lambda_{1,1} - \Lambda_{\sigma,\mu})[g](y)  \; dS_y.
\label{asympt5}
\end{equation}
For each $m \in \mathbb{Z}$, let $u_m$ be the solution to \eqref{helm} when  $\mathcal{C}_m$ is the source term. For $x \in \partial B$, let us define
$$g_m(x):=\left. \frac{1}{\sigma} \frac{\partial u_m}{\partial \nu}\right|_{-}(x).$$
Since the family of cylindrical waves $(\mathcal{C}_m(x))_m$ is complete \cite{math and stat method}, $U$ admits the following expansion:
$$U(x)= \sum_{m \in \mathbb{Z}} a_m(U) \mathcal{C}_m(x),$$
where $a_m(U)$ are constants. By linearity of \eqref{helm}, we get
$$g(x)=\frac{1}{\sigma} \left. \frac{\partial u}{\partial \nu}\right|_{-}(x) = \sum_{m \in \mathbb{Z}} a_m(U) g_m(x).$$
Then, for $|x| \to \infty$,
\begin{equation} 
\begin{aligned} 
(u-U)(x)= & -\frac{i}{4} \sum_{n,m \in \mathbb{Z}} H^{(1)}_n(\omega |x|) e^{in\theta_x} a_m(U) \int_{\partial B}  \overline{\mathcal{C}_n(y)}  \Lambda_{1,1}^{-1} (\Lambda_{1,1} - \Lambda_{\sigma,\mu})[g_m](y)  dS_y.
\label{asympt6}
\end{aligned}
\end{equation}

Now we can define the scattering coefficients associated with $\sigma$ and $\mu$.
\begin{definition}
We define the scattering coefficients associated with the inhomogeneous permittivity $\mu(x)$ and the permeability $\sigma(x)$ for a given fixed frequency $\omega>0$ as follows:
$$W_{n,m}[B, \sigma,\mu,\omega]= \int_{\partial B} \overline{\mathcal{C}_n(y)} \Lambda_{1,1}^{-1} (\Lambda_{1,1} - \Lambda_{\sigma,\mu})[g_m](y) \; dS_y.$$
\end{definition} 
Then, for $x \to \infty$, \eqref{asympt6} becomes
\begin{equation} 
\begin{aligned} 
(u-U)(x)= & -\frac{i}{4} \sum_{n,m \in \mathbb{Z}} H^{(1)}_n(\omega |x|) e^{in\theta_x} W_{n,m} [B, \sigma, \mu, \omega] a_m(U).
\label{asympt7}
\end{aligned}
\end{equation}
From \eqref{asympt7}, we get the following theorem.

\begin{theo}
Let $u$ be the solution to \eqref{helm}. If $U$ admits the following expansion:
$$U(x)= \sum_{m \in \mathbb{Z}} a_m(U) \mathcal{C}_n(x),$$
then we have
\begin{equation} 
\begin{aligned} 
(u-U)(x)= & -\frac{i}{4} \sum_{n,m \in \mathbb{Z}} H^{(1)}_n(\omega |x|) e^{in\theta_x} W_{n,m} [B, \sigma, \mu, \omega] a_m(U).
\label{theo-asympt7}
\end{aligned}
\end{equation}
which holds uniformly as $|x| \to \infty$.

\end{theo}

\section{Integral representation of the scattering coefficients} \label{sec3}
In this section, we provide another definition of scattering coefficients which is based on integral formulations. In the following, we suppose that $\omega^2$ is not a Dirichlet eigenvalue of $-\Delta$, unless stated otherwise.

Formula \eqref{asympt4} suggests that the solution $u$  to \eqref{helm} for a given fixed frequency $\omega>0$ may be represented as
\begin{equation}
u(x)=\begin{cases} U(x) + S^{\omega}_B [\phi] (x) & x \in \mathbb{R}^2 \setminus \bar{B}, \\ \mathcal{N}_{B,\sigma,\mu}[\psi] (x) & x \in B,
\end{cases}
\label{solrep}
\end{equation}  
where the pair of densities $(\phi,\psi) \in L^2(\partial B) \times L^2(\partial B)$ satisfy the transmission conditions
\begin{equation}
\begin{cases} U(x) + S^{\omega}_B [\phi] (x) = \mathcal{N}_{B,\sigma,\mu}[\psi] (x) = \Lambda_{\sigma, \mu}[\psi](x), \\ \psi(x) = \left(\frac{1}{2} I + (K^{\omega}_B)^{*}\right)[\phi](x) + \frac{\partial U}{\partial \nu}(x),
\end{cases}
x \in \partial{B}.
\label{solrep1}
\end{equation} 
Here, the single-layer potential and the trace operator are given by 
$$
S^{\omega}_B[\phi](x) = \int_{\partial B} \Gamma_\omega (x-y) \phi(y)\, dS_y,
$$
and
$$
(K^{\omega}_B)^{*}[\phi](x) = \int_{\partial B} \frac{\partial \Gamma_\omega}{\partial \nu_x}(x-y) \phi(y)\, dS_y,
$$
and  $\mathcal{N}_{B,\sigma,\mu}$ is defined by \eqref{NtD}. We now prove that the integral equation \eqref{solrep} is uniquely solvable. 

\begin{lem} 
The operator $\mathcal{A}: L^2(\partial B) \times L^2(\partial B) \to H^1(\partial B) \times L^2(\partial B)$ defined by
$$\mathcal{A}=\begin{bmatrix} -S^{\omega}_B & \Lambda_{\sigma, \mu} \\ -\left(\frac{1}{2} I + (K^{\omega}_B)^{*}\right) & I \end{bmatrix}$$
is invertible.
\label{lem1a}
\end{lem} 

As a consequence of Lemma \ref{lem1a}, we get the following theorem. 

\begin{theo}
The solution $u$ to \eqref{helm} can be represented in the form \eqref{solrep}, where the pair $(\phi,\psi) \in L^2(\partial B) \times L^2(\partial B)$ is the solution to
\begin{equation} 
\mathcal{A} \begin{bmatrix} \phi \\ \psi \end{bmatrix} = \begin{bmatrix} U \\ \frac{\partial U}{\partial \nu} \end{bmatrix}.
\label{def-sys-matrix}
\end{equation}
\end{theo}

\begin{proof}[Proof of Lemma $3.1$] 
Let $(F,G) \in H^1(\partial B) \times L^2(\partial B)$. Proving that 
\begin{equation}\mathcal{A} \begin{bmatrix} \phi \\ \psi \end{bmatrix} = \begin{bmatrix} F \\ G \end{bmatrix}\label{system}
\end{equation}
is uniquely solvable is equivalent to prove existence and uniqueness of a solution in $H^1_{\text{loc}}(\mathbb{R}^2)$ to the problem
\begin{equation}
\begin{cases} \nabla \cdot \frac{1}{\sigma} \nabla u + \omega^2 \mu u = 0 & \mbox{ in } B, \\ \Delta u + \omega^2 u = 0 & \mbox{ in } \mathbb{R}^2 \setminus \bar{B}, \\ u|_{+} - u|_{-} = F & \mbox{ in } \partial B, \\ \left. \frac{\partial u}{\partial \nu}\right|_{+} - \left. \frac{1}{\sigma} \frac{\partial u}{\partial \nu} \right|_{-}=G & \mbox{ in } \partial B, \\ \left |\frac{\partial u}{\partial|x|}-i \omega u \right | \leq \frac{K}{|x|^{\frac{3}{2}}} & \mbox{if } |x| \to \infty.
\end{cases}
\label{transm1}
\end{equation}
To prove the injectivity of $\mathcal{A}$, let us suppose that $F=G=0$. Using the fact that
$$\int_{\partial B} \left. \frac{\partial u}{\partial \nu} \right|_{+} \bar{u} = \int_{\partial B} \left. \frac{1}{\sigma}\frac{\partial u}{\partial \nu} \right|_{-} \bar{u} = \int_{B} \frac{1}{\sigma} \left| \nabla u \right|^2 - \omega^2 \mu |u|^2,$$
we find that the solution $u$ satisfies
$$\Im \int_{\partial B} \left. \frac{\partial u}{\partial \nu} \right|_{+} \bar{u} = 0.$$
By applying Lemma $11.3$ of \cite{reconstruction teo 11.4} and the unique continuation property to \eqref{transm1}, we readily get $u=0$ in $\mathbb{R}^2$. Since
\begin{equation*}
u(x)=\begin{cases} S^{\omega}_B [\phi] (x) & x \in \mathbb{R}^2 \setminus \bar{B}, \\ \mathcal{N}_{B,\sigma,\mu}[\psi] (x) & x \in B,
\end{cases}
\end{equation*}
we get 
$$S^{\omega}_B [\phi] (x) = 0 \; \; x \in \mathbb{R}^2 \setminus \bar{B},$$
$$\mathcal{N}_{B,\sigma,\mu}[\psi] (x)=0 \; \; x \in B.$$
In particular, $\mathcal{N}_{B,\sigma,\mu}[\psi] (x)= \Lambda_{\sigma, \mu}[\psi](x)= 0$ on $\partial B$. Since $\Lambda_{\sigma, \mu}$ is invertible, $\psi=0$ on $\partial B$. On the other hand, $S^{\omega}_B[\phi](x)=0$ on $\partial B$. Suppose that $\omega^2$ is not a Dirichlet eigenvalue for $-\Delta$ on $B$. Since $(\Delta + \omega^2) S^{\omega}_B[\phi](x)=0$ in $B$, we have $S^{\omega}_B[\phi](x) =0$ in $B$, and hence in $\mathbb{R}^2$. It then follows from \cite{reconstruction teo 11.4} that
$$\phi =  \left. \frac{\partial S^{\omega}_B[\phi]}{\partial \nu}\right|_{+} - \left. \frac{\partial S^{\omega}_B[\phi]}{\partial \nu}\right|_{-} = 0 \; \; \mbox{on } \partial D.$$
This finishes the proof of the injectivity of $\mathcal{A}$. Since $u$ is solution to $\Delta u + \omega^2 u=0$ in $\mathbb{R}^2 \setminus \bar{B}$ and $\left |\frac{\partial u}{\partial|x|}-i \omega u \right | \leq \frac{K}{|x|^{\frac{3}{2}}}$ as $|x| \to \infty$, then there exists $\phi \in L^2(\partial B)$ such that 
\begin{equation}u(x)=S^{\omega}_B[\phi](x), \; \; \; x \in \mathbb{R}^2 \setminus \bar{B}.
\label{solout}
\end{equation}
If we set 
\begin{equation} 
\psi = \frac{1}{\sigma} \left. \frac{\partial u}{\partial \nu}\right|_{-},
\label{psi}
\end{equation}
then
$$\Lambda_{\sigma, \mu}[\psi]=u|_{-}.$$
By \eqref{solout}, 
$$ \left. \frac{\partial u}{\partial \nu}\right|_{+} = \left(\frac{1}{2} I + (K^{\omega}_B)^{*}\right) [\phi], $$
and hence
$$\psi = \left. \frac{1}{\sigma} \frac{\partial u}{\partial \nu} \right|_{-} = \left. \frac{\partial u}{\partial \nu}\right|_{+}+G = \left(\frac{1}{2} I + (K^{\omega}_B)^{*}\right) [\phi] + G.$$ 
Thus, for $x \in \partial B$,
$$G(x)= - \left(\frac{1}{2} I + (K^{\omega}_B)^{*}\right) [\phi](x) + \psi(x),$$
and from the transmission condition
$$F(x)=\Lambda_{\sigma, \mu}[\psi](x) - S^{\omega}_B[\phi](x),$$
which shows that $\phi$ and $\psi = \left. \frac{1}{\sigma} \frac{\partial u}{\partial \nu}\right|_{-}$ solve \eqref{system}.

\end{proof}

We can now define the scattering coefficients associated with $\mu$ and $\sigma$ using the operator $\mathcal{A}$. 

\begin{definition}
For $m \in \mathbb{Z}$, let $(\phi_m,\psi_m) \in L^2(\partial B) \times L^2(\partial B)$ be the solution to
\begin{equation}\mathcal{A} \begin{bmatrix} \phi_m \\ \psi_m \end{bmatrix} = \begin{bmatrix} \mathcal{C}_m \\ \frac{\partial \mathcal{C}_m}{\partial \nu} \end{bmatrix} \mbox{ on } \partial B,
\label{scatsys}
\end{equation}
where $\mathcal{C}_m$ is the cylindrical wave. For $n \in \mathbb{Z}$, we define the scattering coefficients associated with the permittivity distribution $\mu(x)$ and permeability $\sigma(x)$  for a given fixed frequency $\omega>0$ as follows:
\begin{equation}
W_{n,m}=W_{n,m}[B, \sigma,\mu,\omega]=\int_{\partial B} \overline{\mathcal{C}_n(y)} \phi_m(y) \; dS_y.
\label{scatcoef}
\end{equation}
\end{definition}

\section{Properties of the scattering coefficients} \label{sec4}

In this section, we prove important properties for the scattering coefficients. 

\subsection{Decay of the Scattering Coefficients} 

Like the homogeneous case, the coefficient $W_{n,m}$ decays exponentially as the orders $m,n$ increase. We can prove the following Proposition.

\begin{prop} For a given fixed frequency $\omega>0$, there is a constant $K$ (depending on $\sigma$, $\mu$ and $\omega$) such that 
\begin{equation}
\left|W_{n,m} [B,\sigma,\mu,\omega]\right| \leq \frac{K^{|n|+|m|}}{|n|^{|n|}|m|^{|m|}} \; \; \; \forall n,m \in \mathbb{Z}.
\label{decay}
\end{equation}
\end{prop} 

\begin{proof} 
Recall that
$$\mathcal{C}_m(x)=J_m(\omega|x|)e^{im\theta_x}.$$
Since
$$J_m(t) \sim \frac{1}{\sqrt{2 \pi |m|}} \left( \frac{et}{2|m|} \right)^{|m|}$$ 
as $m \to \infty$, we have
$$\|\mathcal{C}_m\|_{L^2(\partial  B)} + \|\nabla \mathcal{C}_m\|_{L^2(\partial B)} \leq \frac{K^{|m|}}{|m|^{|m|}}.$$ 
Then, with the same arguments as those of \cite{boundary layer}, there exists another constant $K$ such that
\begin{equation*} 
\begin{aligned} 
\|\phi_m\|_{L^2(\partial  B)} + \|\psi_m \|_{L^2(\partial B)}  & \leq K \left(\|\mathcal{C}_m\|_{L^2(\partial  B)} + \left\|\frac{ \partial \mathcal{C}_m}{\partial \nu} \right\|_{L^2(\partial B)}\right) \leq \\ & \leq K \left(\|\mathcal{C}_m\|_{L^2(\partial  B)} + \|\nabla \mathcal{C}_m\|_{L^2(\partial B)}\right) .
\end{aligned}
\end{equation*} 
Hence, for another constant $K$,
$$\|\phi_m\|_{L^2(\partial  B)} \leq \frac{K^{|m|}}{|m|^{|m|}}.$$
Since
$$\left\| \overline{\mathcal{C}_n(y)} \right\|_{L^2(\partial  B)} \leq \frac{K^{|n|}}{|n|^{|n|}},$$
by the definition of the scattering coefficients $W_{n,m}$,  we have \eqref{decay}.

\end{proof}

\subsection{Transformation formulas}
We introduce the notation for translation, scaling, and rotation of a shape $B$ as
$$B^z:= B+z, \; B^s:= sB, \; B^{\theta}= e^{i\theta}B,$$
and those of the material parameter $\sigma$ (resp. $\mu$) as 
$$\sigma^z:= \sigma(\cdot - z), \; \sigma^s := \sigma(s^{-1} \cdot), \; \sigma^{\theta}:= \sigma(e^{-i\theta} \cdot).$$ 
We denote by $\phi_{U,B}=\phi_{U,B, \omega}$ the solution to \eqref{def-sys-matrix} given the domain $B$, the source term $U$, and the frequency $\omega$. We can prove that there exist explicit relations between the inhomogeneous scattering coefficients of $B$ and $B^z$, $B^s$, $B^{\theta}$. We prove the following Propositions.

\begin{prop} [Translation formula]

For any $z \in \mathbb{R}^2$, the following relation holds
\begin{equation}
W_{n,m} [B^z,\sigma^z,\mu^z,\omega] = \sum_{a,b}\overline{\mathcal{C}_a(z)} \mathcal{C}_b(z) W_{n-a,m-b} [B,\sigma,\mu,\omega].
\label{translation}
\end{equation}
\end{prop}

\begin{proof}
Let $\tilde{\phi}_m = \phi_{\mathcal{C}_m,B^z}$ and $\tilde{y}=y+z$ with $y \in B$. By the definition of the scattering coefficients:
\begin{equation}
\begin{aligned}
W_{n,m} [B^z,\sigma^z,\mu^z,\omega] & = \int_{\partial B^z} \overline{\mathcal{C}_n(\tilde{y})} \tilde{\phi}_m(\tilde{y}) \; dS_{\tilde{y}} = \\ & = \int_{\partial B} \overline{\mathcal{C}_n(y+z)} \tilde{\phi}_m(y+z) \; dS_y.
\label{scatcoeftrasl}
\end{aligned}
\end{equation}
From the identity \cite{math and stat method}, 
$$\mathcal{C}_n(y-z) =\sum_{l \in \mathbb{Z}}\mathcal{C}_{l+n} (y) \overline{\mathcal{C}_l(z)},$$
we have
\begin{equation*}
\begin{aligned}
\overline{\mathcal{C}_n(y+z)} & =  \sum_{a \in \mathbb{Z}} \overline{\mathcal{C}_{a+n} (y)} \mathcal{C}_a(-z) ,
\end{aligned}
\end{equation*}
and  
\begin{equation*}
\begin{aligned}
\mathcal{C}_m(\tilde{x}) = \mathcal{C}_m(x+z) & = \sum_{b \in \mathbb{Z}} \mathcal{C}_{b+m} (x) \overline{\mathcal{C}_b(-z)},
\end{aligned}
\end{equation*}
where $\tilde{x}=x+z$ with $x \in B$. To find $\tilde{\phi}_m$, let us consider 
$$\tilde{\mathcal{A}} \begin{bmatrix} \tilde{\phi}_{m} \\ \tilde{\psi}_{m} \end{bmatrix} (\tilde{x})=\begin{bmatrix} -S^{\omega}_{B^z}[\tilde{\phi}_{m}](\tilde{x}) & \Lambda_{B^z, \sigma^z, \mu^z}[\tilde{\psi}_{m}](\tilde{x}) \\ -\left(\frac{1}{2} I + (K^{\omega}_{B^z})^{*}\right)[\tilde{\phi}_{m}](\tilde{x})  & I[\tilde{\psi}_{m}](\tilde{x}) \end{bmatrix}.$$
Recall that $(\tilde{\phi}_m, \tilde{\psi}_m) \in L^2(\partial B^z) \times L^2 (\partial B^z)$ is the solution to 
\begin{equation} 
\tilde{\mathcal{A}} \begin{bmatrix} \tilde{\phi}_{m}\\ \tilde{\psi}_{m} \end{bmatrix} (\tilde{x}) = \begin{bmatrix} {\mathcal{C}}_{m}(\tilde{x}) \\ \frac{\partial {\mathcal{C}}_{m}}{\partial \nu_{\tilde{x}}}(\tilde{x}) \end{bmatrix}.
\label{systransl1}
\end{equation}
Let $\tilde{w}= w + z$ for $w \in B$. Let us prove that 
\begin{equation}
\begin{aligned}\begin{bmatrix} \tilde{\phi}_m(\tilde{w}) \\ \tilde{\psi}_m(\tilde{w}) \end{bmatrix}  = \sum_{b\in \mathbb{Z}}  \overline{\mathcal{C}_b(-z)}  \begin{bmatrix} \phi_{m+b}(w) \\ \psi_{m+b}(w) \end{bmatrix} 
\label{traslsol}
\end{aligned}
\end{equation}
is the solution. Since 
$$\begin{bmatrix} {\mathcal{C}}_{m}(\tilde{x}) \\ \frac{\partial {\mathcal{C}}_{m}}{\partial \nu_{\tilde{x}}}(\tilde{x}) \end{bmatrix}= \begin{bmatrix} {\mathcal{C}}_{m}(x+z) \\ \frac{\partial {\mathcal{C}}_{m}}{\partial \nu_{x+z}}(x+z) \end{bmatrix} = \sum_{b\in \mathbb{Z}}  \overline{\mathcal{C}_b(-z)}  \begin{bmatrix}  \mathcal{C}_{m+b}(x) \\ \frac{\partial \mathcal{C}_{m+b}}{\partial \nu_{x}}(x) \end{bmatrix}, $$
if we prove that
\begin{equation} 
\tilde{\mathcal{A}} \begin{bmatrix} \tilde{\phi}_{m} \\ \tilde{\psi}_{m} \end{bmatrix} (\tilde{x})=\begin{bmatrix} -S^{\omega}_{B}[\tilde{\phi}_{m}(\cdot + z)](x) & \Lambda_{B, \sigma, \mu}[\tilde{\psi}_{m}(\cdot + z)](x) \\ -\left(\frac{1}{2} I + (K^{\omega}_{B})^{*}\right)[\tilde{\phi}_{m}(\cdot +z)](x)  & I[\tilde{\psi}_{m}(\cdot+z)](x) \end{bmatrix},
\label{transl-A}
\end{equation}
then by the linearity of operator $\mathcal{A}$ and the existence and uniqueness of a solution to system \eqref{systransl1}, we have \eqref{traslsol}. Let us prove \eqref{transl-A}. We write
\begin{equation*}
\begin{aligned}
S^{\omega}_{B^z}[\tilde{\phi}_m](\tilde{x}) & = \int_{B^z} \Gamma_{\omega} (\tilde{x}-\tilde{y}) \tilde{\phi}_m(\tilde{y}) \; dS_{\tilde{y}} = \\ & = \int_{B} \Gamma_{\omega} (x+z-(y+z)) \tilde{\phi}_m(y+z) \; dS_y = \\ & = S^{\omega}_{B}[ \tilde{\phi}_m(\cdot+z)](x), 
\end{aligned}
\end{equation*}
$$ -\left(\frac{1}{2} I + (K^{\omega}_{B^z})^{*}\right)[\tilde{\phi}_m](\tilde{x}) = -\left(\frac{1}{2} I + (K^{\omega}_B)^{*}\right)[\tilde{\phi}_m(\cdot+z)](x),$$
\begin{equation*}
\begin{aligned}\mathcal{N}_{B^z, \sigma^z, \mu^z}[\tilde{\psi}_m](\tilde{x}) & =\int_{\partial B^z} \tilde{N}_{\sigma^z, \mu^z}(\tilde{x},\tilde{y}) \tilde{\psi}_m (\tilde{y}) \; dS_{\tilde{y}} = \\ & = \int_{\partial B} N_{\sigma, \mu}(x,y) \tilde{\psi}_m(y+z) \; dS_{y} = \\ & = \mathcal{N}_{B, \sigma, \mu}[\tilde{\psi}_m (\cdot + z)](x),
\end{aligned}
\end{equation*}
where $ \tilde{N}_{\sigma^z, \mu^z}(\tilde{x},\tilde{y}) = N_{\sigma, \mu}(x,y)$ follows from the existence and uniqueness of the Neumann function result. In fact, for $\tilde{w} \in B^z$ and by a change of variables, the Neumann problem 
\begin{equation}
\begin{cases} \nabla_{\tilde{y}} \cdot \frac{1}{\sigma^z(\tilde{y})} \nabla_{\tilde{y}} \tilde{N}_{\sigma^z, \mu^z}(\tilde{w},\tilde{y}) + \omega^2 \mu^z(\tilde{y}) \tilde{N}_{\sigma^z, \mu^z}(\tilde{w},\tilde{y}) = - \delta_{\tilde{w}}(\tilde{y}) & \tilde{y} \in B^z, \\ \frac{1}{\sigma^z(\tilde{y})}\frac{\partial \tilde{N}_{\sigma^z, \mu^z}}{\partial_{\nu_{\tilde{y}}}} (\tilde{w}, \tilde{y})= 0 & \tilde{y} \in \partial B^z,
\end{cases}
\label{neumann 1}
\end{equation}
 can be rewritten as
\begin{equation}
\begin{cases} \nabla_{y} \cdot \frac{1}{\sigma(y)} \nabla_{y} \tilde{N}_{\sigma^z, \mu^z} (w+z,y+z) + \omega^2 \mu(y) \tilde{N}_{\sigma^z, \mu^z}(w+z,y+z) = - \delta_{w}(y) & y \in B, \\ \frac{1}{\sigma(y)}\frac{\partial \tilde{N}_{\sigma^z, \mu^z}}{\partial_{\nu_{y}}} (w+z,y+z)= 0 & y \in \partial B.
\end{cases}
\label{neumann 2}
\end{equation}
From the uniqueness of the Neumann function, $\tilde{N}_{\sigma^z, \mu^z} (w+z,y+z)= N_{\sigma, \mu}(x,y)$ is the solution to \eqref{neumann 2}. 

Then, \eqref{scatcoeftrasl} yields
\begin{equation*}
\begin{aligned}
W_{n,m} [B^z,\sigma^z,\mu^z,\omega] & = \sum_{a,b} \mathcal{C}_a(-z) \overline{\mathcal{C}_b(-z)} W_{n+a,m+b} [B,\sigma,\mu,\omega]. 
\end{aligned}
\end{equation*}  
Since $\overline{\mathcal{C}_a(-z)}=\mathcal{C}_{-a}(z)$, \eqref{translation} holds.

\end{proof}

\begin{prop} [Scaling formula] For any $s>0$, the following relation holds 
\begin{equation}
W_{n,m} [B^s,\sigma^s,\mu^s,\omega] = W_{n,m} [B,\sigma,\mu,s \omega].
\label{scaling}
\end{equation}
\end{prop}
\begin{proof} 
Let $\tilde{\phi}_{m, \omega} = \phi_{\mathcal{C}_m,B^s, \omega}$ and $\tilde{y}=sy$ with $y \in B$. Since $(|\tilde{y}|, \theta_{\tilde{y}})=(s |y|, \theta_{y})$, by the definition of the scattering coefficients: 
\begin{equation}
\begin{aligned}
W_{n,m} [B^s,\sigma^s,\mu^s,\omega] & = \int_{\partial B^s} \overline{\mathcal{C}_{n, \omega} (\tilde{y})} \tilde{\phi}_{m,\omega}(\tilde{y}) \; dS_{\tilde{y}} = \\ & = s \int_{\partial B} \overline{\mathcal{C}_{n,s\omega}(y)} \tilde{\phi}_{m,\omega}(sy) \; dS_y.
\label{scatcoefscal}
\end{aligned}
\end{equation}
We have
\begin{equation*}
\begin{aligned}
\mathcal{C}_{m,\omega}(\tilde{x}) = \mathcal{C}_{m, \omega}(sx) & = \mathcal{C}_{m, s\omega}(x),
\end{aligned}
\end{equation*}
where $\tilde{x}=sx$ with $x \in B$. To find $\tilde{\phi}_{m,\omega}$, let us consider 
$$\tilde{\mathcal{A}} \begin{bmatrix} \tilde{\phi}_{m, \omega} \\ \tilde{\psi}_{m, \omega} \end{bmatrix} (\tilde{x})=\begin{bmatrix} -S^{\omega}_{B^s}[\tilde{\phi}_{m, \omega}](\tilde{x}) & \Lambda_{B^s, \sigma^s, \mu^s, \omega}[\tilde{\psi}_{m, \omega}](\tilde{x}) \\ -\left(\frac{1}{2} I + (K^{\omega}_{B^s})^{*}\right)[\tilde{\phi}_{m, \omega}](\tilde{x})  & I[\tilde{\psi}_{m, \omega}](\tilde{x}) \end{bmatrix}.$$
Recall that $(\tilde{\phi}_{m, \omega}, \tilde{\psi}_{m, \omega}) \in L^2(\partial B^s) \times L^2 (\partial B^s)$ is the solution to 
\begin{equation}
\tilde{\mathcal{A}} \begin{bmatrix} \tilde{\phi}_{m, \omega}\\ \tilde{\psi}_{m, \omega} \end{bmatrix} (\tilde{x}) = \begin{bmatrix} {\mathcal{C}}_{m, \omega}(\tilde{x}) \\ \frac{\partial {\mathcal{C}}_{m, \omega}}{\partial \nu_{\tilde{x}}}(\tilde{x}) \end{bmatrix}.
\label{sysscal1}
\end{equation}
Let $\tilde{w}= sw $ for $w \in B$. Let us prove that 
\begin{equation}
\begin{aligned}
\begin{bmatrix} \tilde{\phi}_{m,\omega}(\tilde{w}) \\ \tilde{\psi}_{m,\omega}(\tilde{w}) \end{bmatrix} & = s^{-1}  \begin{bmatrix} \phi_{m,s\omega} (w) \\ \psi_{m,s\omega} (w) \end{bmatrix}
\label{scalsol}
\end{aligned}
\end{equation} 
is the solution. Since 
$$\begin{bmatrix} {\mathcal{C}}_{m, \omega}(\tilde{x}) \\ \frac{\partial {\mathcal{C}}_{m, \omega}}{\partial \nu_{\tilde{x}}}(\tilde{x}) \end{bmatrix}= \begin{bmatrix} {\mathcal{C}}_{m, \omega}(sx) \\ \frac{\partial {\mathcal{C}}_{m, \omega}}{\partial \nu_{sx}}(sx) \end{bmatrix}= \begin{bmatrix} \mathcal{C}_{m, s \omega}(x) \\ s^{-1} \frac{\partial\mathcal{C}_{m, s \omega}}{\partial \nu_{x} }(x) \end{bmatrix}, $$
if we prove that
\begin{equation} 
\tilde{\mathcal{A}} \begin{bmatrix} \tilde{\phi}_{m, \omega} \\ \tilde{\psi}_{m, \omega} \end{bmatrix} (\tilde{x})=\begin{bmatrix} -sS^{s\omega}_B [ \tilde{\phi}_{m, \omega}(s \; \cdot)](x) & s\Lambda_{B, \sigma, \mu, s\omega} [ \tilde{\psi}_{m, \omega}(s \; \cdot)](x) \\ -\left(\frac{1}{2} I + (K^{s\omega}_B)^{*} \right) [ \tilde{\phi}_{m, \omega}(s \; \cdot)](x) & I [ \tilde{\psi}_{m, \omega}(s \; \cdot)](x) \end{bmatrix},
\label{sysscan}
\end{equation}
then by the linearity of operator $\mathcal{A}$ and the existence and uniqueness of a solution to system \eqref{sysscal1}, we have \eqref{scalsol}. Let us prove \eqref{sysscan}. We write
\begin{equation*}
\begin{aligned}
S^{\omega}_{B^s}[\tilde{\phi}_{m,\omega}](\tilde{x}) & = \int_{B^s} \Gamma_{\omega} (\tilde{x}-\tilde{y}) \tilde{\phi}_{m, \omega}(\tilde{y}) \; dS_{\tilde{y}} = \\ & = \int_{B} \Gamma_{\omega} (sx -sy) \tilde{\phi}_{m, \omega}(sy) s \; dS_y = \\ & = s S^{s \omega}_{B}[ \tilde{\phi}_{m, \omega}(s \; \cdot)](x), 
\end{aligned}
\end{equation*}
$$ -\left(\frac{1}{2} I + (K^{\omega}_{B^s})^{*}\right)[\tilde{\phi}_{m, \omega}](\tilde{x}) = -\left(\frac{1}{2} I + (K^{s\omega}_B)^{*}\right)[\tilde{\phi}_{m, \omega}(s \; \cdot )](x),$$
\begin{equation*}
\begin{aligned}\mathcal{N}_{B^s, \sigma^s, \mu^s, \omega}[\tilde{\psi}_{m,\omega}](\tilde{x}) & =\int_{\partial B^s} \tilde{N}_{\sigma^s, \mu^s, \omega}(\tilde{x},\tilde{y}) \tilde{\psi}_{m, \omega} (\tilde{y}) \; dS_{\tilde{y}} = \\ & = s \int_{\partial B} N_{\sigma, \mu, s \omega}(x,y) \tilde{\psi}_{m, \omega}(s \; \cdot) \; dS_{y} = \\ & = s \mathcal{N}_{B, \sigma, \mu, s \omega}[\tilde{\psi}_{m, \omega}(s \; \cdot) ](x) ,
\end{aligned}
\end{equation*}
where $ \tilde{N}_{\sigma^s, \mu^s, \omega}(\tilde{x},\tilde{y}) = N_{\sigma, \mu, s\omega} (x,y)$ follows from existence and uniqueness of the Neumann function. In fact, for $\tilde{w} \in B^s$ and by a change of variables,  the Neumann problem 
\begin{equation}
\begin{cases} \nabla_{\tilde{y}} \cdot \frac{1}{\sigma^s(\tilde{y})} \nabla_{\tilde{y}} \tilde{N}_{\sigma^s, \mu^s, \omega}(\tilde{w},\tilde{y}) + \omega^2 \mu^s(\tilde{y}) \tilde{N}_{\sigma^s, \mu^s, \omega}(\tilde{w},\tilde{y}) = - \delta_{\tilde{w}}(\tilde{y}) & \tilde{y} \in B^s, \\ \frac{1}{\sigma^s(\tilde{y})}\frac{\partial \tilde{N}_{\sigma^s, \mu^s, \omega}}{\partial_{\nu_{\tilde{y}}}} (\tilde{w}, \tilde{y})= 0 & \tilde{y} \in \partial B^s,
\end{cases}
\label{neumann scal 1}
\end{equation} can be rewritten as 
\begin{equation}
\begin{cases} \nabla_{y} \cdot \frac{1}{\sigma(y)} \nabla_{y} \tilde{N}_{\sigma^s, \mu^s, \omega}(sw,sy) + s^2 \omega^2 \mu(y) \tilde{N}_{\sigma^s, \mu^s, \omega}(sw,sy) = - \delta_{w}(y) & y \in B, \\ \frac{1}{\sigma(y)}\frac{\partial \tilde{N}_{\sigma^s, \mu^s, \omega}}{\partial_{\nu_{y}}} (sw,sy)= 0 & y \in \partial B.
\end{cases}
\label{neumann scal 2}
\end{equation} 
from the uniqueness of the Neumann function, $\tilde{N}_{\sigma^s, \mu^s, \omega}(sx,sy) = N_{\sigma, \mu, s\omega} (x,y)$ is the solution to \eqref{neumann scal 2}.

Then, \eqref{scatcoefscal} yields
\begin{equation*}
\begin{aligned}
W_{n,m} [B^s,\sigma^s,\mu^s,\omega] & = W_{n,m} [B,\sigma,\mu, s \omega]. 
\end{aligned}
\end{equation*}

\end{proof}

\begin{prop}[Rotation formula] For any $\theta$, the following relation holds 
\begin{equation}
W_{n,m} [B^{\theta},\sigma^{\theta},\mu^{\theta},\omega] = e^{i(m-n) \theta} W_{n,m} [B,\sigma,\mu, \omega].
\label{rotation}
\end{equation}
\end{prop} 
\begin{proof}
Let $\tilde{\phi}_m=\phi_{U_m, B^{\theta}}$ and $\tilde{y}= e^{i \theta }y$ with $y \in B$. Since $(|\tilde{y}|, \theta_{\tilde{y}})=(|y|, \theta_{y} + \theta)$, by the definition of the scattering coefficients:
\begin{equation}
\begin{aligned}
W_{n,m} [B^{\theta},\sigma,\mu,\omega] & = \int_{\partial B^{\theta}} \overline{\mathcal{C}_n(\tilde{y})} \tilde{\phi}_m(\tilde{y}) \; dS_{\tilde{y}} = \\ & = \int_{\partial B} \overline{\mathcal{C}_n(y)} e^{-in{\theta}} \tilde{\phi}_m(e^{i\theta} y) \; dS_y.
\label{scatcoefrot}
\end{aligned}
\end{equation}
We have
\begin{equation*}
\begin{aligned}
\mathcal{C}_m(\tilde{x}) = \mathcal{C}_m(e^{i\theta}x) & = \mathcal{C}_m(x) e^{im\theta},
\end{aligned}
\end{equation*}
where $\tilde{x}=e^{i \theta}x$ with $x \in B$. To find $\tilde{\phi}_m$, let us consider 
$$\tilde{\mathcal{A}} \begin{bmatrix} \tilde{\phi}_m \\ \tilde{\psi}_m \end{bmatrix} (\tilde{x})=\begin{bmatrix} -S^{\omega}_{B^{\theta}}[\tilde{\phi}_m](\tilde{x}) & \Lambda_{B^{\theta}, \sigma^{\theta}, \mu^{\theta}}[\tilde{\psi}_m](\tilde{x}) \\ -(\frac{1}{2} I + (K^{\omega}_{B^{\theta}})^{*})[\tilde{\phi}_m](\tilde{x}) & I [\tilde{\psi}_m](\tilde{x}) \end{bmatrix}.$$
Recall that $(\tilde{\phi}_{m}, \tilde{\psi}_{m}) \in L^2(\partial B^{\theta}) \times L^2 (\partial B^{\theta})$ is the solution to 
\begin{equation}
\tilde{\mathcal{A}} \begin{bmatrix} \tilde{\phi}_{m}\\ \tilde{\psi}_{m} \end{bmatrix} (\tilde{x}) = \begin{bmatrix} {\mathcal{C}}_{m}(\tilde{x}) \\ \frac{\partial {\mathcal{C}}_{m}}{\partial \nu_{\tilde{x}}}(\tilde{x}) \end{bmatrix}.
\label{sysrot1}
\end{equation}
Let $\tilde{w}= e^{i \theta}w$ for $w \in B$. Let us prove that 
\begin{equation}
\begin{aligned} \begin{bmatrix} \tilde{\phi}_m(\tilde{w}) \\ \tilde{\psi}_m(\tilde{w}) \end{bmatrix}  = e^{i m \theta} \begin{bmatrix} \phi_{m} (w) \\  \psi_{m} (w) \end{bmatrix}
\label{rotsol}
\end{aligned}
\end{equation}
is the solution. Since 
$$\begin{bmatrix} {\mathcal{C}}_{m}(\tilde{y}) \\ \frac{\partial {\mathcal{C}}_{m}}{\partial \nu_{\tilde{y}}}(\tilde{y}) \end{bmatrix} = \begin{bmatrix} {\mathcal{C}}_{m}(e^{i \theta} y) \\ \frac{\partial {\mathcal{C}}_{m}}{\partial \nu_{e^{i \theta} y}}(e^{i \theta} y) \end{bmatrix}=  e^{im\theta} \begin{bmatrix}  \mathcal{C}_m (y) \\ \frac{\partial \mathcal{C}_m}{\partial \nu_{y}}(y) \end{bmatrix}, $$
if we prove that
\begin{equation} 
\begin{aligned} 
\tilde{\mathcal{A}} \begin{bmatrix} \tilde{\phi}_{m} \\ \tilde{\psi}_{m} \end{bmatrix} (\tilde{x})= 
\begin{bmatrix} -S^{\omega}_B[ \tilde{\phi}_m(e^{i\theta} \; \cdot)] (x) & \Lambda_{B, \sigma, \mu}[ \tilde{\psi}_m(e^{i\theta} \; \cdot)] (x) \\ -(\frac{1}{2} I + (K^{\omega}_B)^{*}) [ \tilde{\phi}_m(e^{i\theta} \; \cdot)] (x)& I[ \tilde{\psi}_m(e^{i\theta} \; \cdot)]  (x) \end{bmatrix} ,  \label{sysrot}
\end{aligned}
\end{equation}
then by the linearity of operator $\mathcal{A}$ and the existence and uniqueness of a solution system \eqref{sysrot1}, we have \eqref{rotsol}. Let us prove identity \eqref{sysrot}. We write
\begin{equation*}
\begin{aligned}
S^{\omega}_{B^{\theta}}[\tilde{\phi}_m](\tilde{x}) & = \int_{B^{\theta}} \Gamma_{\omega} (\tilde{x}-\tilde{y}) \tilde{\phi}_m(\tilde{y}) \; dS_{\tilde{y}} = \\ & = \int_{B} \Gamma_{\omega} (x - y) \tilde{\phi}_m(e^{i\theta}y) s \; dS_y = \\ & =  S^{ \omega}_{B}[ \tilde{\phi}_m(e^{i\theta} \; \cdot)](x), 
\end{aligned}
\end{equation*}
$$ -\left(\frac{1}{2} I + (K^{\omega}_{B^{\theta}})^{*} \right)[\tilde{\phi}_m](\tilde{x}) = -\left(\frac{1}{2} I + (K^{\omega}_B)^{*} \right)[\tilde{\phi}_{m} (e^{i\theta} \; \cdot )](x),$$
\begin{equation*}
\begin{aligned}\mathcal{N}_{B^{\theta}, \tilde{\sigma}, \tilde{\mu}}[\tilde{\psi}_m](\tilde{x}) & =\int_{\partial B^{\theta}} \tilde{N}_{\tilde{\sigma}, \tilde{\mu}}(\tilde{x},\tilde{y}) \tilde{\psi}_m (\tilde{y}) \; dS_{\tilde{y}} = \\ & = \int_{\partial B} N_{\sigma, \mu}(x,y) \tilde{\psi}_m(e^{i\theta} \; \cdot) \; dS_{y} = \\ & = \mathcal{N}_{B, \sigma, \mu}[\tilde{\psi}_m(e^{i\theta} \; \cdot) ](x) ,
\end{aligned}
\end{equation*}
where $ \tilde{N}_{\sigma^{\theta}, \mu^{\theta}}(\tilde{x},\tilde{y}) = N_{\sigma, \mu} (x, y)$ follows from the existence and uniqueness of the Neumann function. In fact, for $\tilde{w} \in B^{\theta}$ and by the change of variables, the Neumann problem 
\begin{equation}
\begin{cases} \nabla_{\tilde{y}} \cdot \frac{1}{\sigma^{\theta}(\tilde{y})} \nabla_{\tilde{y}} \tilde{N}_{\sigma^{\theta}, \mu^{\theta}}(\tilde{w},\tilde{y}) + \omega^2 \mu^{\theta}(\tilde{y}) \tilde{N}_{\sigma^{\theta}, \mu^{\theta}}(\tilde{w},\tilde{y}) = - \delta_{\tilde{w}}(\tilde{y}) & \tilde{y} \in B^{\theta}, \\ \frac{1}{\sigma^{\theta}(\tilde{y})}\frac{\partial \tilde{N}_{\sigma^{\theta}, \mu^{\theta}}}{\partial_{\nu_{\tilde{y}}}} (\tilde{w}, \tilde{y})= 0 & \tilde{y} \in \partial B^{\theta},
\end{cases}
\label{neumann rot 1}
\end{equation}
can be rewritten as 
\begin{equation}
\begin{cases} \nabla_{y} \cdot \frac{1}{\sigma(y)} \nabla_{y} \tilde{N}_{\sigma^{\theta}, \mu^{\theta}}(e^{i\theta}w,e^{i\theta}y) + \omega^2 \mu(y) \tilde{N}_{\sigma^{\theta}, \mu^{\theta}}(e^{i\theta}w,e^{i\theta}y) = - \delta_{w}(y) & y \in B, \\ \frac{1}{\sigma(y)} \frac{\partial \tilde{N}_{\sigma^{\theta}, \mu^{\theta}}}{\partial_{\nu_{y}}} (e^{i\theta}w,e^{i\theta}y)= 0 & y \in \partial B.
\end{cases}
\label{neumann rot 2}
\end{equation}
By the uniqueness of the Neumann function, $\tilde{N}_{\sigma^{\theta}, \mu^{\theta}}(e^{i \theta} x,e^{i \theta} y) = N_{\sigma, \mu} (x,y)$ is the solution to \eqref{neumann scal 2}.

Then, \eqref{scatcoefrot} yields
\begin{equation*}
\begin{aligned}
W_{n,m} [B^{\theta},\sigma^{\theta},\mu^{\theta},\omega] = e^{i(m-n) \theta} W_{n,m} [B,\sigma,\mu, \omega]. 
\end{aligned}
\end{equation*}  
\end{proof}

\section{Distribution descriptors and identification in a dictionary} \label{sec5}

In this section, we construct the distribution descriptors which are invariant to rigid transformations. In the following, we proceed as in \cite{echolocation}. We denote by $B$ a reference shape of size $1$ centered at the origin, so that the unknown target $D$  is generated from $B$ by a rotation with angle $\theta$, a scaling $s>0$, and a translation $z \in \mathbb{R}^2$ as 
$$D=z+sR_{\theta}B.$$

\subsection{Far-field pattern} To construct the distribution descriptors which are invariant to rigid transformations, we derive the far-field pattern of the scattering field in terms of the inhomogeneous scattering coefficients. The result is similar to the homogeneous case \cite{cloaking}.

If $U$ is given by a plane wave $e^{i \omega \xi \cdot x}$ such that $\omega \xi \cdot \omega \xi = \omega^2$, then by the Jacobi-Anger decomposition, we have
$$U(x)=\sum_{m \in \mathbb{Z}} e^{i m (\frac{\pi}{2}-\theta_{\xi})}J_m(\omega |x|)e^{im\theta_x}=\sum_{m \in Z} e^{i m (\frac{\pi}{2}-\theta_{\xi})} \mathcal{C}_m(x).$$
By the linearity of operator $\mathcal{A}$ and existence and uniqueness of a solution to system \eqref{sysscal1}, we obtain
$$\phi = \sum_{m \in \mathbb{Z}} e^{i m (\frac{\pi}{2}-\theta_{\xi})} \phi_m.$$
Using \eqref{asympt4},
\begin{equation} 
\begin{aligned}
u-e^{i \omega \xi \cdot x} = (u-U)(x)=  -\frac{i}{4} \sum_{m \in \mathbb{Z}} e^{i m (\frac{\pi}{2}-\theta_{\xi})} \int_{\partial B} H_0^{(1)} (\omega |x-y|)  \phi_m(y) \; dS_y.
\label{far1}
\end{aligned}
\end{equation}
Recall that \cite{cloaking}
$$H_0^{(1)} (\omega |x-y|)= \sqrt{\frac{2}{\pi \omega |x-y|}} e^{i(\omega|x-y| - \frac{\pi}{4})} + O(|x|^{-1}) \mbox{ as } x \to \infty.$$
Since $|x-y|=|x|-|y|\cos(\theta_x-\theta_y)+O(|x|^{-1})$, \eqref{far1} becomes
\begin{equation} 
\begin{aligned}
u-e^{i \omega \xi \cdot x} = -i e^{-i \frac{\pi}{4}} \frac{e^{i \omega |x|}}{\sqrt{8 \pi \omega |x|}} \sum_{m \in \mathbb{Z}} e^{i m (\frac{\pi}{2}-\theta_{\xi})} \int_{\partial B} e^{-i\omega|y|\cos(\theta_x-\theta_y)} \phi_m(y) \; dS_y + O(|x|^{-1})   .
\label{far2}
\end{aligned}
\end{equation}
By Jacobi-Anger identity,
$$e^{-i\omega|y|\cos(\theta_x-\theta_y)} = e^{i\omega|y|\cos(\theta_x-\theta_y + \pi)} = \sum_{n \in \mathbb{Z}} i^{n} J_n(\omega |y|) e^ {i n (\theta_x - \theta_y + \pi)},$$
\eqref{far2} becomes
\begin{equation} 
\begin{aligned}
u-e^{i \omega \xi \cdot x} = -i e^{-i \frac{\pi}{4}} \frac{e^{i \omega |x|}}{\sqrt{8 \pi \omega |x|}} \sum_{m, n \in \mathbb{Z}} e^{i m (\frac{\pi}{2}-\theta_{\xi})} e^ {i n (\theta_x - \frac{\pi}{2})} \int_{\partial B} i^{n} J_n(\omega |y|) e^ {-i n (\theta_y - \frac{3}{2} \pi)} \phi_m(y) \; dS_y + O(|x|^{-1})   .
\label{far3}
\end{aligned}
\end{equation}
Since $i^n e^{in \frac{3 \pi}{2}} = i^n (-i)^n = 1$, we obtain
\begin{equation} 
\begin{aligned}
u-e^{i \omega \xi \cdot x} = -i e^{-i \frac{\pi}{4}} \frac{e^{i \omega |x|}}{\sqrt{8 \pi \omega |x|}} \sum_{m, n \in \mathbb{Z}} e^{i m (\frac{\pi}{2}-\theta_{\xi})} e^ {i n (\theta_x - \frac{\pi}{2})} W_{n,m} + O(|x|^{-1})   ,
\label{far4}
\end{aligned}
\end{equation}
where $W_{n,m}$ are the inhomogeneous scattering coefficients. 

We define the far-field pattern (the scattering amplitude) when the incident field is the plane wave $U(x)=e^{i \omega \xi \cdot x}$, $|\xi|=1$, as the two-dimensional $2 \pi$-periodic function 
$$A^{\infty}_B((\theta_{\xi}, \theta_x)^T; \sigma, \mu, \omega)= \sum_{m \in \mathbb{Z}} e^{i m (\frac{\pi}{2}-\theta_{\xi})} \int_{\partial B} e^{-i\omega|y|\cos(\theta_x-\theta_y)} \phi_m(y) \; dS_y.$$
From \eqref{far4}, we get the following proposition.
\begin{prop}
Let $(\theta_{\xi},\theta_x)^T \in [0, 2 \pi]^2$. Then, we have
\begin{equation}
\begin{aligned}
A^{\infty}_B((\theta_{\xi}, \theta_x)^T; \sigma, \mu, \omega) = \sum_{m, n \in \mathbb{Z}} e^{i m (\frac{\pi}{2}-\theta_{\xi})} e^ {i n (\theta_x - \frac{\pi}{2})} W_{n,m}[B,\sigma,\mu,\omega].
\end{aligned}
\label{far-field}
\end{equation}
\end{prop}

\subsection{Translation- and rotation-invariant distribution descriptors}

A simple relation exists between far-field patterns of $D=z+sR_{\theta}B$ and $B$. The following result generalizes to the inhomogeneous case the result proved in \cite{echolocation} in the case of homegenous scattering coefficients.
\begin{prop}
Let $D=z+sR_{\theta}B$. We denote by $\theta_z$ the angle of $z$ in polar coordinates, and define 
$$\phi_z((\theta_{\xi}, \theta_x)^T) :=e^{i\omega|z|\cos(\theta_{\xi} - \theta_z) } e^{-i \omega |z| \cos (\theta_x - \theta_z)}. $$
Then, we have
$$A^{\infty}_D((\theta_{\xi}, \theta_x)^T; \sigma, \mu, \omega) = \phi_z ((\theta_{\xi}, \theta_x)^T) A^{\infty}_B((\theta_{\xi} - \theta, \theta_x - \theta)^T, \sigma, \mu, s \omega) .$$ 
\end{prop}

\begin{proof}
By transformation formulas \eqref{translation}, \eqref{scaling}, and \eqref{rotation}, we have 
\begin{equation}
\begin{aligned}
W_{n,m}[D,\sigma,\mu, \omega] & = \sum_{a,b \in \mathbb{Z}} \overline{\mathcal{C}_a(z)} \mathcal{C}_b(z) W_{n-a,m-b}[sR_{\theta}B, \sigma, \mu, \omega] = \\ & =  \sum_{a,b \in \mathbb{Z}} \overline{\mathcal{C}_a(z)} \mathcal{C}_b(z) W_{n-a,m-b}[R_{\theta}B, \sigma, \mu, s\omega] = \\ & =
\sum_{a,b \in \mathbb{Z}} \overline{\mathcal{C}_a(z)} \mathcal{C}_b(z) e^{i(m-b) \theta} e^{-i(n-a) \theta} W_{n-a,m-b}[B, \sigma, \mu, s \omega].
\end{aligned}
\end{equation}
Therefore, using the Jacobi-Anger identity, 
\begin{equation*}
\begin{aligned}
e^{i\omega|z|\cos(\theta_{\xi}-\theta_z)} & = e^{i\omega|z|\cos(\theta_z - \theta_{\xi})} = \sum_{a \in \mathbb{Z}} i^{a} J_a(\omega |z|) e^ {i a (\theta_z - \theta_{\xi})} = \sum_{a \in \mathbb{Z}} J_a(\omega |z|) e^ {i a (\frac{\pi}{2} - (\theta_{\xi} - \theta_z))} = \\ & = \sum_{a \in \mathbb{Z}} \mathcal{C}_{a}(z) e^{ia(\frac{\pi}{2} - \theta_{\xi})},
\end{aligned}
\end{equation*}
we obtain
\begin{equation*}
\begin{aligned}
& A^{\infty}_D((\theta_{\xi}, \theta_x)^T; \sigma, \mu, \omega)  = \sum_{m, n \in \mathbb{Z}} e^{i m (\frac{\pi}{2}-\theta_{\xi})} e^ {i n (\theta_x - \frac{\pi}{2})} W_{n,m}[D,\sigma,\mu,\omega] = \\ & = \sum_{m, n, a, b \in \mathbb{Z}} e^{i m (\frac{\pi}{2}-\theta_{\xi})} e^ {i n (\theta_x - \frac{\pi}{2})} \overline{\mathcal{C}_a(z)} \mathcal{C}_b(z) e^{i(m-b) \theta} e^{-i(n-a) \theta} W_{n-a,m-b}[B, \sigma, \mu, s \omega] = \\ & = \sum_{m, n, a, b \in \mathbb{Z}}  \overline{\mathcal{C}_a(z)} e^{-ia(\frac{\pi}{2}-\theta_{x})} e^{ ia(\frac{\pi}{2}-\theta_{x})} \mathcal{C}_b(z) e^{i m (\frac{\pi}{2}-\theta_{\xi})} e^ {i n (\theta_x - \frac{\pi}{2})} e^{i(m-b) \theta} e^{-i(n-a) \theta} W_{n-a,m-b}[B, \sigma, \mu, s \omega] = \\ & = \sum_{m, n, a, b \in \mathbb{Z}}  \overline{\mathcal{C}_a(z)} e^{-ia(\frac{\pi}{2}-\theta_{x})} \mathcal{C}_b(z) e^{i m(\frac{\pi}{2}-\theta_{\xi})} e^ {i (n-a) (\theta_x - \frac{\pi}{2})} e^{i(m-b) \theta} e^{-i(n-a) \theta} W_{n-a,m-b}[B, \sigma, \mu, s \omega] = \\ & = \sum_{m, n, a, b \in \mathbb{Z}}  \overline{\mathcal{C}_a(z)} e^{-ia(\frac{\pi}{2}-\theta_{x})} \mathcal{C}_b(z) e^ {-i b (\theta_{\xi} - \frac{\pi}{2})} e^{i (m-b) (\frac{\pi}{2}-\theta_{\xi})} e^ {i (n-a) (\theta_x - \frac{\pi}{2})} e^{i(m-b) \theta} e^{-i(n-a) \theta} W_{n-a,m-b}[B, \sigma, \mu, s \omega] = \\ & = \sum_{m', n', a, b \in \mathbb{Z}}  \overline{\mathcal{C}_a(z)} e^{-ia(\frac{\pi}{2}-\theta_{x})} \mathcal{C}_b(z) e^ {-i b (\theta_{\xi} - \frac{\pi}{2})} e^{i m' (\frac{\pi}{2}-\theta_{\xi})} e^ {i n' (\theta_x - \frac{\pi}{2})} e^{im' \theta} e^{-in' \theta} W_{n',m'}[B, \sigma, \mu, s \omega] = \\ & = \sum_{m', n' \in \mathbb{Z}} e^{i \omega |z| \cos(\theta_{\xi}-\theta_z)} e^{-i \omega |z| \cos(\theta_{x}-\theta_z)} e^{i m' (\frac{\pi}{2}-\theta_{\xi})} e^ {i n' (\theta_x - \frac{\pi}{2})} e^{i(m'-n') \theta} W_{n',m'}[B, \sigma, \mu, s \omega] = \\ & = \sum_{m', n' \in \mathbb{Z}} \phi_z((\theta_{\xi}, \theta_x)^T) e^{i m' (\frac{\pi}{2}-\theta_{\xi})} e^ {i n' (\theta_x - \frac{\pi}{2})} W_{n',m'}[B^{\theta}, \sigma, \mu, s \omega] = \\ & = \phi_z ((\theta_{\xi}, \theta_x)^T) A^{\infty}_B((\theta_{\xi} - \theta, \theta_x - \theta)^T, \sigma, \mu, s \omega).
\end{aligned}
\end{equation*} 

\end{proof}

In the following, we introduce the descriptor construction based on the far-field pattern. We proceed as in \cite{echolocation}. Given $\eta = (\theta_{\xi},\theta_x)^T$, we define the frequency-dependent \textit{distribution descriptor} of an inhomogeneous object $D$ as follows:

\begin{equation}
S_D(v; \sigma, \mu, \omega) := \int_{[0,2\pi]^2} |A^{\infty}_D (\eta; \sigma, \mu, \omega) A^{\infty}_D (\eta-v; \sigma, \mu, \omega)| \; d\eta . 
\end{equation}

The distribution descriptor $S_D$ is invariant to any translation and rotation. More precisely, we can prove the following identity. 

\begin{prop}
Let $D=z+sR_{\theta}B$. We have
$$S_D(v; \sigma, \mu, \omega) = S_B(v; \sigma, \mu, s \omega).$$
\end{prop}

\begin{proof}
Given $\eta= (\eta_1, \eta_2)$, $\theta= (\theta, \theta)$, by $|\phi_z ( \cdot )|= 1$ and
\begin{equation*}
\begin{aligned}
\int_{[0,2 \pi]^2} A^{\infty}_B(\eta - \theta; \sigma, \mu, s\omega) \; d \eta & = \int_{[0,2 \pi]^2} \sum_{m, n \in \mathbb{Z}} e^{i m (\frac{\pi}{2}-(\eta_1 - \theta))} e^ {i n ((\eta_2 - \theta) - \frac{\pi}{2})} W_{n,m}[B,\sigma,\mu,s\omega] \; d \eta = \\ & = \int_{[-\theta,2 \pi-\theta]^2} \sum_{m, n \in \mathbb{Z}} e^{i m (\frac{\pi}{2}-\eta'_1)} e^ {i n (\eta'_2 - \frac{\pi}{2})} W_{n,m}[B,\sigma,\mu,s\omega] \; d \eta' = \\ & = \int_{[0,2 \pi]^2} \sum_{m, n \in \mathbb{Z}} e^{i m (\frac{\pi}{2}-\eta'_1)} e^ {i n (\eta'_2 - \frac{\pi}{2})} W_{n,m}[B,\sigma,\mu,s\omega] \; d \eta' = \\ & = \int_{[0,2 \pi]^2} A^{\infty}_B(\eta'; \sigma, \mu, s\omega) \; d \eta',
\end{aligned}
\end{equation*}
we have
\begin{equation*}
\begin{aligned}
& S_D(v; \sigma, \mu, \omega) = \int_{[0,2\pi]^2} |A^{\infty}_D (\eta; \sigma, \mu, \omega) A^{\infty}_D (\eta-v; \sigma, \mu, \omega)| \; d\eta = \\ & = \int_{[0,2\pi]^2} |\phi_z (\eta) A^{\infty}_B(\eta - (\theta, \theta), \sigma, \mu, s \omega) \phi_z (\eta - v) A^{\infty}_B (\eta- v -(\theta, \theta); \sigma, \mu, s \omega)| \; d\eta = \\ & = \int_{[0,2\pi]^2} | A^{\infty}_B(\eta - (\theta, \theta), \sigma, \mu, s \omega) A^{\infty}_B (\eta-(\theta, \theta) - v; \sigma, \mu, s \omega)| \; d\eta = \\ & = \int_{[0,2\pi]^2} | A^{\infty}_B(\eta', \sigma, \mu, s \omega) A^{\infty}_B (\eta' - v; \sigma, \mu, s \omega)| \; d\eta' = \\ & = S_B(v, \sigma, \mu, s \omega).    
\end{aligned}
\end{equation*}

\end{proof}

\section{Numerical experiments} \label{sec6}

In this section, we present a variety of numerical results in order to demonstrate the applicability of the theoretical framework presented in the previous sections. In particular, we investigate the identification of a target by reconstructing inhomogeneous scattering coefficients from the measurements of the multistatic response (MSR) matrix. In the multistatic configuration, directions of incidence and observations are sampled. For each incident direction, the scattered wave is measured in all the observation directions \cite{math and stat method}.  The overall procedure is similar to the one of \cite{echolocation} for the homogeneous case. In the following, we consider the case of piecewise constant (inhomogeneous) material parameters. For a collection of (inhomogeneous) targets, based on the code developed in \cite{code} for homogeneous targets, we build a frequency dependent dictionary of distribution descriptors and use a target identification algorithm like the one of \cite{echolocation} in order to identify an inhomogeneous target from the dictionary up to some translation, rotation and scaling. Our dictionary will include three kinds of objects: 

\begin{itemize}
	\item \textit{Homogeneous targets}, i.e. a disk, a triangle, \textit{etc.}
	\item \textit{Inhomogeneous targets with one inclusion inside}, i.e. a circular inclusion inside a circular target, \textit{etc.}
	\item \textit{Inhomogeneous targets with two (distinct) inclusions inside}, i.e. a circular inclusion and a square inside a circular target, \textit{etc.}
\end{itemize}
Note that these inclusions have different material parameters than the ones of target and the background. In the following, we use the results of Section \ref{sec4} for a suitable integral representation of the solutions for the case of an inhomogeneous target with one inclusion inside and the case of an inhomogeneous object with two inclusions inside (see Appendix \ref{secappendix}). The case of a homogeneous target is taken into account in \cite{echolocation}. Finally, we perform numerical experiments in order to test the performance of the inhomogeneous scattering coefficients in inhomogeneous target identification. 

Given a target $D$, we proceed as follows:

\begin{itemize}
	\item \textit{Data simulation}. The MSR matrix is simulated for a frequency range $[\omega_{\text{min}}, \omega_{\text{max}}]$ by evaluating the integral representation \eqref{solrep}, where the densities are computed by solving \eqref{scatsys}. We adopt a circular acquisition system (a full-view acquisition).
	\item \textit{Reconstruction of scattering coefficients}. For each frequency, we reconstruct the matrix $\textbf{W}=(W_{mn}[D^{-z_0}])_{mn}$ of scattering coefficients directly by the formula (3.32) of \cite{echolocation}.
	\item \textit{Target identification}. We calculate the distribution descriptors and use a target identification algorithm, see \cite{echolocation}.
\end{itemize}

\subsection{Dictionary}
The dictionary $\mathcal{D}$ that we consider is composed by 14 targets with different material parameters: 

\begin{itemize}

\item 6 elements of the dictionary are homogeneous targets: a disk, an ellipse, a triangle, a square, a rectangle, and the letter A (see Figure \ref{fig:dico}). All homogeneous targets share the same permittivity $\mu=3$ and permeability $\sigma=3$.

\item 5 elements of the dictionary are inhomogeneous targets with a single inclusion inside: a disk with a circular inclusion inside, a disk with an ellipse inside, a disk with a triangular inclusion inside, a disk with a square inside, and a disk with a rectangular inclusion inside (see Figure \ref{fig:dico}). Note that these 5 targets share the same permittivity ($\mu_e=3$) and permeability ($\sigma_e=3$) for the exterior domain, while all inclusions have permittivity $\mu_i=6$ and permeability $\sigma_i=6$. 

\item 3 elements of the dictionary are (inhomogeneous) disks with two distinct inclusions inside: two circular inclusions for the first disk, a circle and an ellipse for the second disk, and two distinct ellipses for the third one (see Figure \ref{fig:dico}). These 3 targets share the same permittivity ($\mu_e=3$) and permeability ($\sigma_e=3$) for the exterior domain, while the two distinct inclusions have permittivity $\mu_i=6$ and permeability $\sigma_i=6$, $i=1,2$. 
\end{itemize}

\begin{figure}[H]
	\centering
	\begin{tabular}{@{}c@{\hspace{0.75cm}}c@{\hspace{0.75cm}}c@{\hspace{0.75cm}}c@{\hspace{0.75cm}}c@{\hspace{0.75cm}}c}
		& & & & & \\
		\quad & \raisebox{7ex - \height}{\includegraphics[ scale=0.25]{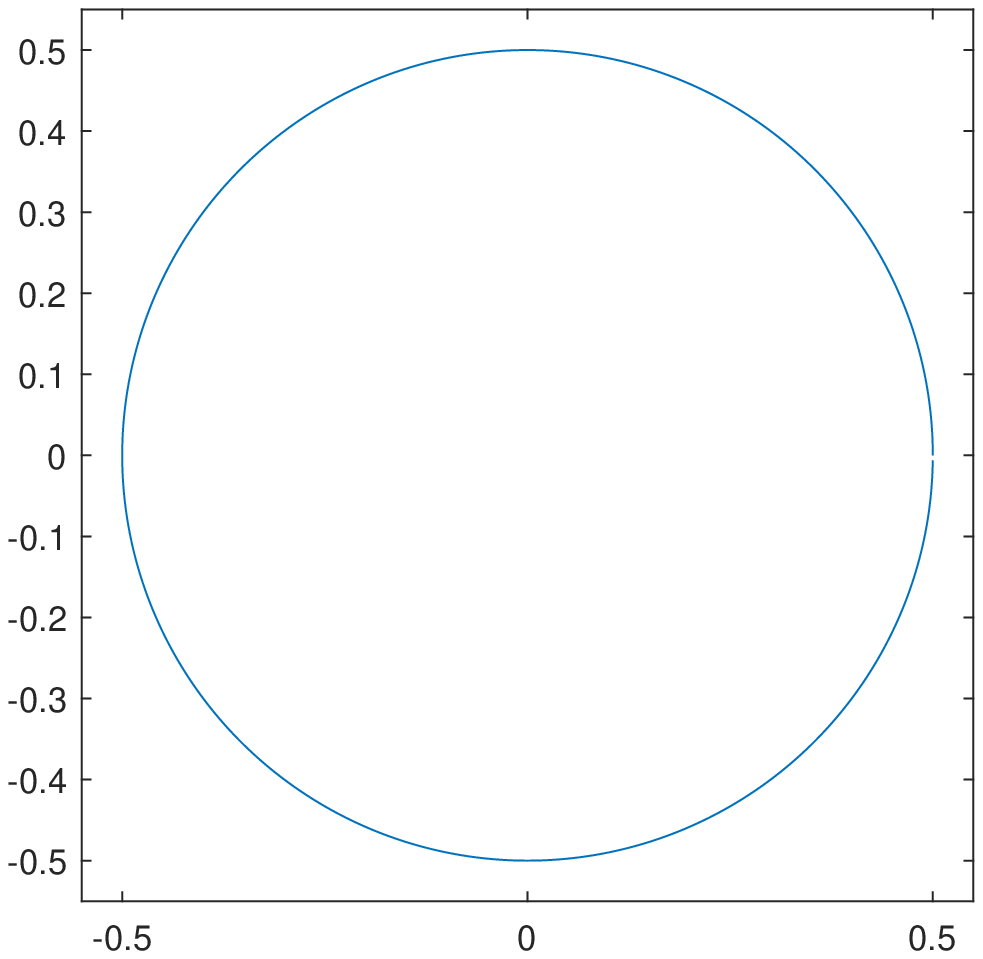}} &
		\raisebox{7ex - \height}{\includegraphics[ scale=0.25]{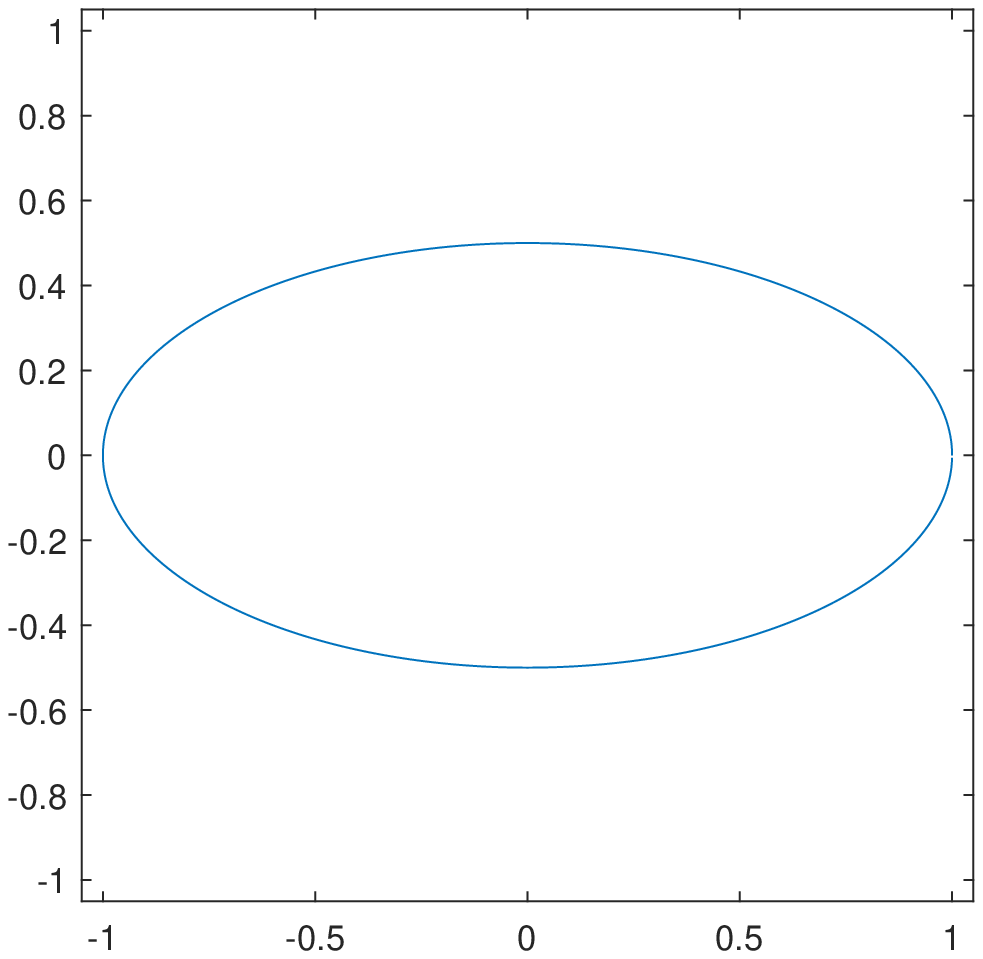}} &
		\raisebox{7ex - \height}{\includegraphics[ scale=0.25]{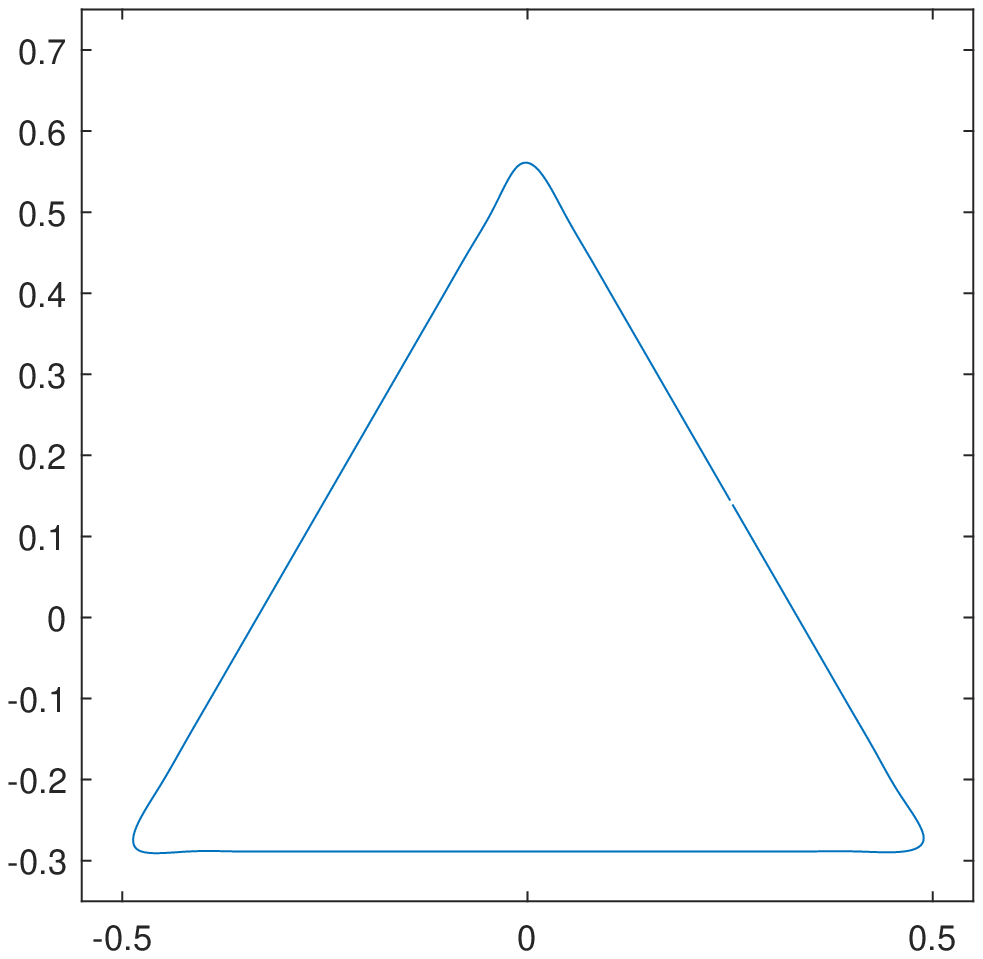}} & \raisebox{7ex - \height}{\includegraphics[ scale=0.25]{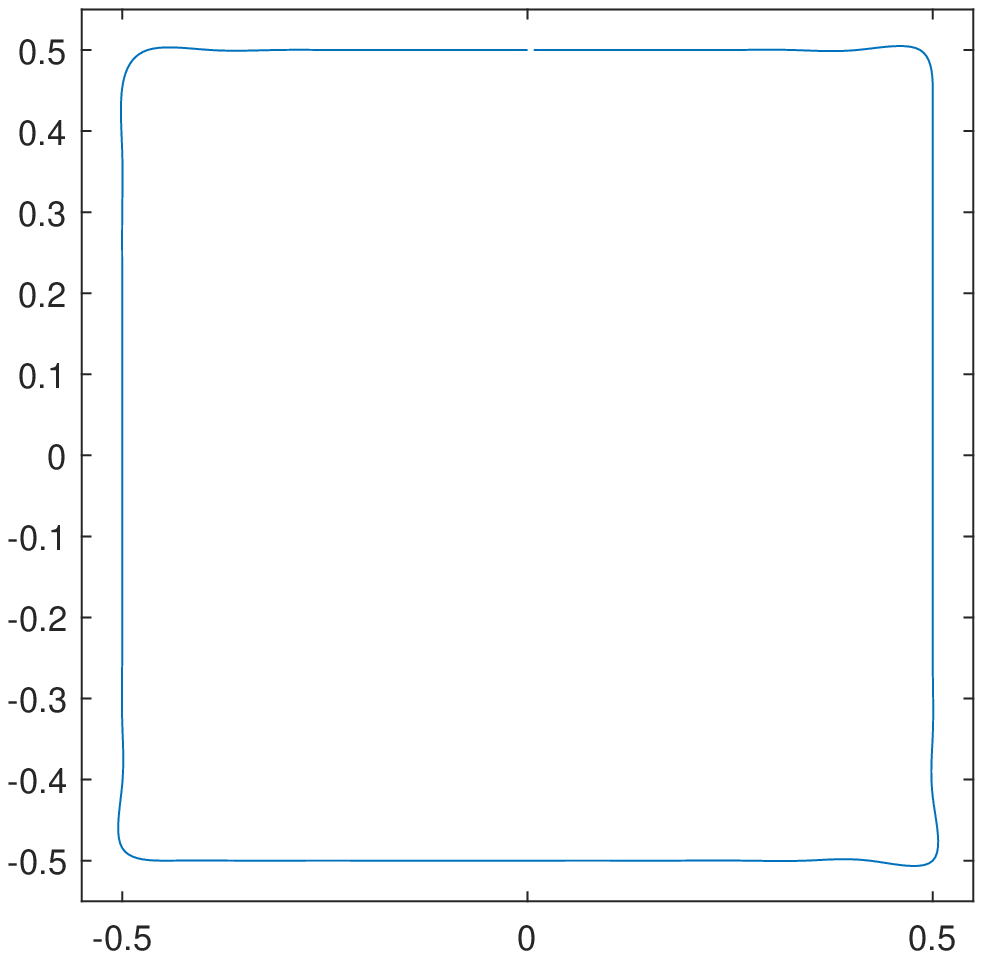}} &
		\raisebox{7ex - \height}{\includegraphics[scale=0.25]{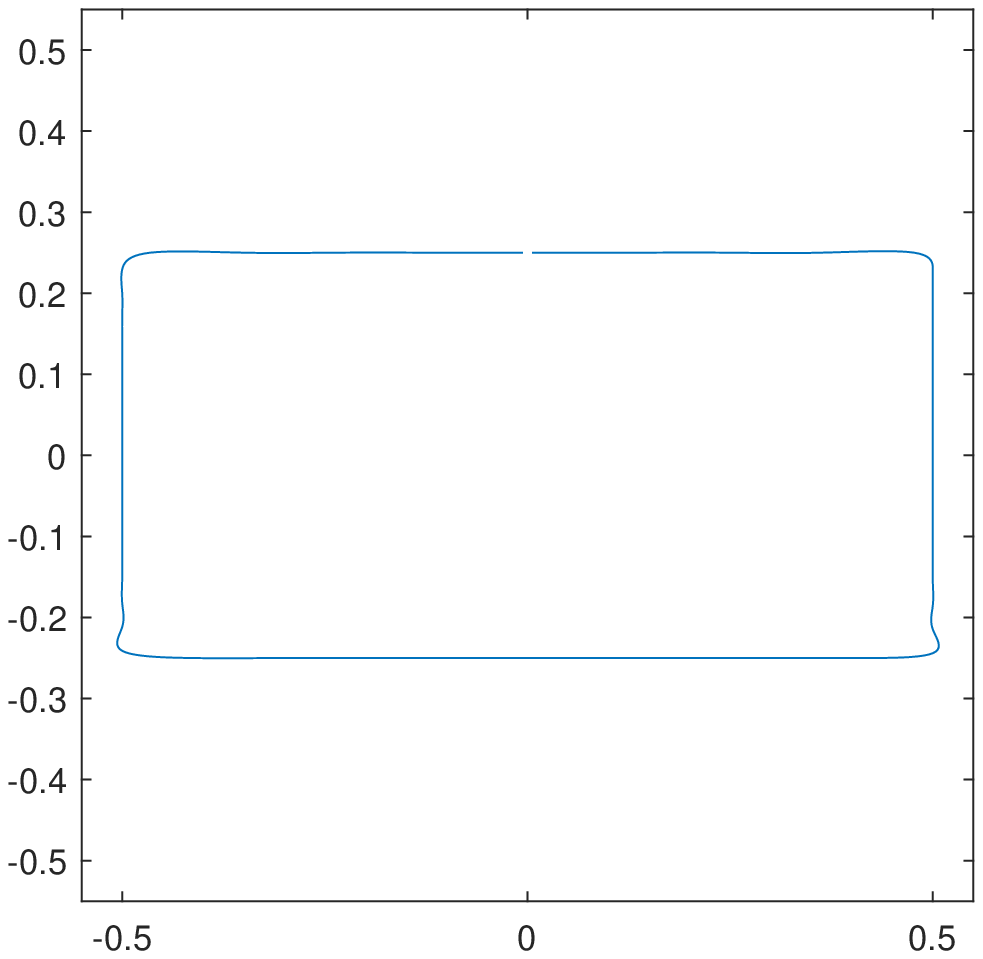}} \\ \rule{0pt}{7ex}
				\quad & \raisebox{5ex - \height}{\includegraphics[ scale=0.25]{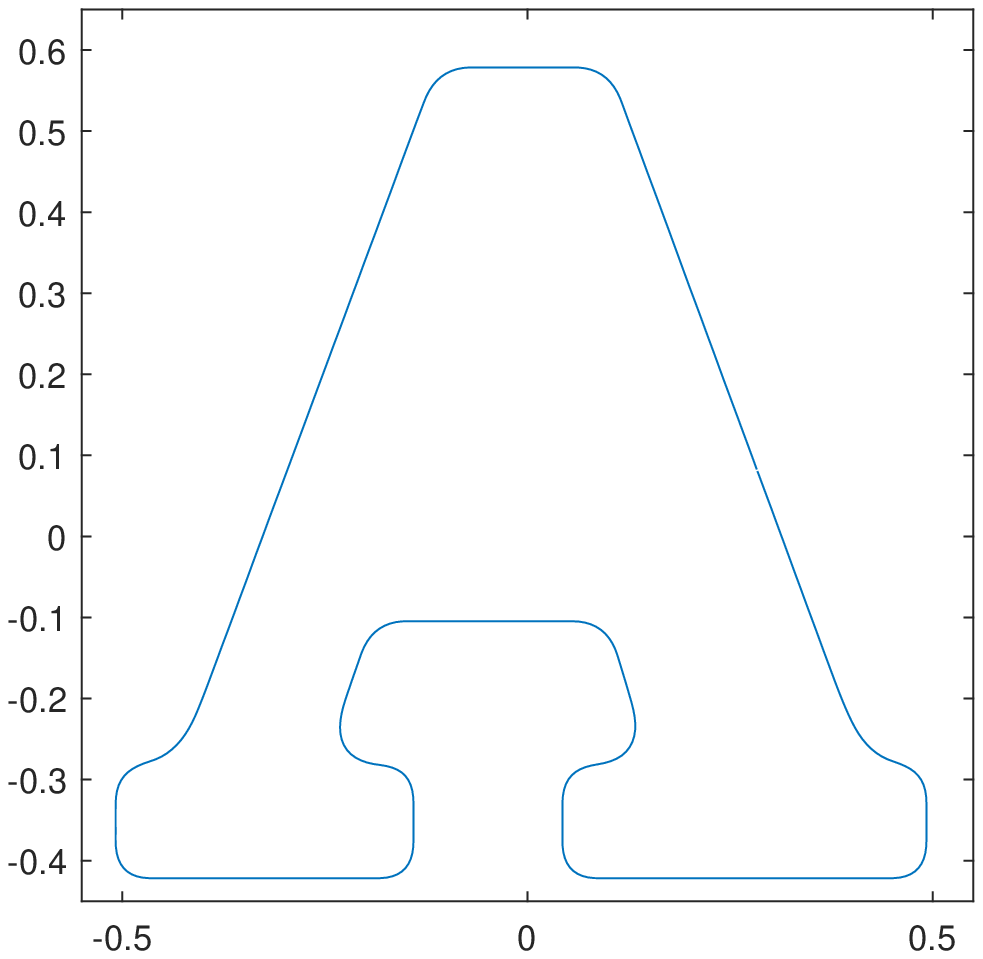}} &
		\raisebox{5ex - \height}{\includegraphics[ scale=0.25]{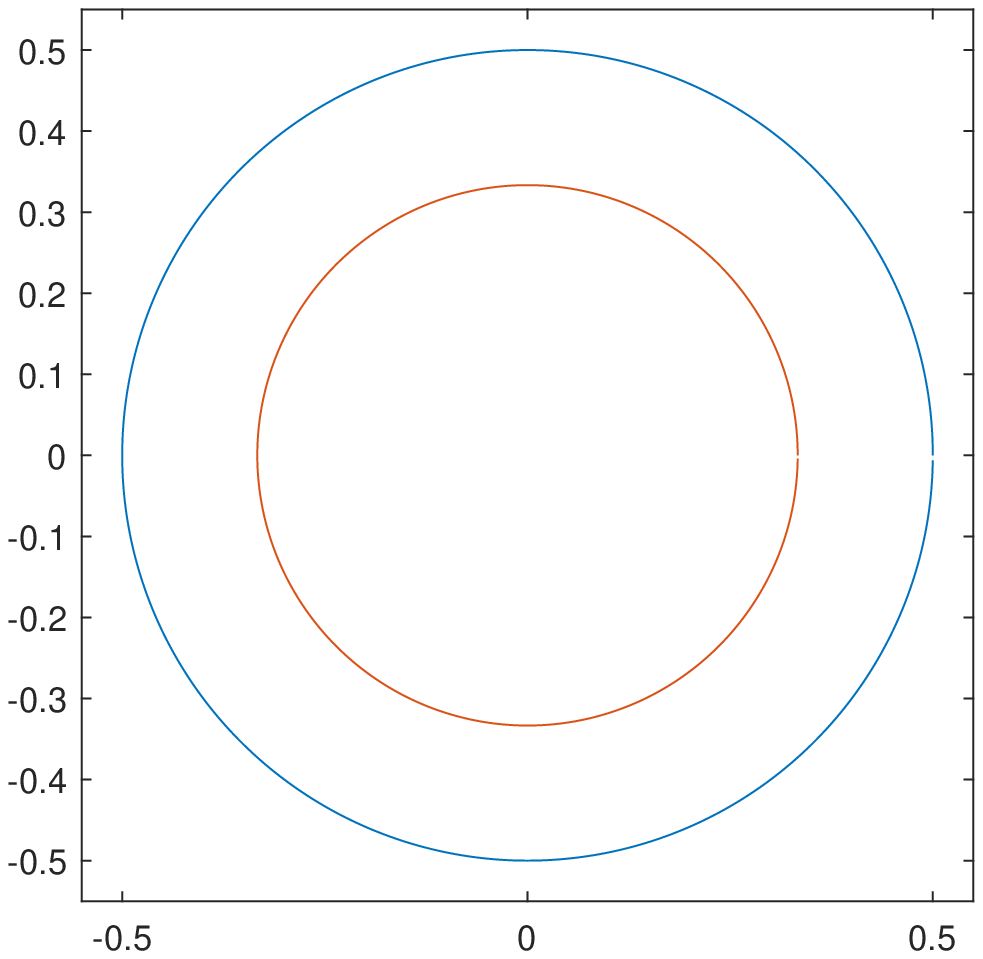}} &
		\raisebox{5ex - \height}{\includegraphics[ scale=0.25]{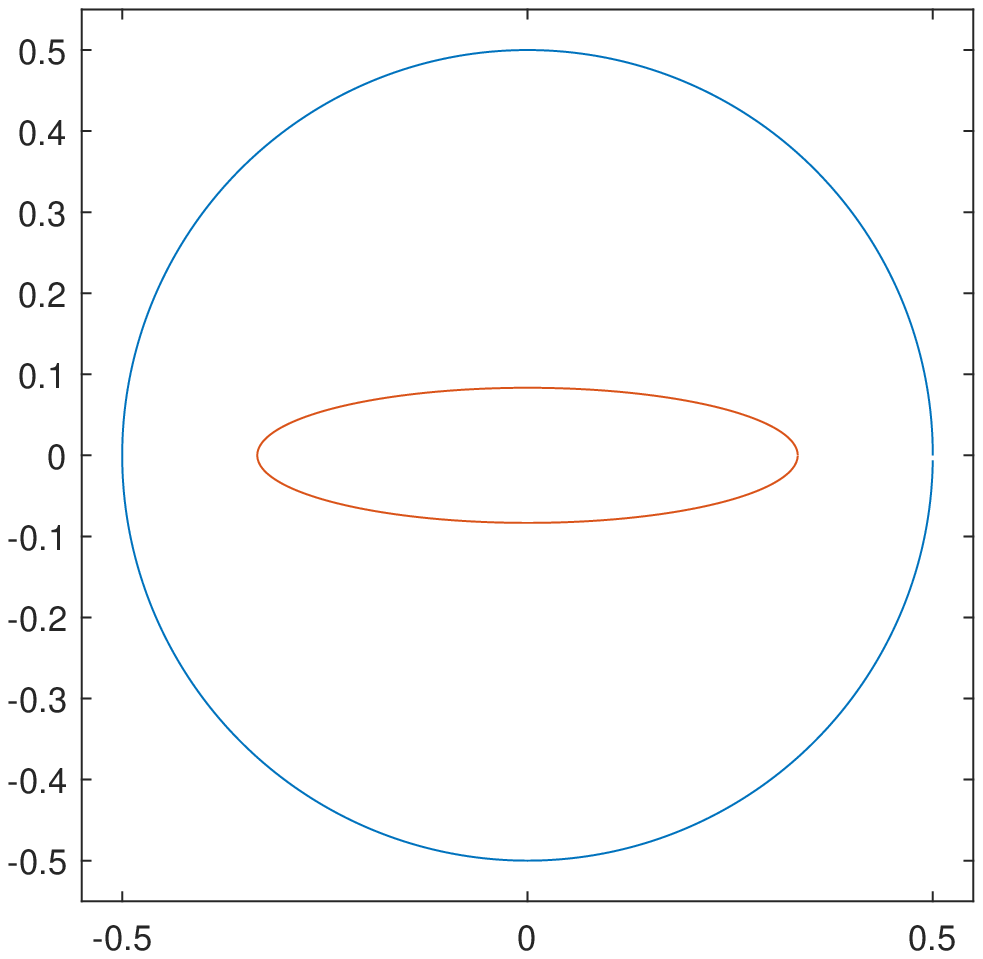}} & \raisebox{5ex - \height}{\includegraphics[ scale=0.25]{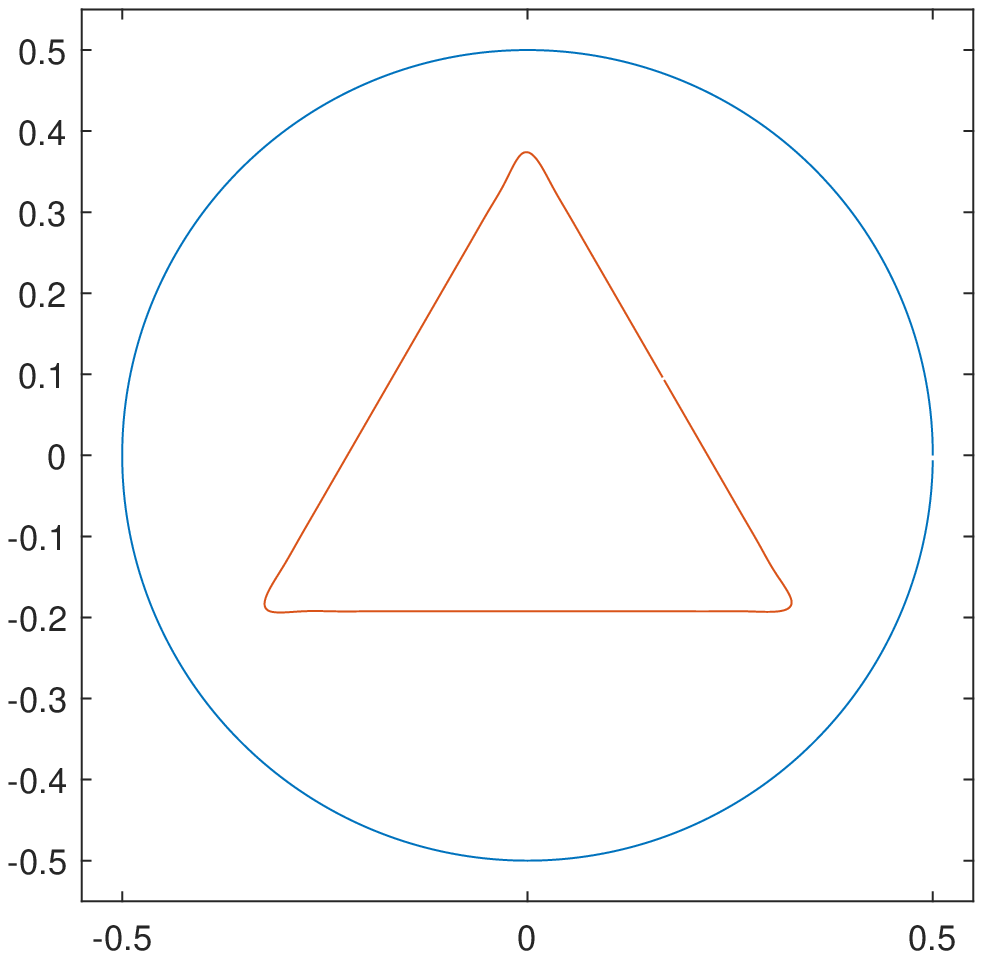}} &
		\raisebox{5ex - \height}{\includegraphics[scale=0.25]{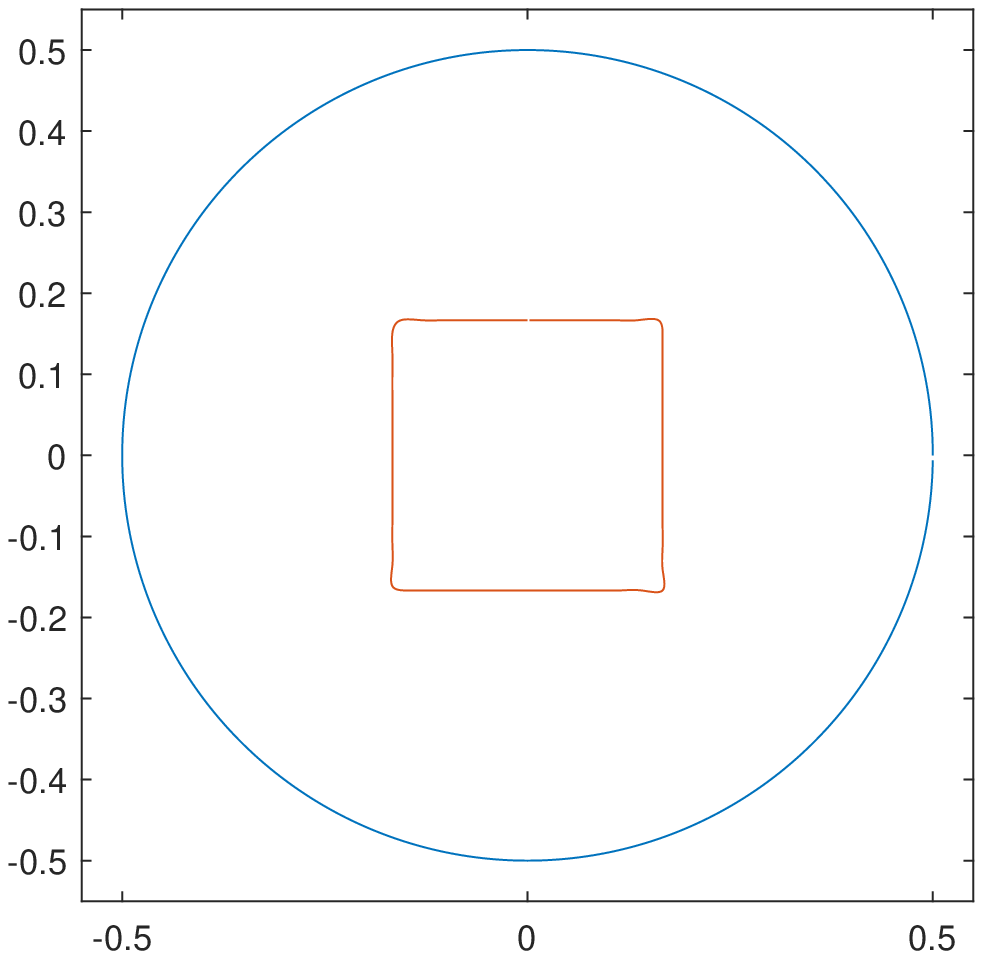}} \\ \rule{0pt}{7ex}
				\quad & \raisebox{5ex - \height}{\includegraphics[ scale=0.25]{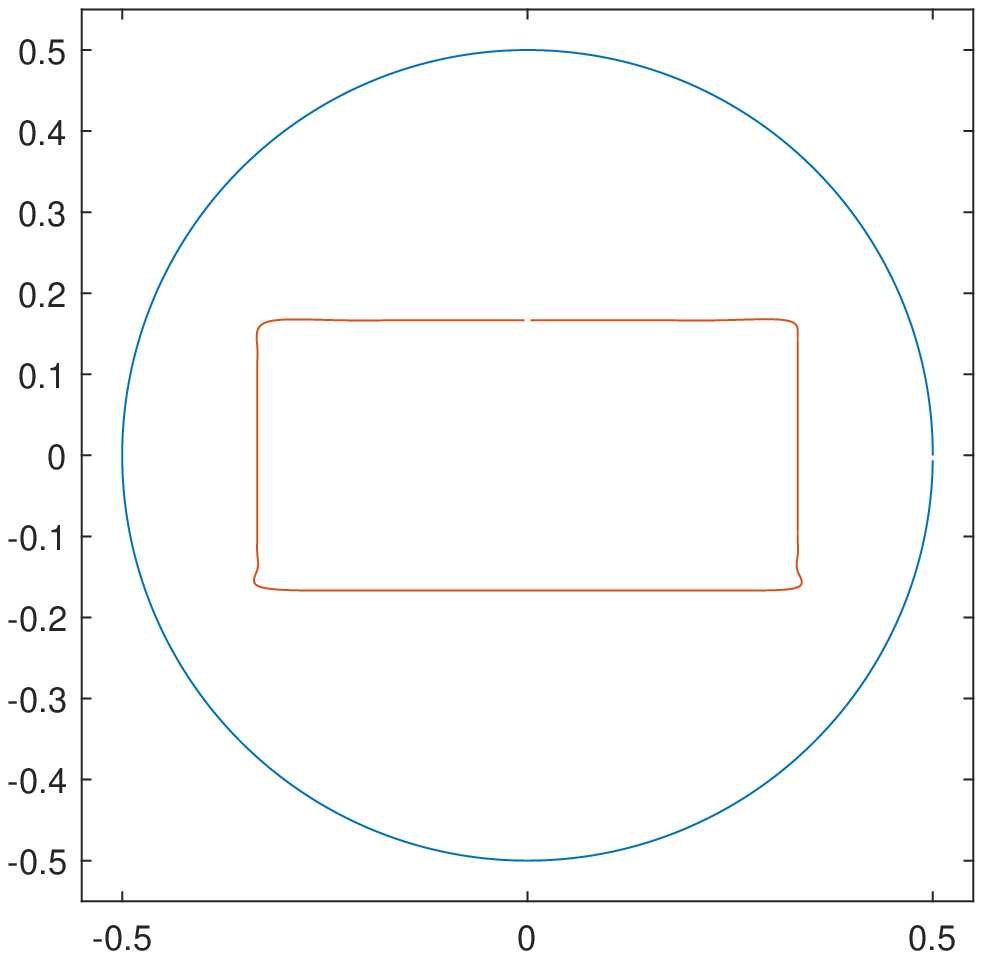}} &
		\raisebox{5ex - \height}{\includegraphics[ scale=0.25]{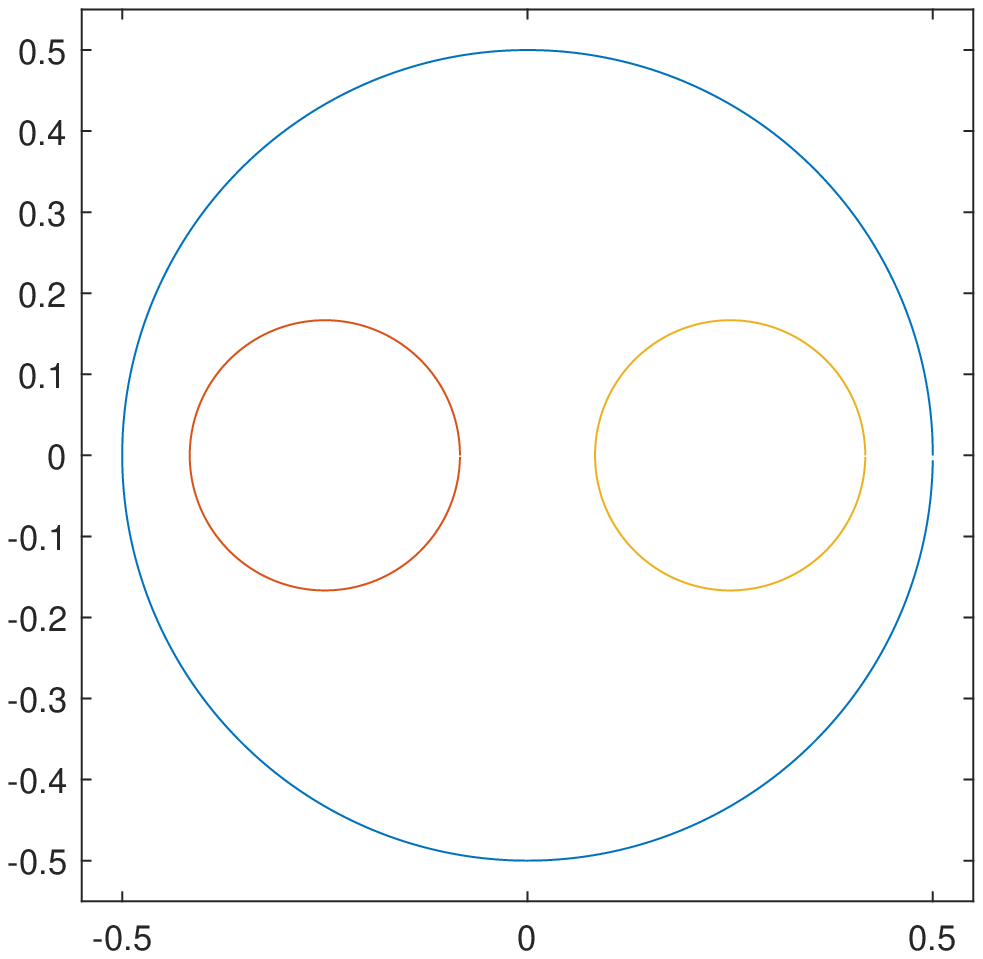}} &
		\raisebox{5ex - \height}{\includegraphics[ scale=0.25]{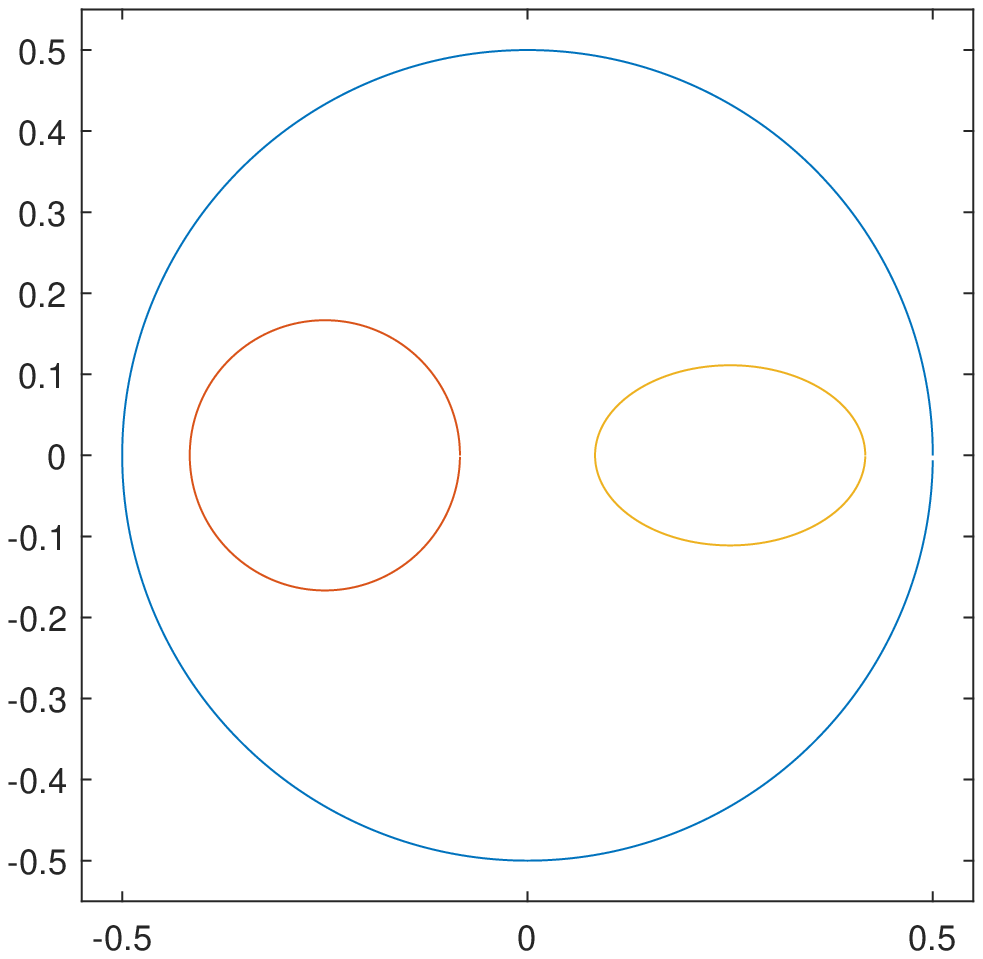}} & \raisebox{5ex - \height}{\includegraphics[ scale=0.25]{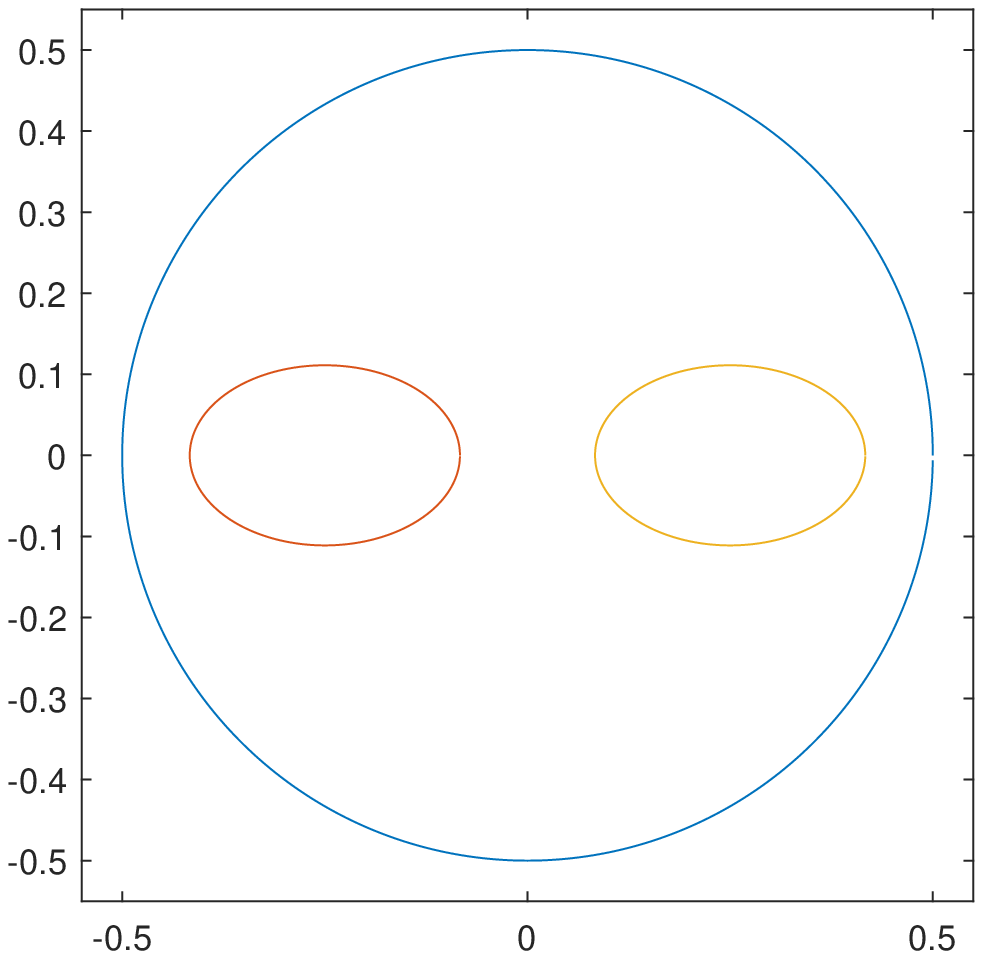}} &
		 \\
	\end{tabular}
	\caption{\textit{Dictionary of targets}}
	\label{fig:dico}
\end{figure}

\subsection{Acquisition system}
We generate a circular acquisition system using plane waves: the receivers $x_r$ are uniformly distributed on a circle of radius $R$ and centered at $z_0$, and the sources are plane waves with equally distributed wave direction. More precisely, for the $r$th receiver we have $|x_r - z_0| = R$ and the angle $\theta_r := \theta_{x_r - z_0} = 2 \pi r/N_r$, and the $s$th plane wave source is given by
$$U_s(x) = e^{ik_0 \xi_s \cdot x},$$
where the vector $\xi_s$ is such that $|\xi_s|=1$ and $\theta_s := 2 \pi s/N_s$. We denote by $N_s$ the total number of plane waves as sources, and by $N_r$ the number of receivers. Note that $z_0$ can be obtained using some localization algorithm \cite{math and stat method}. Here, we assume that $z_0$ is close to the center of the target $D$. 

For this experiment, we adopt a circular acquisition system with $R = 3$, $N_s = 91$, and $N_r = 91$. For simplicity, we always choose the center $z_0 = [0, 0]^T$. Figure \ref{fig:acq} illustrates this acquisition system for the three different elements of the dictionary $\mathcal{D}$.

\subsection{Measurements}
In each numerical experiment, the unknown target $D$ is obtained from one of the elements of the dictionary by a rotation with angle $\theta=\pi/3$, a scaling $s=1.2$ and a translation $z = [-0.5, 0.5]^T$.

Each element of the dictionary is approximated by $2^9$ points. We reconstruct the matrix $\textbf{W}=(W_{mn})_{mn}$ of scattering coefficients at order $25$. Figure \ref{fig:err-rel} plots the relative error of the (analytical) reconstruction $\| \textbf{W}^{\text{est}} - \textbf{W} \|_F / \| \textbf{W} \|_F$ as a function of $K$ for the three kinds of targets in the dictionary ($\|\cdot\|_F$ denotes the Frobenius norm of matrices). It can be seen that the reconstruction is robust: for example, in the case of a disk with a circular inclusion inside, with 20\% of noise, the error is less than 10\% for an order $K$ up to $45$.

\subsection{Scale estimation}

Given an unknown target $D_n=z + s R_{\theta} B_n$ and a dictionary of (inhomogeneous) objects $\mathcal{D}=(B_n)_n$, by measurements we reconstruct the distribution descriptor $S_{D_n} (v; \omega)$ and build a frequency dependent dictionary of distribution descriptors $(S_{B_n}(v; \omega))_
n$. 

Note that the distribution descriptor of the target $S_{D_n}$ is frequency dependent. As we proved in the previous sections, since the frequency $\omega$ is coupled with the scaling factor $s$, which is unknown and arbitrary in $(0, \infty)$, to adapt the distribution descriptor $S_{D_n}$ to target identification we assume that the physical operating frequency is limited, that is $0 < \omega_{\min} \leq \omega \leq \omega_{\max} < \infty $, and that $0 < s_{\min} \leq s \leq s_{\max} < \infty$, which means that the target we are interested in should not be too small or too large. Finally, $s^{\text{est}}$ can be estimated as in \cite{echolocation} by solving
\begin{equation}
s^{\text{est}} = \mbox{arg min}_{s \in [s_{\min}, s_{\max}]} \left\{ \int_{\omega_{\min}}^{\omega_{\max}} \left( \int_{[0, 2 \pi]^2} [S_{D_n}(v; \omega) - S_{B_n}(v ; s \omega)] \; \mbox{d} v \right)^2 \mbox{d} \omega \right\}.
\label{estim}
\end{equation} 
Note that a wide range of frequencies $[\omega_{\min}, \omega_{\max}]$ brings more information and therefore improves the estimation \eqref{estim}. 

\begin{figure}[t]
	\centering
	\begin{tabular}{@{\hspace{-0.75cm}}c@{\hspace{1.1cm}}c@{\hspace{1.1cm}}c@{\hspace{1.1cm}}c}
		& & & \\
		\quad & \raisebox{7ex - \height}{\includegraphics[ scale=0.425]{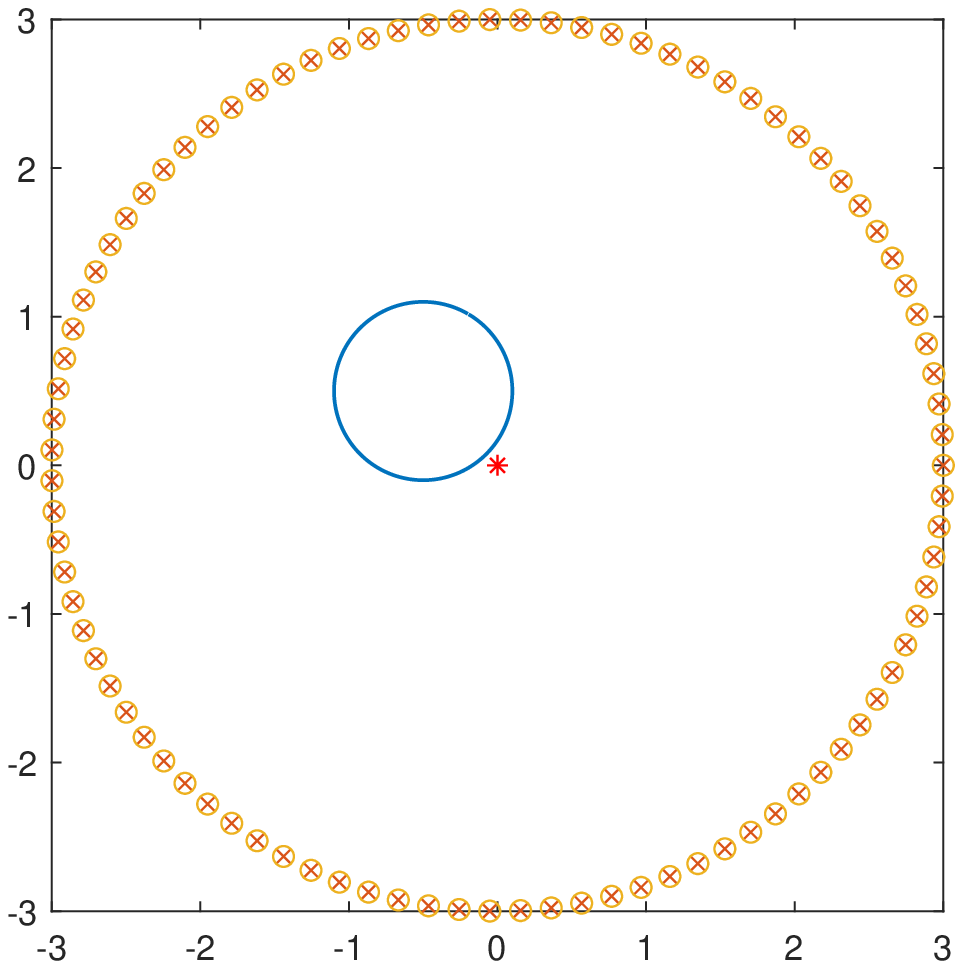}} &
		\raisebox{7ex - \height}{\includegraphics[ scale=0.425]{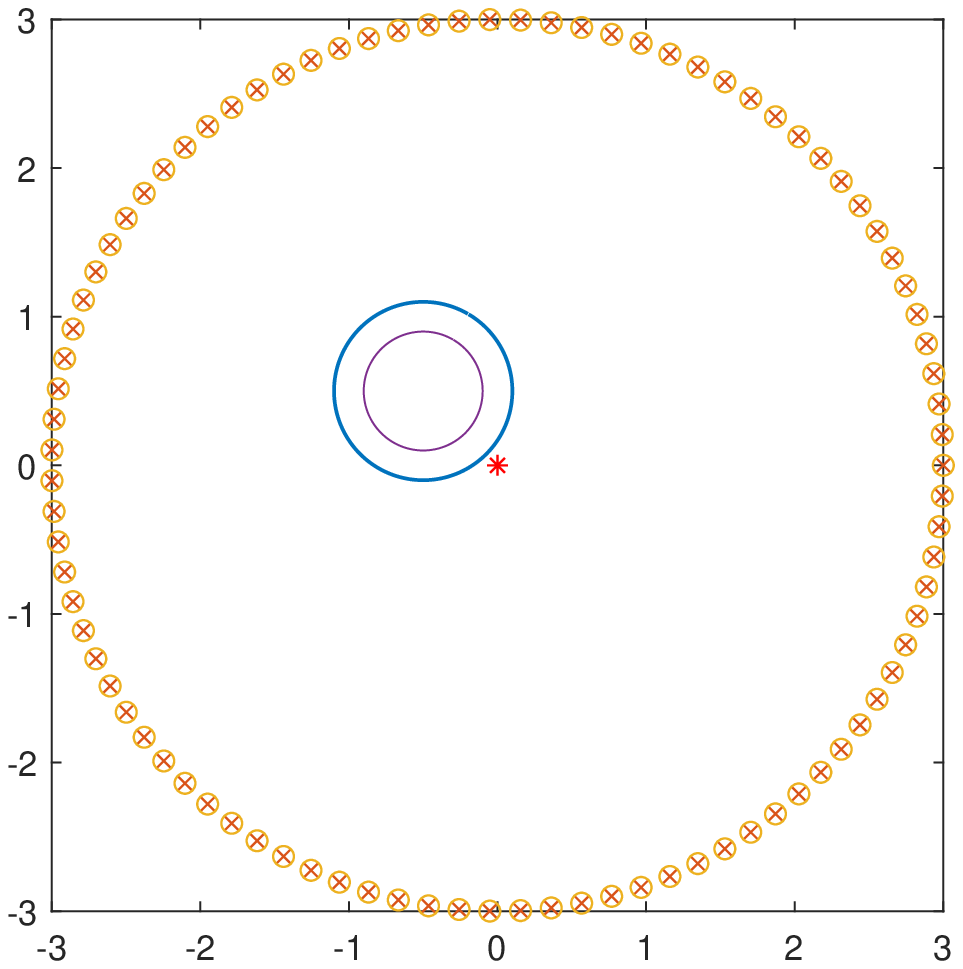}} &
		\raisebox{7ex - \height}{\includegraphics[ scale=0.425]{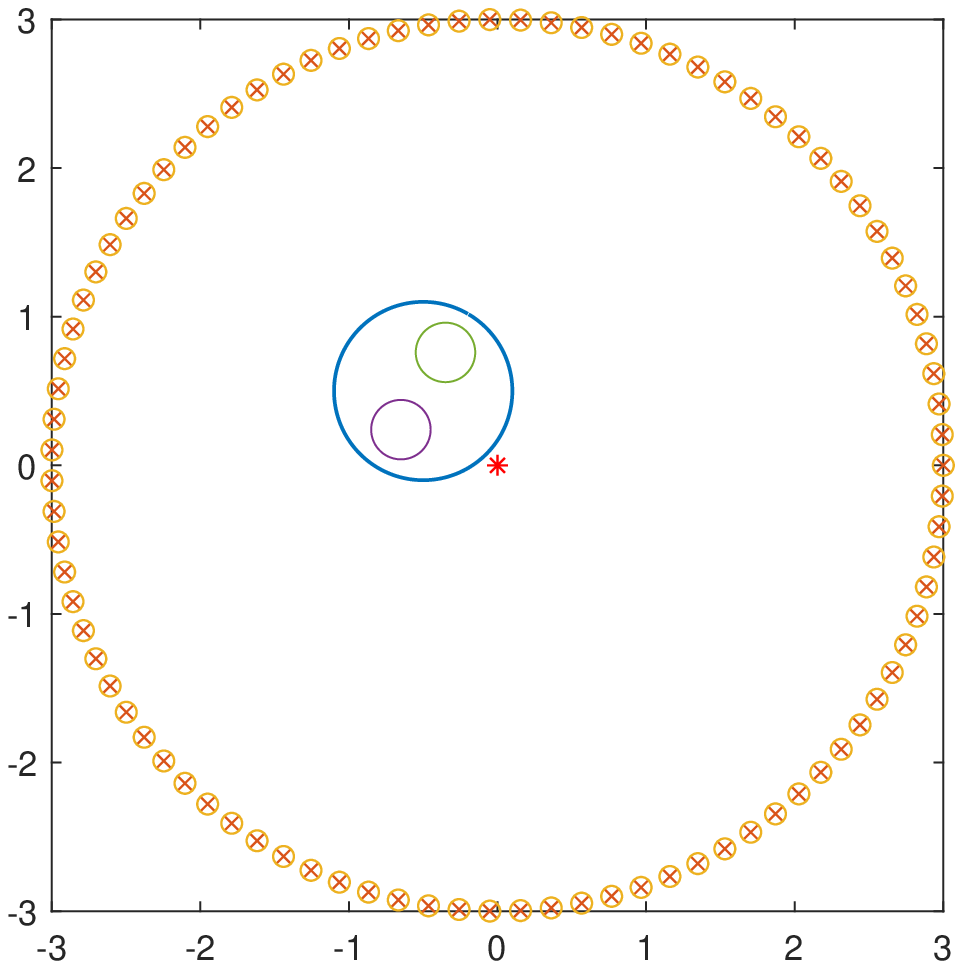}} \\ \quad & \raisebox{4ex - \height} \; (a) Homogeneous disk & \raisebox{4ex - \height} \; (b) Disk with a disk inside & \raisebox{4ex - \height} \; (c) Disk with two disks inside\\
	\end{tabular}
	\caption{\textit{Circular acquisition systems for three kinds of targets. We adopt full aperture with Ns = 91 plane wave sources (angular position is marked by $\circ$) and Nr = 91 receivers (marked by $\times$). Measurement center is marked by $\ast$. Figure (a) shows a circular acquisition system for an homogeneous target (a disk); (b) shows a circular acquisition system for an inhomogeneous target (a disk with a circular inclusion inside); (c) shows a circular acquisition system for an inhomogeneous disk with two distinct inclusions inside.}}
	\label{fig:acq}
\end{figure}

\vspace{6mm}

\subsection{Numerical implementation}

We can solve \eqref{estim} by sampling. The overall procedure is similar to the one of \cite{echolocation} for the homogeneous case. Let $N^{\text{dic}}_{\omega}$, $N_{\omega}$, $N_v$, and $N_{\delta}$ be positive integers. We define: 

\begin{itemize} 
\item $(\omega_l^{\text{dic}})_{l = 0, \ldots , N^{\text{dic}}_{\omega}}$ uniformly distributed points on $[\omega^{\text{dic}}_{\min}, \omega^{\text{dic}}_{\max}]$, with 

$$\omega^{\text{dic}}_{\min}:= \omega_{\min} s_{\min},$$ 
$$\omega^{\text{dic}}_{\max}:= \omega_{\max} s_{\max}.$$ 

\item $(\omega_k)_{k = 0, \ldots ,N_{\omega}}$ uniformly distributed points on $[\omega_{\min}, \omega_{\max}]$. 

\item $((v^1_i,v^2_j))_{i,j = 1, \ldots, N_v}$ uniformly distributed points on $[0, 2 \pi]^2$.

\item $(s_t)_{t = 0, . . . ,N_{\delta}}$ uniformly distributed points on $[s_{\min}, s_{\max}]$. 

\item $I_k(s):=\{ 1 \leq l \leq N^{\text{dic}}_{\omega}, \mbox{ such that } \omega^{\text{dic}}_{l-1} \leq s \omega_k \leq \omega^{\text{dic}}_l\}$. 

\end{itemize}.	
The distribution descriptors $S_{B_n}$ and $S_{D_n}$ are sampled at discrete positions as follows: 
$$S^{D_n}_{ijk} := S_{D_n}((v^1_i , v^2_j);\omega_k), \; S^{B_n}_{ijl} := S_{B_n}((v^1_i , v^2_j);\omega_l).$$
Finally, we discretize the functional inside the argmin in \eqref{estim}:
$$J(t; D_n,B_n) = \sum_{k=0}^{N_{\omega}} \; \sum_{l \in I_k(s_t)} \left( \sum_{i,j=1}^{N_v} (S^{D_n}_{ijk}-S^{B_n}_{ijl}) \right)^2,$$ 
and the scaling factor $s^{\text{est}}$ can be estimated by solving
$$\epsilon(D_n,B_n) = \min_{t=0,\ldots,N_{\delta}} J (t;D_n,B_n).$$

\begin{figure}[h!]
	\centering
	\begin{tabular}{ccc}
		& & \\
		\quad & \raisebox{7ex - \height}{\includegraphics[ scale=0.525]{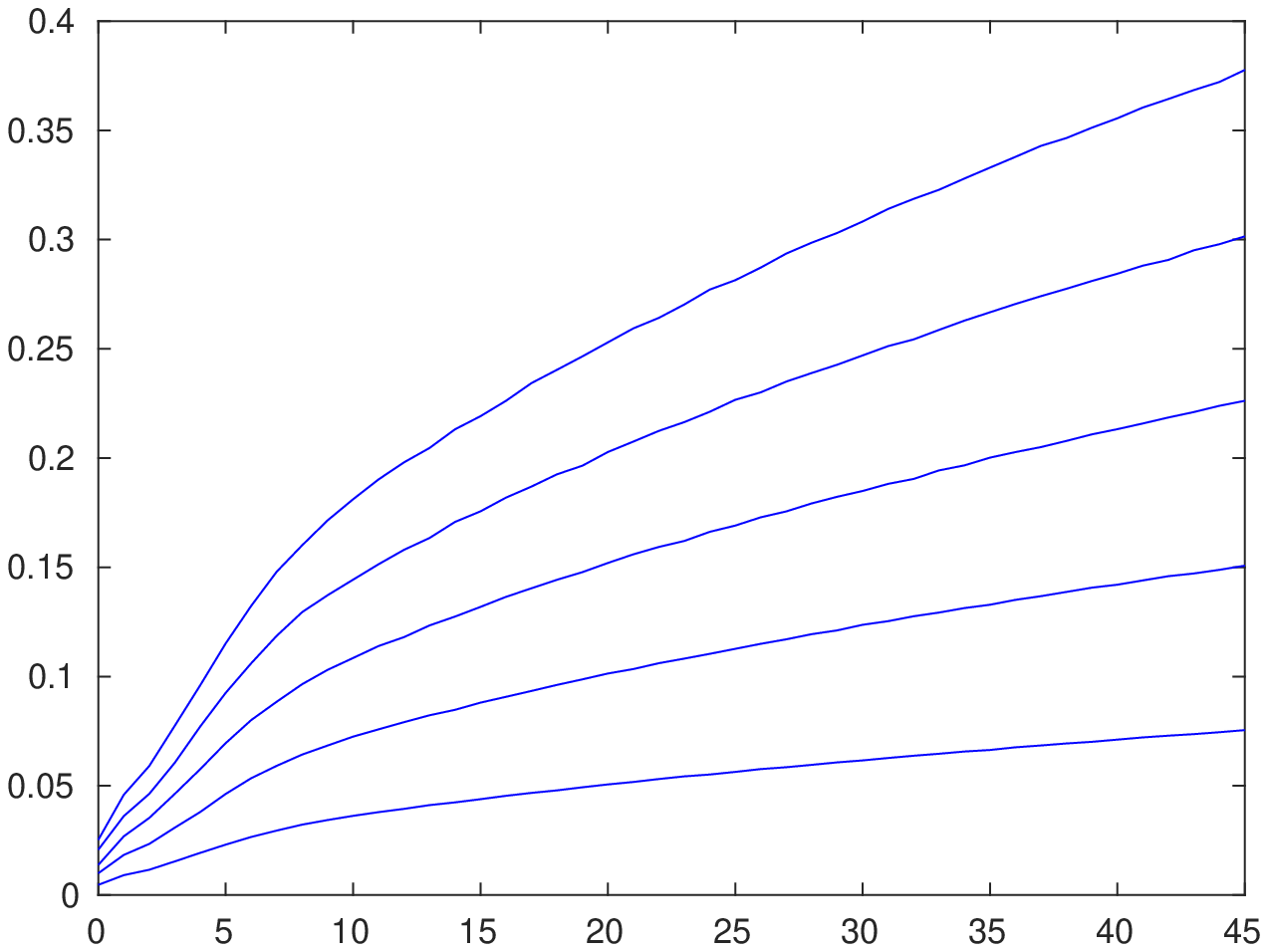}} &
		\raisebox{7ex - \height}{\includegraphics[ scale=0.525]{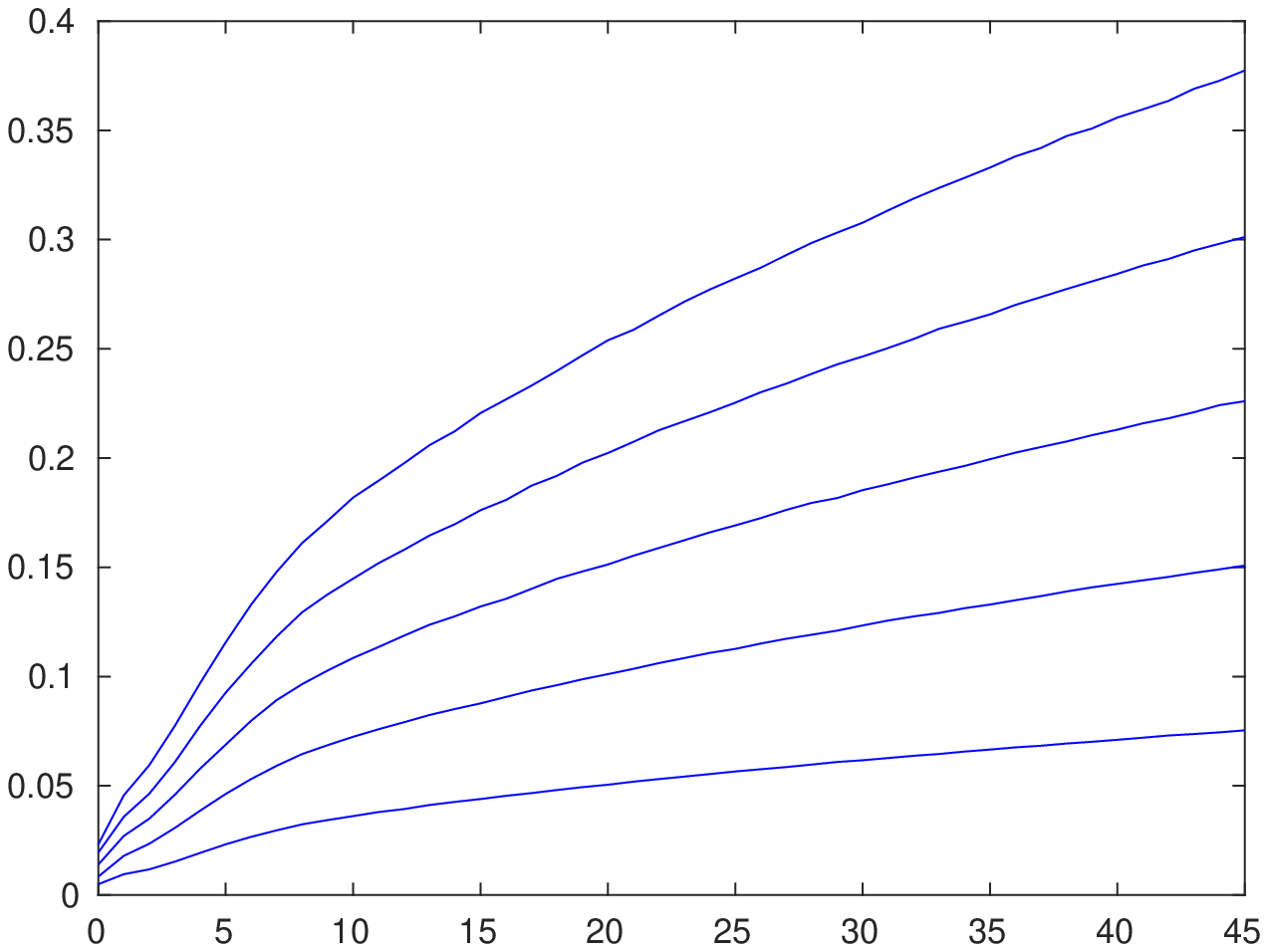}} \\ \quad & \raisebox{4ex - \height} \; (a) Disk with a circular inclusion inside & \raisebox{4ex - \height} \; (b) Disk with two circular inclusions inside\\
	\end{tabular}
	\caption{\textit{Relative error of the reconstruction} $\| \textbf{W}^{\text{est}} - \textbf{W} \|_F / \| \textbf{W} \|_F$ \textit{for the systems (b) and (c) in Figures \ref{fig:acq} at a different order K and fixed frequency $\omega= 0.75 \pi$. The curves from bottom to top correspond to percentage of noise $\sigma_0 = 20\%$, $40\%$, $60\%$, $80\%$, and $100\%$. The experiments have been repeated 100 times.}}
	\label{fig:err-rel}
\end{figure}

\subsection{Frequency-dependent dictionary and matching algorithm}

We construct the frequency-dependent dictionary of distribution descriptors as follows. For a collection of standard elements of the dictionary $(B_n)_n$, we precompute the discrete samples $(S^{B_n}_{ijl})_{ijl}$ of the distribution descriptor $S_{B_n}(v; \omega)$, for $v \in [0, 2 \pi]^2$ and $\omega \in [\omega^{\text{dic}}_{\min}, \omega^{\text{dic}}_{\max}]$. The discrete samples $((S^{B_n}_{ijl})_{ijl})_n$ constitute our frequency-dependent dictionary.

Assume that our (inhomogeneous) target $D$ is generated by an element of the dictionary $(B_n)_n$, up to some unknown translation, rotation, and scaling. Suppose that the scaling factor is such that $s_{\min} \leq s \leq s_{\max}$, where $s_{\min}$ and $s_{\max}$ are known. In order to detect the target $D$ among the elements of the dictionary, we compute the discrete samples $(S^D_{ijk})_{ijk}$ of the distribution descriptor $S_D(v; \omega)$, and calculate $\epsilon(D,B_n)$ for all elements of the above mentioned dictionary. The minimizer of $(\epsilon(D,B_n))_n$ is taken as the identified target and is expected to give the best estimation of $s^{\text{est}}$. This procedure is described in detail in Algorithm \ref{algo:target-identif}, which was first introduced by Ammari \textit{et al.} \cite{echolocation}. 

\begin{algorithm}
	\begin{algorithmic}
		
		\STATE Input: $(S^D_{ijk})_{ijk}$ of unknown target $D$; $((S^{B_n}_{ijl})_{ijl})_n$ of the whole dictionary.
		\FOR {$B_n$ in the dictionary}
		\STATE $\epsilon_n \leftarrow \epsilon(D, B_n)$;
		\STATE $n \leftarrow n+1$; 
		\ENDFOR
		\STATE Output: The true dictionary element $n^* \leftarrow \arg\min_n \epsilon_n$.
	\end{algorithmic}
	\caption{Target identification algorithm}
	\label{algo:target-identif}
\end{algorithm}

\subsection{Parameter settings for identification and scaling estimation} 

For this experiment, the frequency-dependent dictionary of distribution descriptors $((S^{B_n}_{ijl})_{ijl})_n$ is computed for the range of frequency $[\omega^{\text{dic}}_{\min}, \omega^{\text{dic}}_{\max}] = [0.25 \pi, 1.5 \pi]$, with $N^{\text{dic}}_{\omega} = 78$ and $N_v = 512$. Data simulation is conducted for the range of operating frequency $[\omega_{\min}, \omega_{\max}] = [0.5\pi, \pi]$ with $N_{\omega} = 52$. The range of valid scaling factor is $[s_{\min}, s_{\max}] = [0.5, 1.5]$, with $N_{\delta} = 250$.

\subsection{Results of target identification}

Now,  we present results of target identification obtained using the full-view setting of Figure \ref{fig:acq}: 
\begin{itemize}

\item It can be seen that the identification succeeded for all targets with noise $\sigma_0$ up to $50 \%$. In the case of $\sigma_0=0 \%$ (see Appendix \ref{secappendix2}), the error bars of each identified target have very different numerical value compared to those of the other elements of the dictionary. This means that recognition works well and a dictionary of large size can be used in practice. 

\item Figures \ref{fig:bars1} and \ref{fig:bars2} show the error bars for the dictionary of Figure \ref{fig:dico} for all inhomogeneous targets with noise $\sigma_0=40 \%$, $80 \%$. The $m$th error bar in the $n$th group describes the error $\epsilon(D,B_m)$ of the matching experiment using the generating element of the dictionary $B_n$. The shortest bar in each group is the target identified by the matching procedure and is marked in green; the true target is marked in red where the identification fails. For $\sigma_0=40 \%$, identification succeeded and $s^{\text{est}}$ is also close to the true value $s=1.2$, see Figure \ref{fig:scal1}. For $\sigma_0=80 \%$, identification failed for two inhomogeneous targets.
 
\begin{figure}[h!]
	\centering
		\raisebox{10ex - \height}{\includegraphics[ scale=0.4]{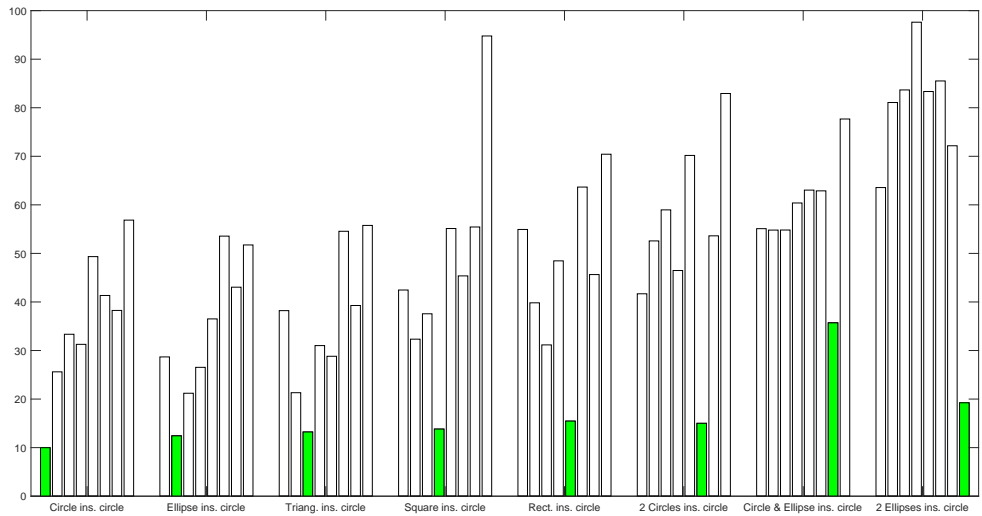}}
	\caption{\textit{Results of identification for all inhomogeneous objects in the full-view setting and $\sigma_0=40\%$. Measurements have been repeated 1000 times.}}
	\label{fig:bars1}
\end{figure} 

\begin{figure}[h!]
	\centering
	\raisebox{10ex - \height}{\includegraphics[ scale=0.4]{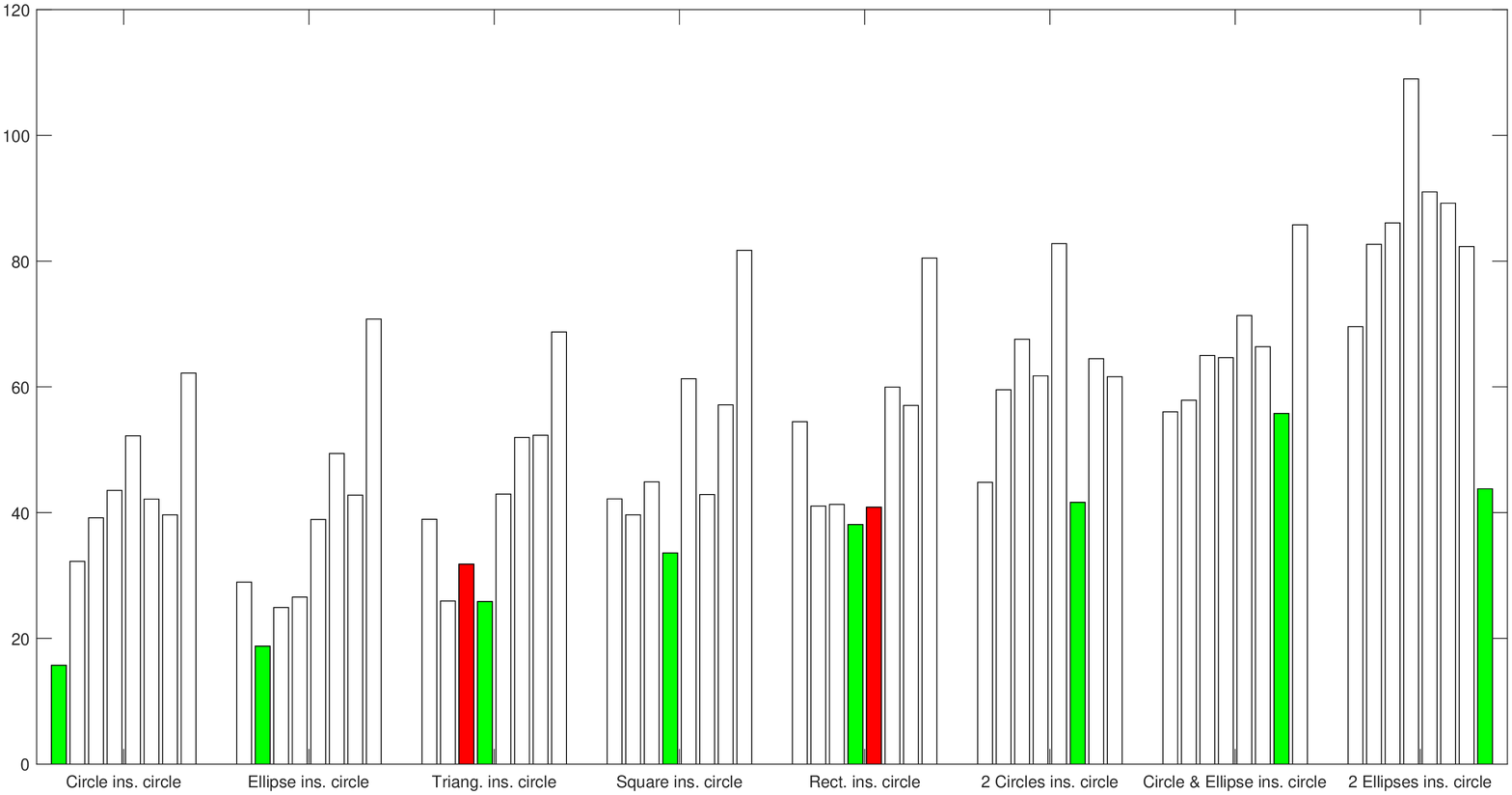}}
	\caption{\textit{Results of identification for all inhomogeneous objects in the full-view setting and $\sigma_0=80\%$. Identification failed for two targets. Measurements have been repeated 1000 times.}}
	\label{fig:bars2}
\end{figure}

\begin{figure}[t!]
	\centering
	\raisebox{10ex - \height}{\includegraphics[ scale=0.4]{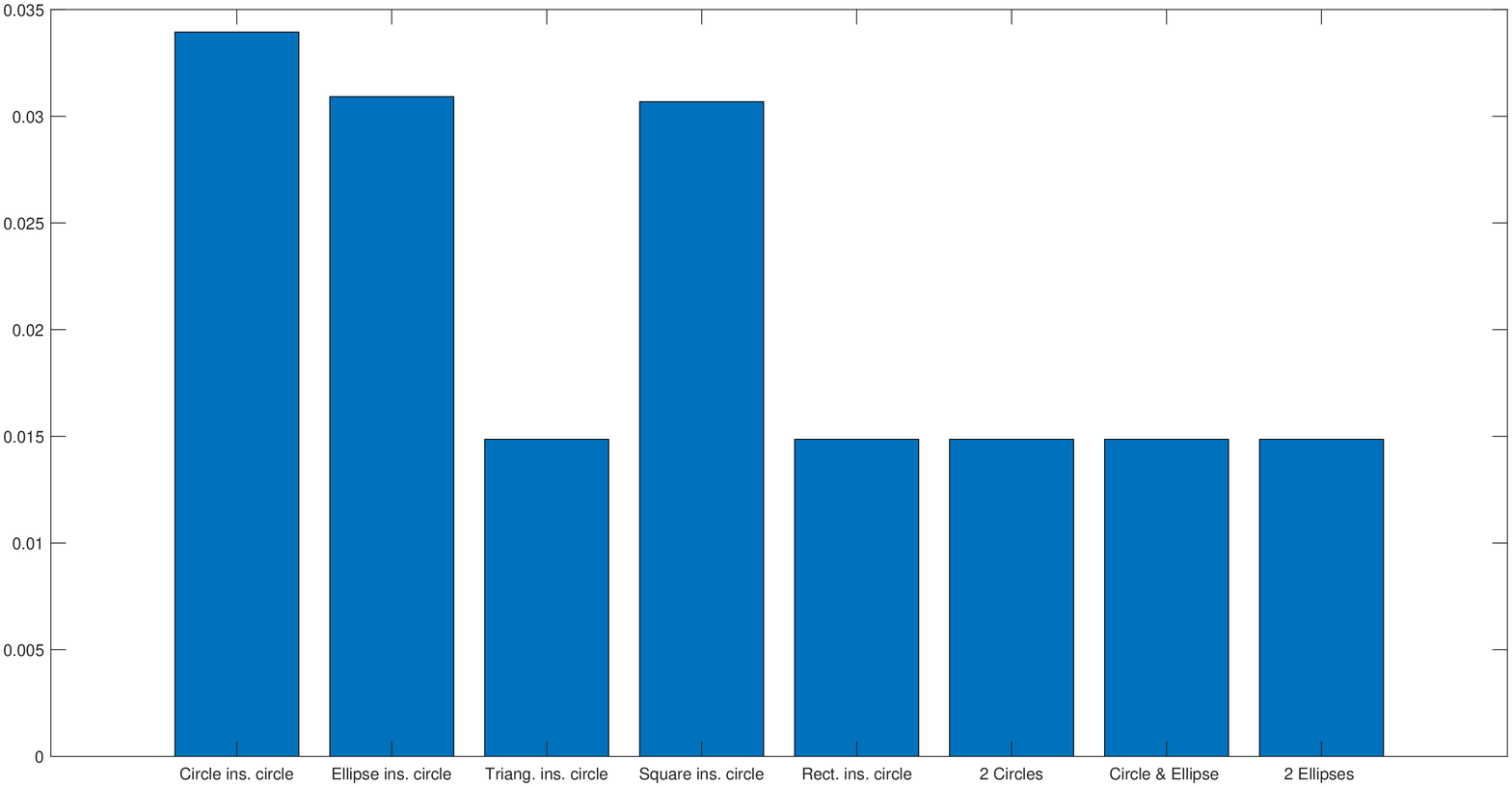}}
	\caption[width=0.5\textwidth]{\textit{Difference between the estimated scaling factor and the true one (s = 1.2) at $\sigma_0=40\%$. Measurements have been repeated 1000 times.}}
	\label{fig:scal1}
\end{figure} 

\hfill

\item Figure \ref{fig:probability} shows the probability of recognition for the inhomogeneous targets of the dictionary at different noise levels. Measurements have been repeated $1000$ times. 
\end{itemize}

\begin{figure}[H]
		\centering
		\raisebox{10ex - \height}{\includegraphics[ scale=0.4]{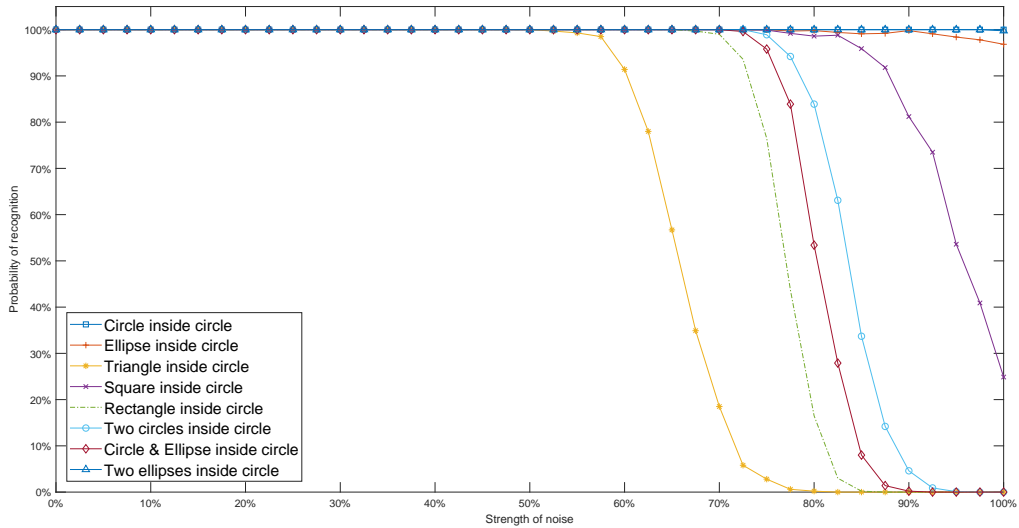}}
	\caption[width=0.5\textwidth]{\textit{Probability of recognition for all inhomogeneous targets of the dictionary of Figure \ref{fig:dico}.}}
	\label{fig:probability}
\end{figure}

\section{Concluding remarks} \label{sec7}
In this paper,  we have presented a framework of target identification for inhomogeneous objects. We have provided and numerically tested in the presence of measurement noise a procedure for target classification in wave imaging based on matching on a dictionary of precomputed frequency-dependent distribution descriptors. The construction of such frequency-dependent distribution descriptors is based on the properties of the inhomogeneous scattering coefficients. For a collection of inhomogeneous targets, we first extracted the scattering coefficients from the reflected waves and then used a target identification algorithm in order to identify an inhomogeneous target from the dictionary up to some translation, rotation and scaling. It can be seen that the identification succeeded for all targets with noise $\sigma_0$ up to $50 \%$.

\appendix

\section{Piecewise constant distributions} \label{secappendix}
In the appendix, we provide an integral representation of the solution to \eqref{helm} for the special case of a domain $B$ with piecewise constant electric permittivity $\mu$ and magnetic permeability $\sigma$. This can be seen as a particular case of \eqref{solrep}.

\subsection{The case of an inhomogeneous object with one inclusion inside} 
We consider the case of a domain $B$ with one inclusion $B_i$ inside. $B$ is immerged in an homogeneous medium. $B_i$ has different constant permeability and permittivity than the one of $B$ and the background.   

Let us consider the following Helmholtz problem
\begin{equation}
\begin{cases} \nabla \cdot \frac{1}{\sigma} \nabla u + \omega^2 \mu u = 0 & \mbox{ in } \mathbb{R}^2, \\ \left |\frac{\partial (u - U)}{\partial |x|}-i\omega(u-U) \right | \leq \frac{K}{|x|^{\frac{3}{2}}} & \mbox{ if } |x| \to \infty,
\end{cases}
\label{helm-pw}
\end{equation}
where
\begin{equation*}
\begin{cases}
\frac{1}{\sigma}(x)= \frac{1}{\sigma_i} \chi_{B_i}(x)+\frac{1}{\sigma_e} \chi_{B_e \setminus \overline{B}_i}(x) + \frac{1}{\sigma_0} \chi_{\mathbb{R}^2 \setminus \overline{B}_e}(x), \\
\mu(x)= \mu_i \chi_{B_i}(x)+ \mu_e \chi_{B_e \setminus \overline{B}_i}(x) + \mu_0 \chi_{\mathbb{R}^2 \setminus \overline{B}_e}(x),
\end{cases}
\end{equation*}
with $B_i \subset B_e=B$. Let us define $k_0 = \omega$, $k_e = \omega \sqrt{\sigma_e \mu_e}$ and $k_i = \omega \sqrt{\sigma_i \mu_i}$. Solution to \eqref{helm-pw} should satisfy
\begin{equation}
\begin{cases} \Delta u + k^2_0 u = 0 & \mbox{ in } \mathbb{R}^2 \setminus \overline{B}_e, \\ \Delta u + k^2_e u = 0 & \mbox{ in } B_e \setminus \overline{B}_i, \\ \Delta u + k^2_i u = 0 & \mbox{ in } B_i, \\ \left |\frac{\partial (u - U)}{\partial |x|}-i\omega(u-U) \right | \leq \frac{K}{|x|^{\frac{3}{2}}} & \mbox{ if } |x| \to \infty,
\end{cases}
\label{helm-pw-explicit}
\end{equation}
with the following transmission conditions
\begin{equation}
\begin{cases} u|_{+} = u|_{-} & \mbox{ on } \partial B_e, \\ u|_{+} = u|_{-} & \mbox{ on } \partial B_i, \\ \left. \frac{1}{\sigma_0} \frac{\partial u}{\partial \nu}\right|_{+} = \left. \frac{1}{\sigma_e} \frac{\partial u}{\partial \nu} \right|_{-} & \mbox{ on } \partial B_e, \\ \left. \frac{1}{\sigma_e} \frac{\partial u}{\partial \nu}\right|_{+} = \left. \frac{1}{\sigma_i} \frac{\partial u}{\partial \nu} \right|_{-} & \mbox{ on } \partial B_i.
\end{cases}
\label{transm-pw}
\end{equation}

Given the cylindrical wave $\mathcal{C}_n$ of index $n \in \mathbb{Z}$ and of wave number $k_0$, we look for a solution to \eqref{helm-pw} of the form 
\begin{equation}
u_n(x) = \begin{cases} \mathcal{C}_n(x)+S^{k_0}_{B_e}[\phi](x) & \mbox{ in } \mathbb{R}^2 \setminus \overline{B}_e, \\ S^{k_e}_{B_e}[\gamma](x) + S^{k_e}_{B_i}[\eta](x) & \mbox{ in } B_e \setminus \overline{B}_i, \\ S^{k_i}_{B_i}[\psi](x) & \mbox{ in } B_i,
\end{cases}
\label{numerical-sol}
\end{equation}
where the densities $\psi_n, \gamma_n, \eta_n$ and $\phi_n$ are the solutions to
\begin{equation}
\begin{cases} S^{k_i}_{B_i}[\psi_n](x)= S^{k_e}_{B_e}[\gamma_n](x) + S^{k_e}_{B_i}[\eta_n](x)  & \mbox{ on } \partial B_i, \\ S^{k_e}_{B_e}[\gamma_n](x) + S^{k_e}_{B_i}[\eta_n](x) = \mathcal{C}_n(x)+S^{k_0}_{B_e}[\phi_n](x) & \mbox{ on } \partial B_e, \\ \left. \frac{1}{\sigma_i}\frac{\partial  S^{k_i}_{B_i}[\psi_n]}{\partial \nu}\right|_{-}  = \left. \frac{1}{\sigma_e}\frac{\partial S^{k_e}_{B_e}[\gamma_n]}{\partial \nu}\right|_{+} + \left. \frac{1}{\sigma_e}\frac{\partial S^{k_e}_{B_i}[\eta_n]}{\partial \nu}\right|_{+} & \mbox{ on } \partial B_i, \\ \left. \frac{1}{\sigma_e}\frac{\partial S^{k_e}_{B_e}[\gamma_n]}{\partial \nu}\right|_{-} + \left. \frac{1}{\sigma_e}\frac{\partial S^{k_e}_{B_i}[\eta_n]}{\partial \nu}\right|_{-}  = \left. \frac{1}{\sigma_0}\frac{\partial \mathcal{C}_n}{\partial \nu}\right|_{+} + \left. \frac{1}{\sigma_0}\frac{\partial S^{k_0}_{B_e}[\phi_n]}{\partial \nu}\right|_{+} & \mbox{ on } \partial B_e.
\end{cases}
\label{transm-numerical}
\end{equation}

As we have proved in \eqref{scatcoef}, the scattering coefficient of order $n,m$ associated to the target $B$ with (inhomogeneous) piecewise constant permittivity and permeability is
$$W_{n,m} [B, \sigma, \mu, \omega] =\int_{ \partial B_e } \overline{C_n(y)}\phi_m(y) dS_y.$$

\subsection{The case of an inhomogeneous object with two (distinct) inclusions inside} 

Now we take into account the case of a domain $B$ with two inclusions $B_1$ and $B_2$ inside. $B$ is immerged in a homogeneous medium. $B_1$ and $B_2$ have different constant permeability and permittivity than the one of $B$ and the background.   

Let us consider the following Helmholtz problem
\begin{equation}
\begin{cases} \nabla \cdot \frac{1}{\sigma} \nabla u + \omega^2 \mu u = 0 & \mbox{ in } \mathbb{R}^2, \\ \left |\frac{\partial (u - U)}{\partial |x|}-i\omega(u-U) \right | \leq \frac{K}{|x|^{\frac{3}{2}}} & \mbox{ if } |x| \to \infty,
\end{cases}
\label{helm-pw-2}
\end{equation}
where
$$\frac{1}{\sigma}(x)= \frac{1}{\sigma_1} \chi_{B_1}(x)+\frac{1}{\sigma_2} \chi_{B_2}(x) + \frac{1}{\sigma_e} \chi_{B_e \setminus \overline{(B_1 \cup B_2)}}(x) + \frac{1}{\sigma_0} \chi_{\mathbb{R}^2 \setminus \overline{B}_e}(x),$$
$$\mu(x)= \mu_1 \chi_{B_1}(x)+ \mu_2 \chi_{B_2}(x) + \mu_e \chi_{B_e \setminus \overline{(B_1 \cup B_2)}}(x) + \mu_0 \chi_{\mathbb{R}^2 \setminus \overline{B}_e}(x),$$
with $B_1, B_2 \subset B_e=B$. Let us define $k_0 = \omega$, $k_e = \omega \sqrt{\sigma_e \mu_e}$ and $k_i = \omega \sqrt{\sigma_i \mu_i}$ for $i=1,2$. The solution to \eqref{helm-pw-2} should satisfy
\begin{equation}
\begin{cases} \Delta u + k^2_0 u = 0 & \mbox{ in } \mathbb{R}^2 \setminus \overline{B}_e, \\ \Delta u + k^2_e u = 0 & \mbox{ in } B_e \setminus \overline{ {B}_1 \cup B_2}, \\ \Delta u + k^2_2 u = 0 & \mbox{ in } B_2, \\ \Delta u + k^2_1 u = 0 & \mbox{ in } B_1, \\ \left |\frac{\partial (u - U)}{\partial |x|}-i\omega(u-U) \right | \leq \frac{K}{|x|^{\frac{3}{2}}} & \mbox{ if } |x| \to \infty,
\end{cases}
\label{helm-pw-explicit-2}
\end{equation}
with the following transmission conditions
\begin{equation}
\begin{cases} u|_{+} = u|_{-} & \mbox{ on } \partial B_e, \\ u|_{+} = u|_{-} & \mbox{ on } \partial B_2, \\ u|_{+} = u|_{-} & \mbox{ on } \partial B_1, \\ \left. \frac{1}{\sigma_0} \frac{\partial u}{\partial \nu}\right|_{+} = \left. \frac{1}{\sigma_e} \frac{\partial u}{\partial \nu} \right|_{-} & \mbox{ on } \partial B_e, \\ \left. \frac{1}{\sigma_e} \frac{\partial u}{\partial \nu}\right|_{+} = \left. \frac{1}{\sigma_2} \frac{\partial u}{\partial \nu} \right|_{-} & \mbox{ on } \partial B_2, \\ \left. \frac{1}{\sigma_e} \frac{\partial u}{\partial \nu}\right|_{+} = \left. \frac{1}{\sigma_1} \frac{\partial u}{\partial \nu} \right|_{-} & \mbox{ on } \partial B_1.
\end{cases}
\label{transm-pw-2}
\end{equation}
As in the previous case, we look for a solution to \eqref{helm-pw-2} of the form 
\begin{equation}
u_n(x) = \begin{cases} \mathcal{C}_n(x)+S^{k_0}_{B_e}[\phi](x) & \mbox{ in } \mathbb{R}^2 \setminus \overline{B}_e, \\ S^{k_e}_{B_e}[\gamma](x) + S^{k_e}_{B_2}[\eta](x) + S^{k_e}_{B_1}[\zeta](x) & \mbox{ in } B_e \setminus \overline{{B}_1 \cup B_2}, \\ S^{k_2}_{B_2}[\psi](x) & \mbox{ in } B_2, \\ S^{k_1}_{B_1}[\xi](x) & \mbox{ in } B_1, 
\end{cases}
\label{numerical-sol-2}
\end{equation}
where the densities $\phi_n, \gamma_n, \eta_n, \zeta_n, \psi_n$ and $\xi_n$ are the solutions to
\begin{equation}
\begin{cases} S^{k_1}_{B_1}[\xi_n](x)= S^{k_e}_{B_e}[\gamma_n](x) + S^{k_e}_{B_2}[\eta_n](x) + S^{k_e}_{B_1}[\zeta_n](x) & \mbox{ on } \partial B_1, \\ S^{k_2}_{B_2}[\psi_n](x)= S^{k_e}_{B_e}[\gamma_n](x) + S^{k_e}_{B_2}[\eta_n](x) + S^{k_e}_{B_1}[\zeta_n](x) & \mbox{ on } \partial B_2, \\ S^{k_e}_{B_e}[\gamma_n](x) + S^{k_e}_{B_2}[\eta_n](x) + S^{k_e}_{B_1}[\zeta_n](x) = \mathcal{C}_n(x)+S^{k_0}_{B_e}[\phi_n](x) & \mbox{ on } \partial B_e, \\ \left. \frac{1}{\sigma_1}\frac{\partial  S^{k_1}_{B_1}[\psi_n]}{\partial \nu}\right|_{-}  = \left. \frac{1}{\sigma_e}\frac{\partial S^{k_e}_{B_e}[\gamma_n]}{\partial \nu}\right|_{+} + \left. \frac{1}{\sigma_e}\frac{\partial S^{k_e}_{B_2}[\eta_n]}{\partial \nu}\right|_{+} + \left. \frac{1}{\sigma_e}\frac{\partial S^{k_e}_{B_1}[\zeta_n]}{\partial \nu}\right|_{+} & \mbox{ on } \partial B_1, \\ \left. \frac{1}{\sigma_2}\frac{\partial  S^{k_2}_{B_2}[\xi_n]}{\partial \nu}\right|_{-}  = \left. \frac{1}{\sigma_e}\frac{\partial S^{k_e}_{B_e}[\gamma_n]}{\partial \nu}\right|_{+} + \left. \frac{1}{\sigma_e}\frac{\partial S^{k_e}_{B_2}[\eta_n]}{\partial \nu}\right|_{+} + \left. \frac{1}{\sigma_e}\frac{\partial S^{k_e}_{B_1}[\zeta_n]}{\partial \nu}\right|_{+} & \mbox{ on } \partial B_2, \\ \left. \frac{1}{\sigma_e}\frac{\partial S^{k_e}_{B_e}[\gamma_n]}{\partial \nu}\right|_{-} + \left. \frac{1}{\sigma_e}\frac{\partial S^{k_e}_{B_i}[\eta_n]}{\partial \nu}\right|_{-}  + \left. \frac{1}{\sigma_e}\frac{\partial S^{k_e}_{B_1}[\zeta_n]}{\partial \nu}\right|_{-} = \left. \frac{1}{\sigma_0}\frac{\partial \mathcal{C}_n}{\partial \nu}\right|_{+} + \left. \frac{1}{\sigma_0}\frac{\partial S^{k_0}_{B_e}[\phi_n]}{\partial \nu}\right|_{+} & \mbox{ on } \partial B_e.
\end{cases}
\label{transm-numerical-2}
\end{equation}

Again, the scattering coefficient of order $n,m$ associated to the target $B$ with (inhomogeneous) piecewise constant permittivity and permeability is given by
$$W_{n,m} [b, \sigma, \mu, \omega] =\int_{ \partial B_e } \overline{C_n(y)}\phi_m(y) dS_y.$$

\section{Target identification with $\sigma_0=0\%$} \label{secappendix2}

We present results of target identification obtained using the full-view setting of Figure \ref{fig:acq} with no noise ($\sigma_0=0\%$). The computation of the error $\epsilon(D,B_n)$ is represented by error bars in Figure $8$, where the $m$th error bar in the $n$th figure corresponds to the error $\epsilon(D,B_m)$ of the matching experiment using the generating element of the dictionary $B_n$. The shortest bar in each group is the identified target and is marked in green, while the true target is marked in red where the identification fails.

\begin{figure}[H]
	\begin{tabular}{cc}
		\includegraphics[width=75mm]{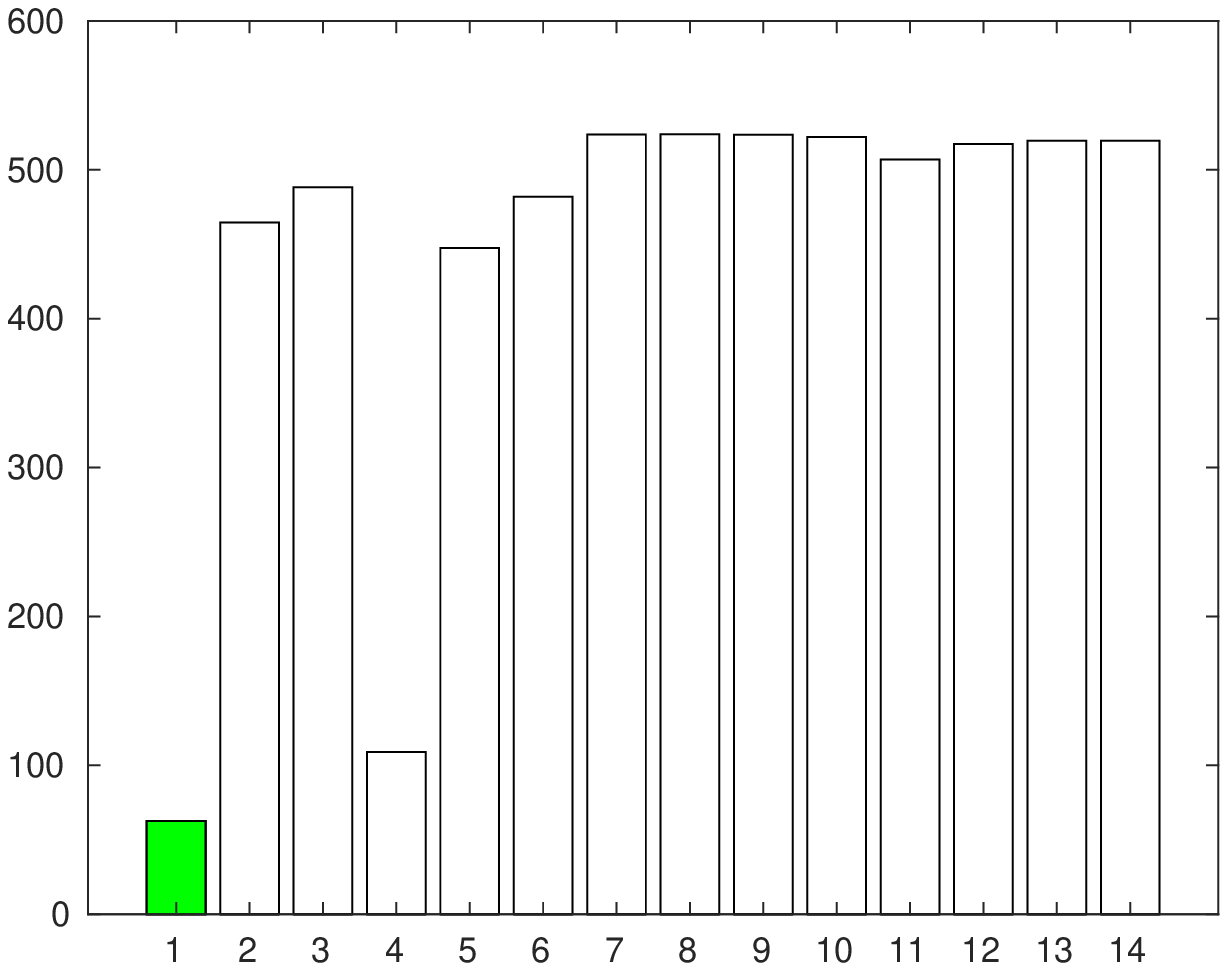} &   \includegraphics[width=75mm]{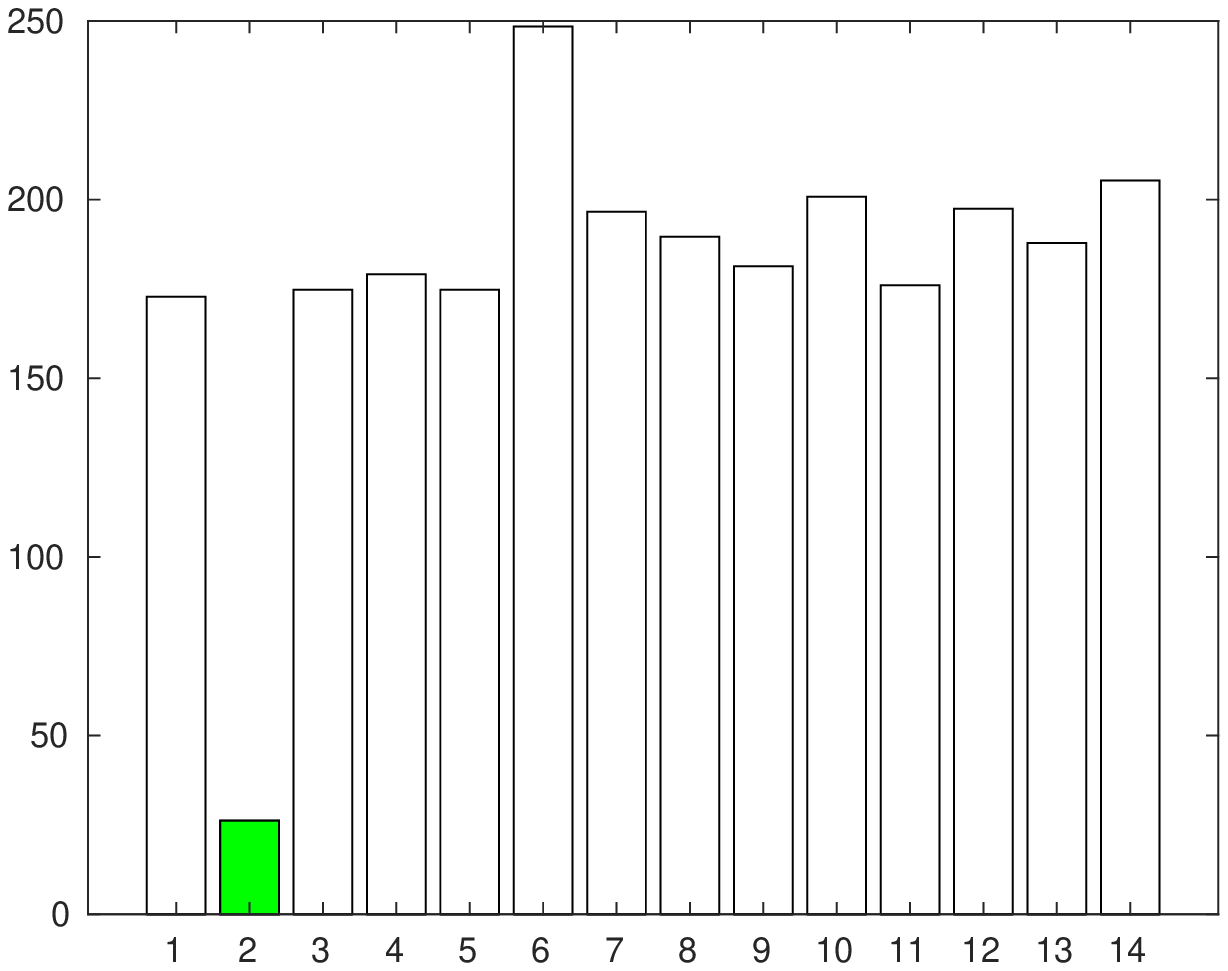} \\
		(1) Disk & (2) Ellipse \\
	\end{tabular}
\end{figure}

\vspace{5mm}

\begin{figure}[H]
	\begin{tabular}{cc}
		\includegraphics[width=75mm]{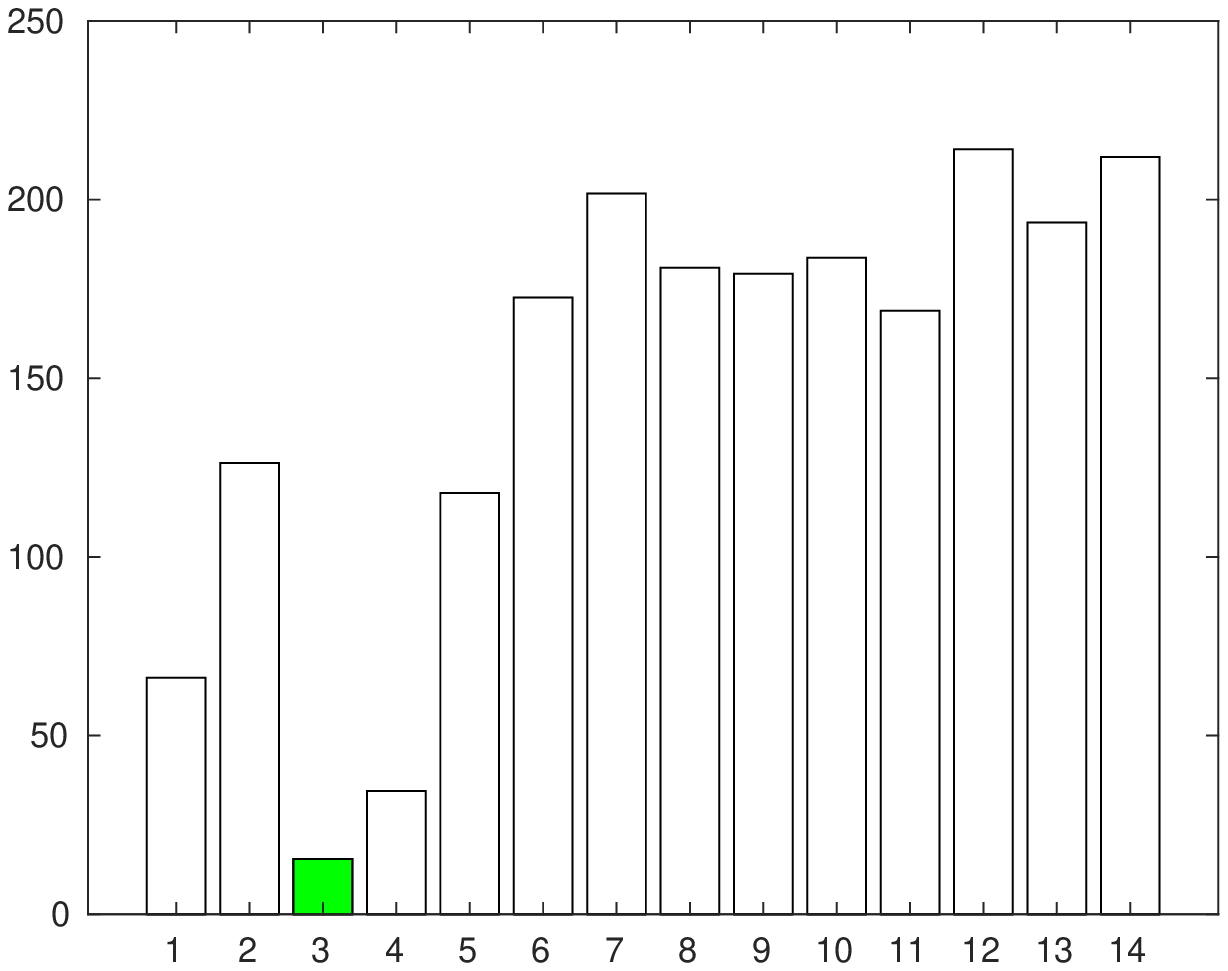} &   \includegraphics[width=75mm]{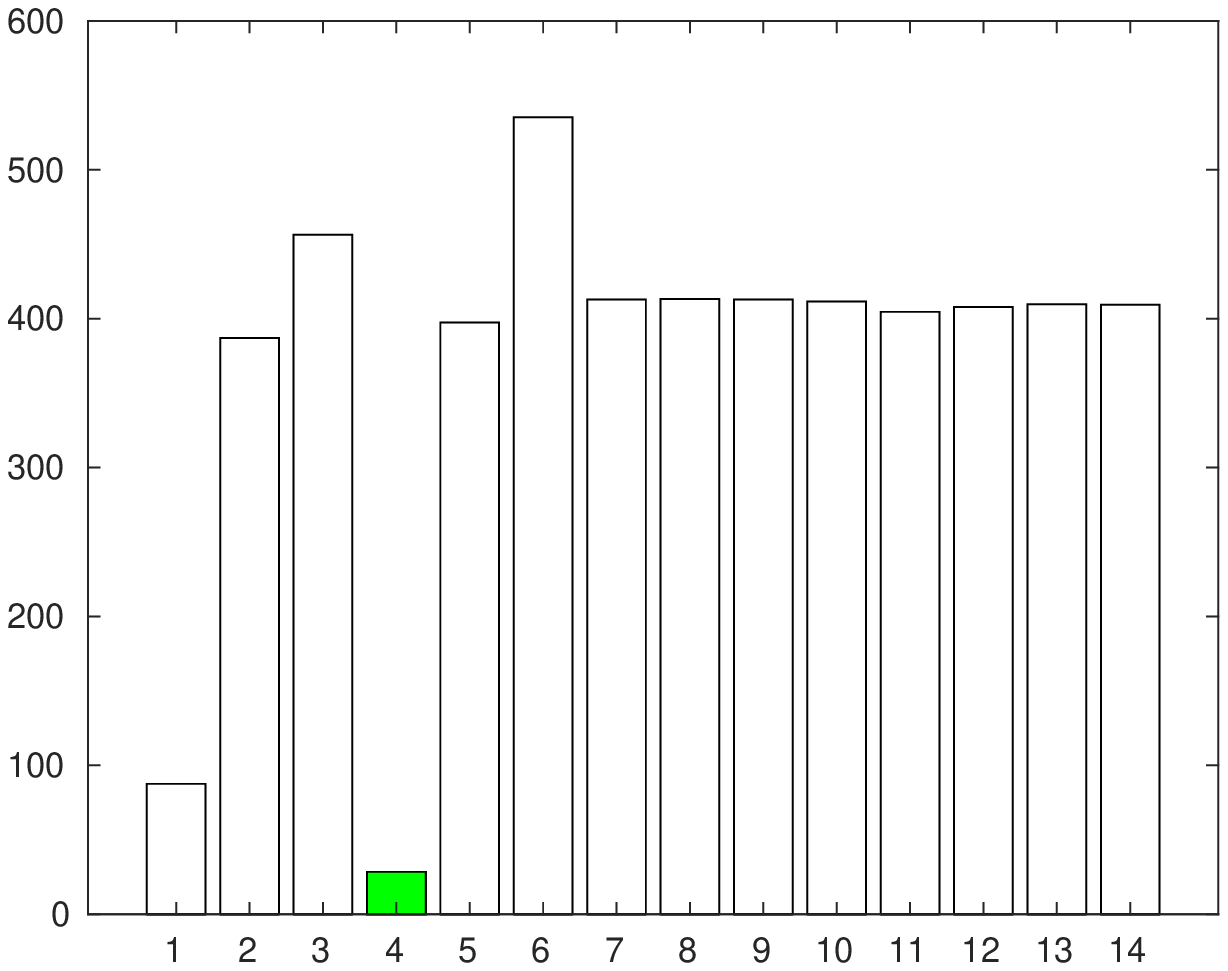} \\
		(3) Triangle & (4) Square \\[4pt]
		\includegraphics[width=75mm]{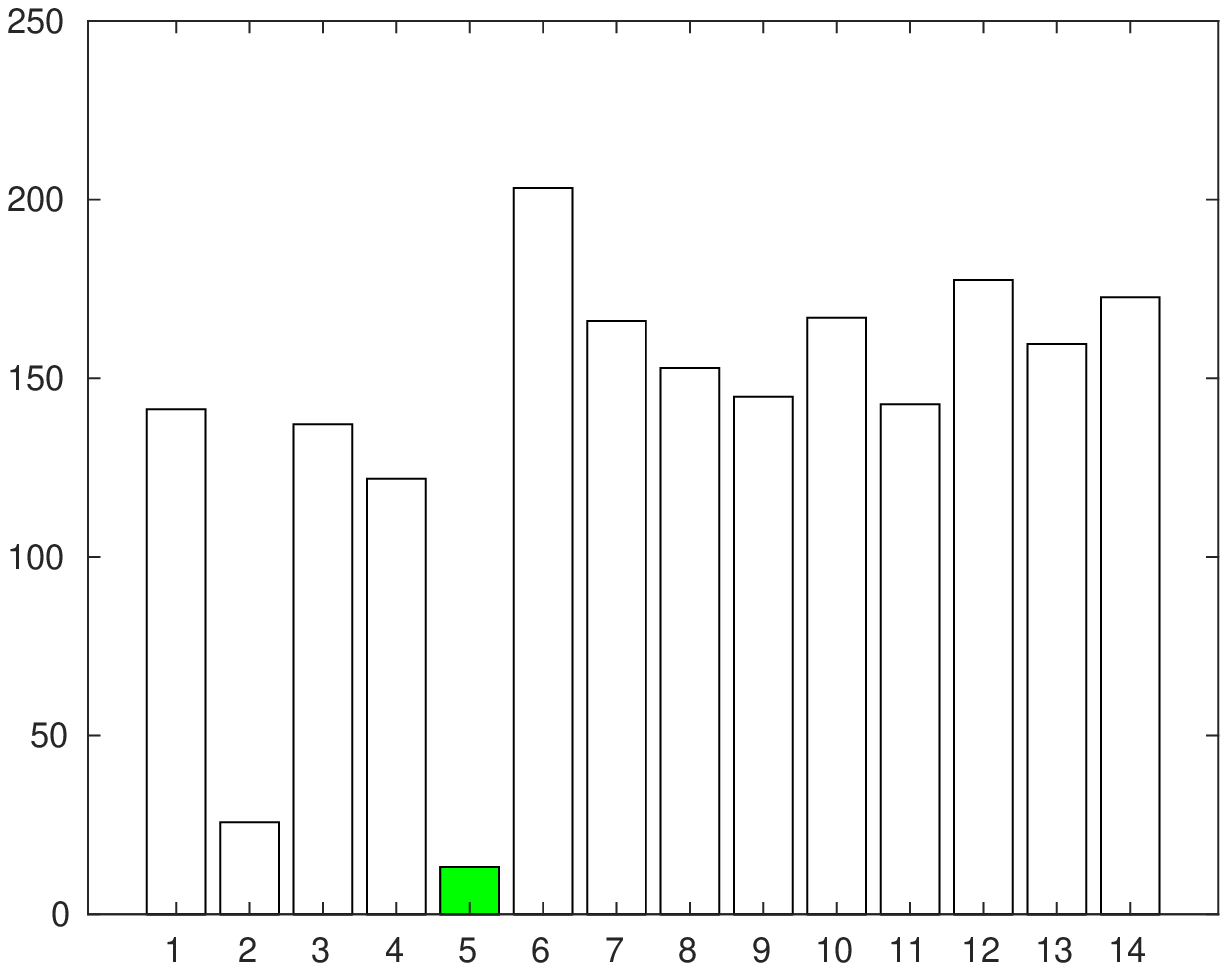} &   \includegraphics[width=75mm]{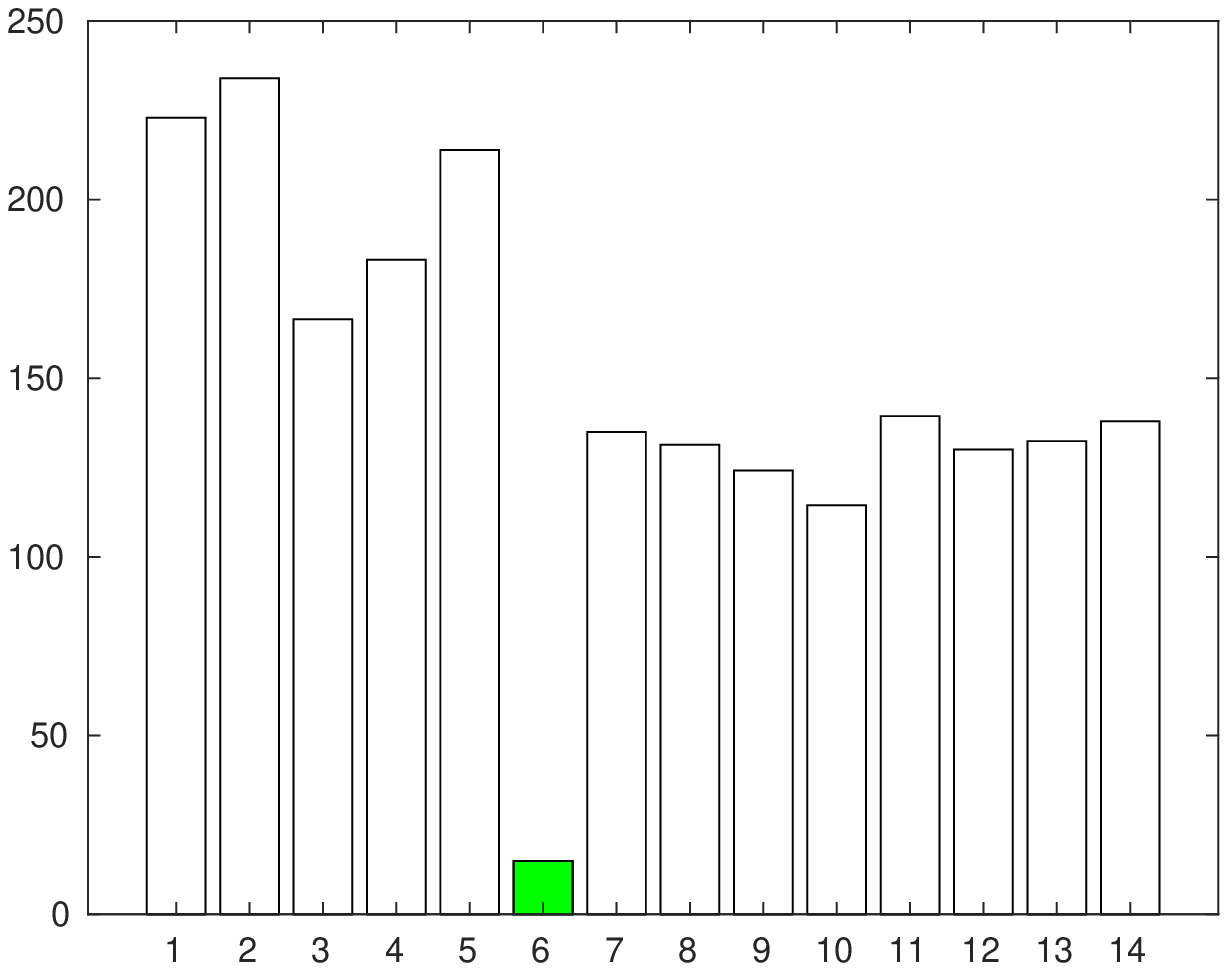} \\
		(5) Rectangle & (6) Letter A \\[4pt]
		\includegraphics[width=75mm]{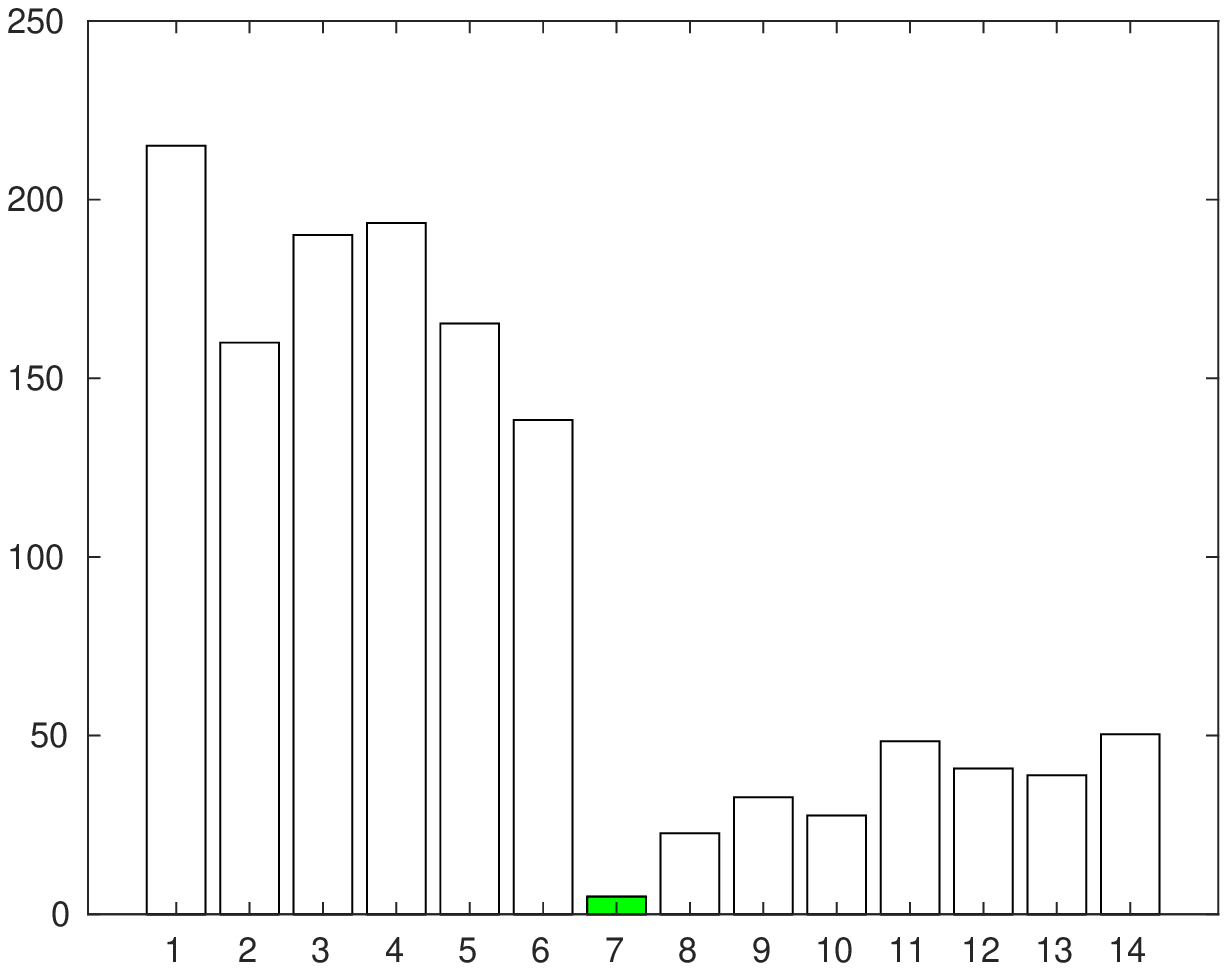} &   \includegraphics[width=75mm]{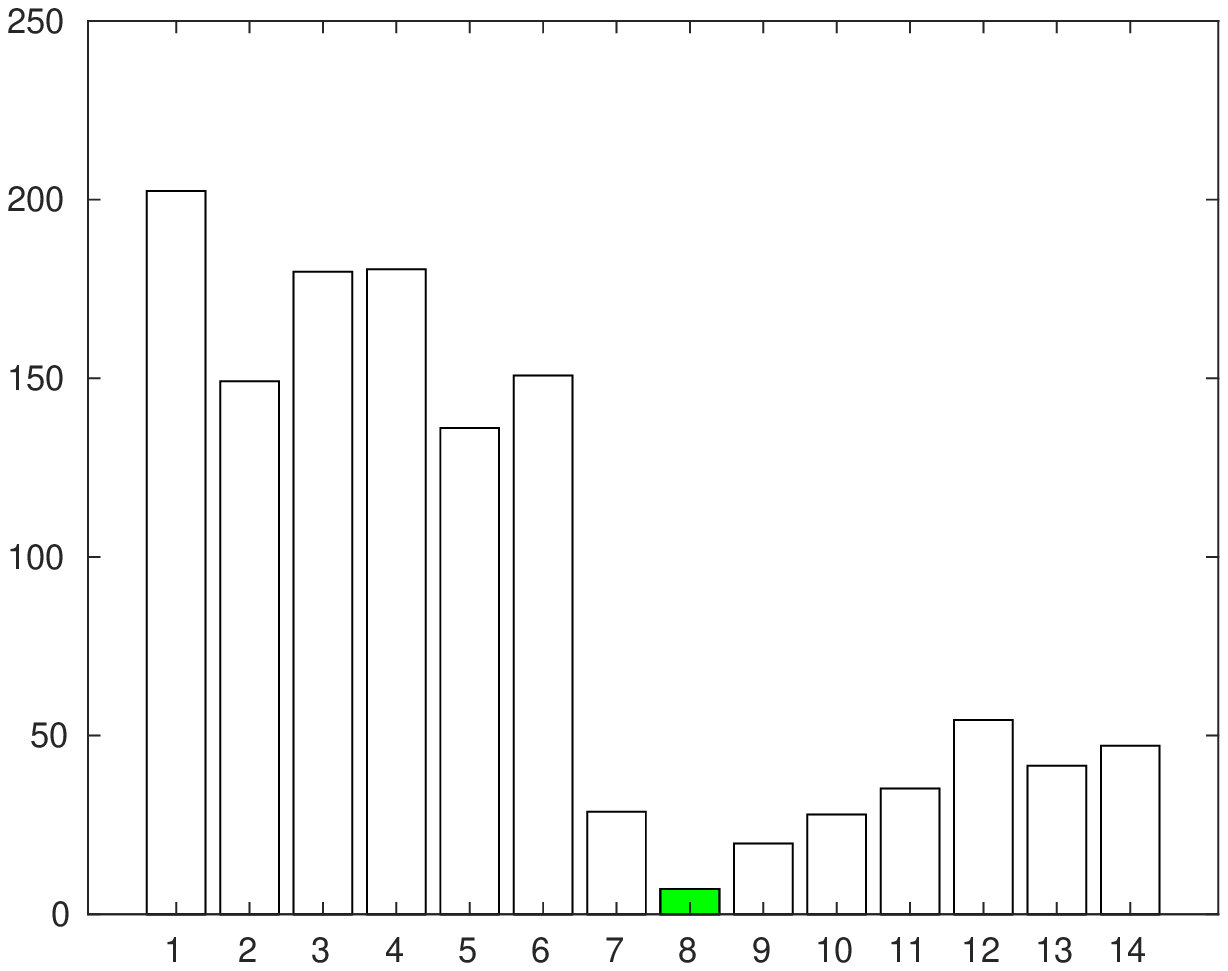} \\
		(7) Disk with a circular inclusion & (8) Ellipse inside a disk
	\end{tabular}
\end{figure}

\begin{figure}[H]
	\begin{tabular}{cc}
		\includegraphics[width=75mm]{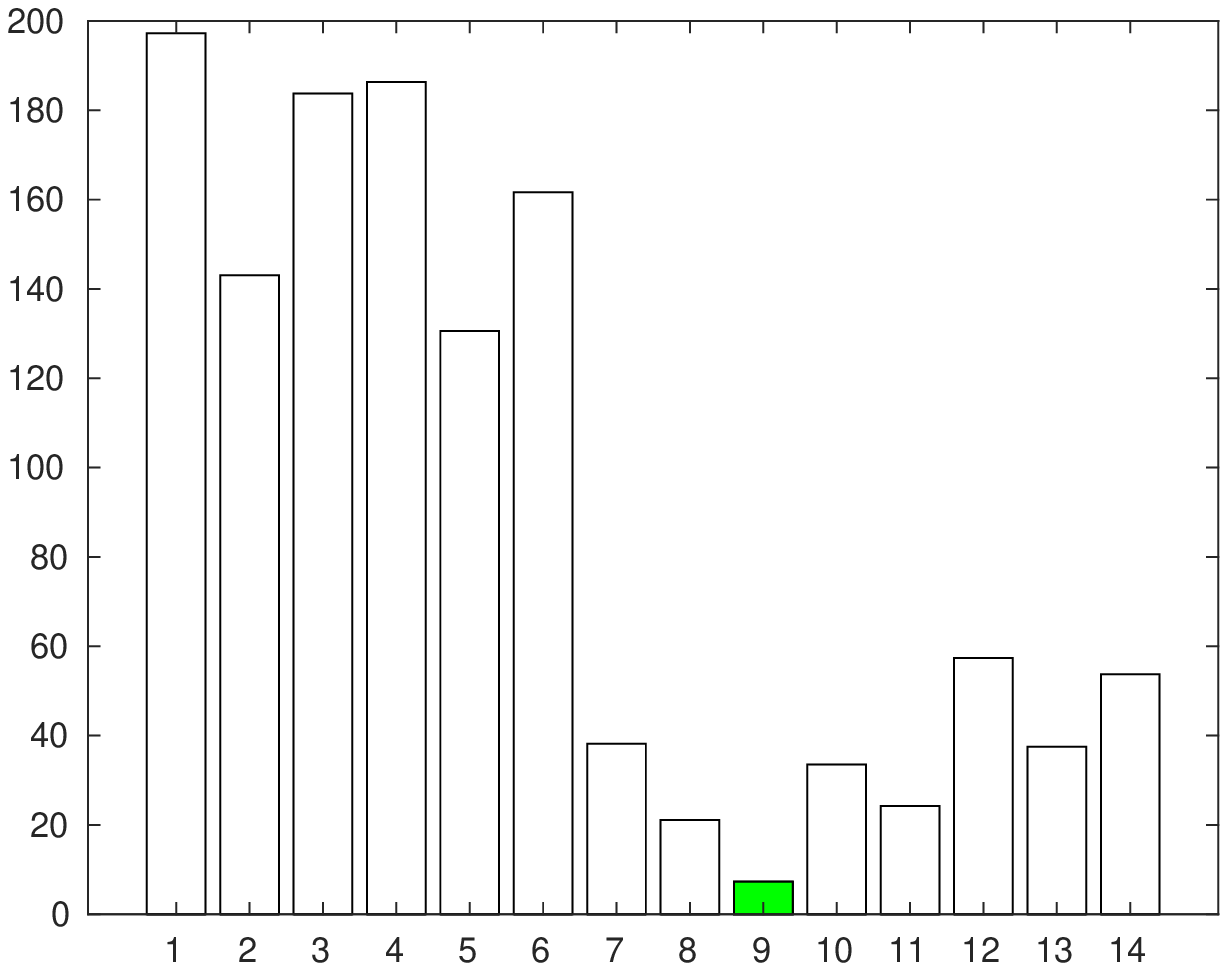} &   \includegraphics[width=75mm]{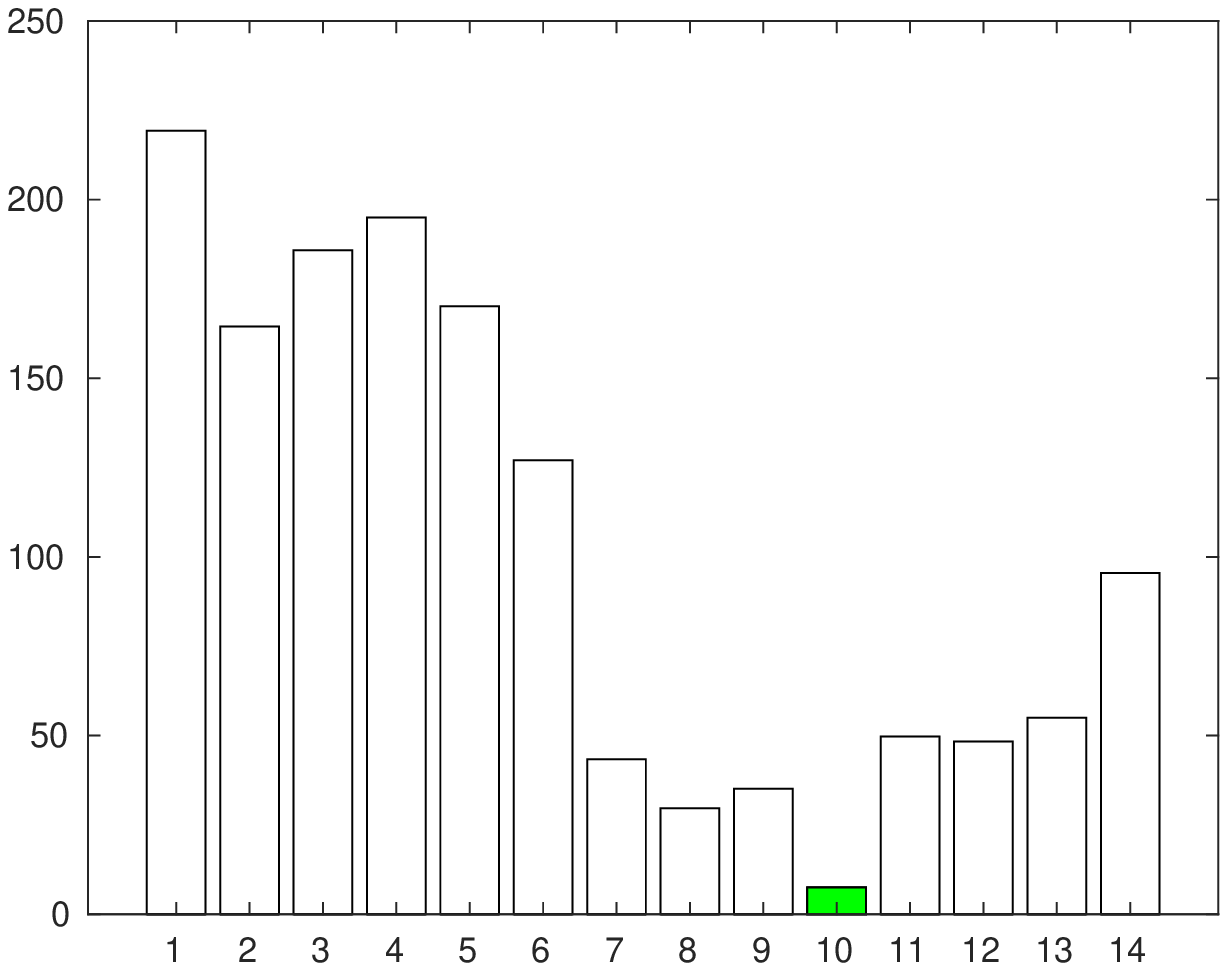} \\
		(9) Triangle inside a disk & (10) Square inside a disk \\[4pt]
		\includegraphics[width=75mm]{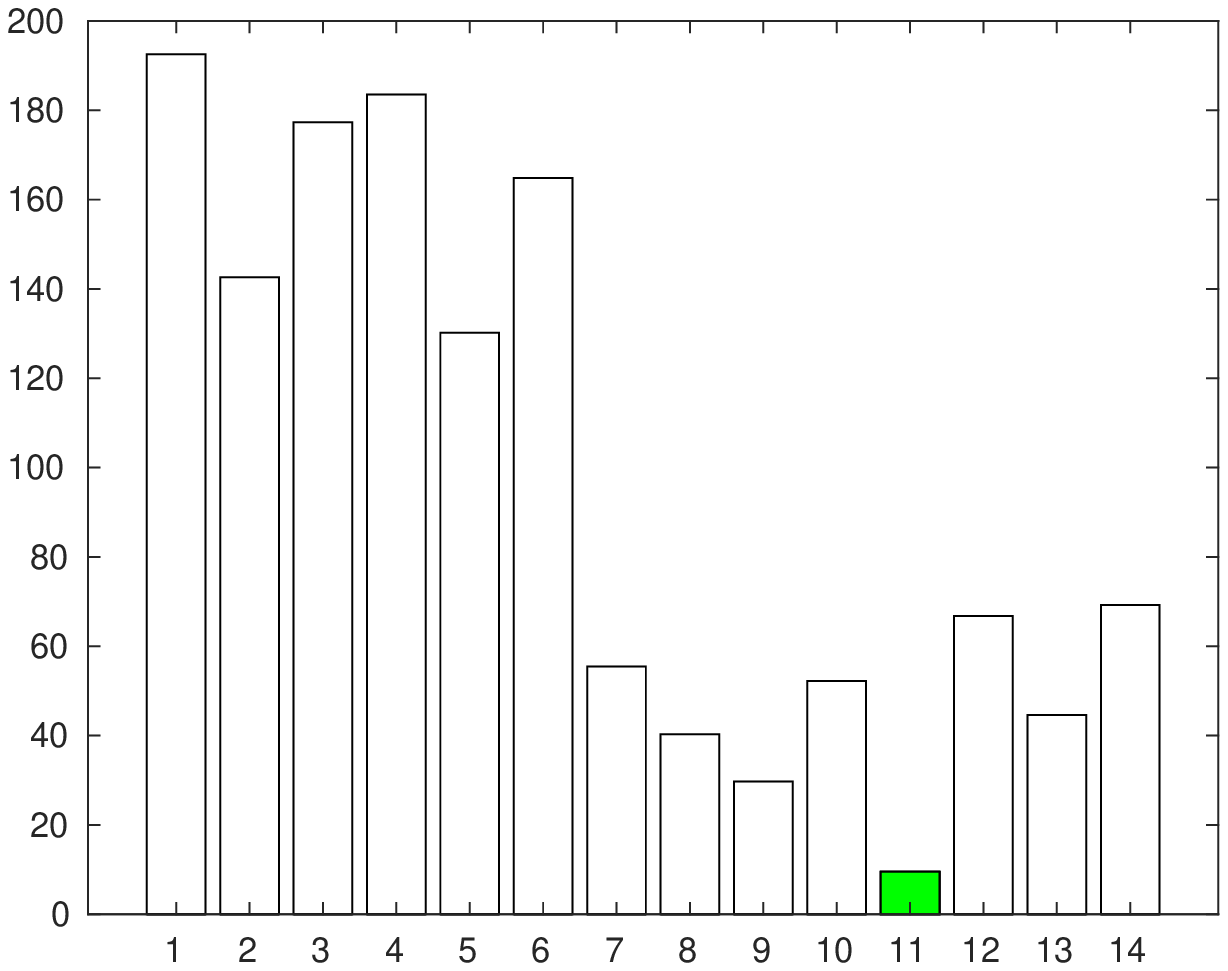} &   \includegraphics[width=75mm]{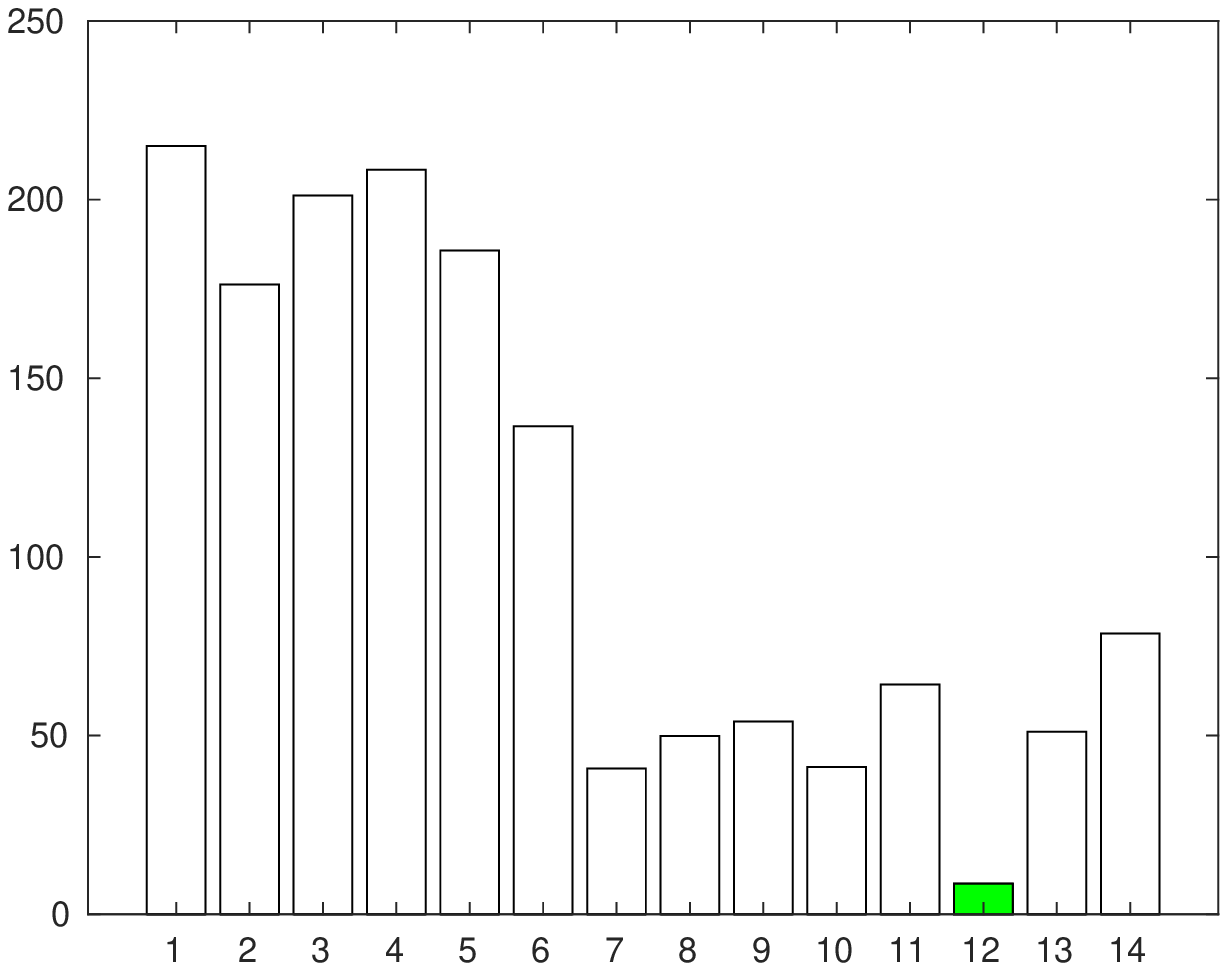} \\
		(11) Rectangle inside a disk & (12) Disk with two circular inclusions \\[4pt]
		\includegraphics[width=75mm]{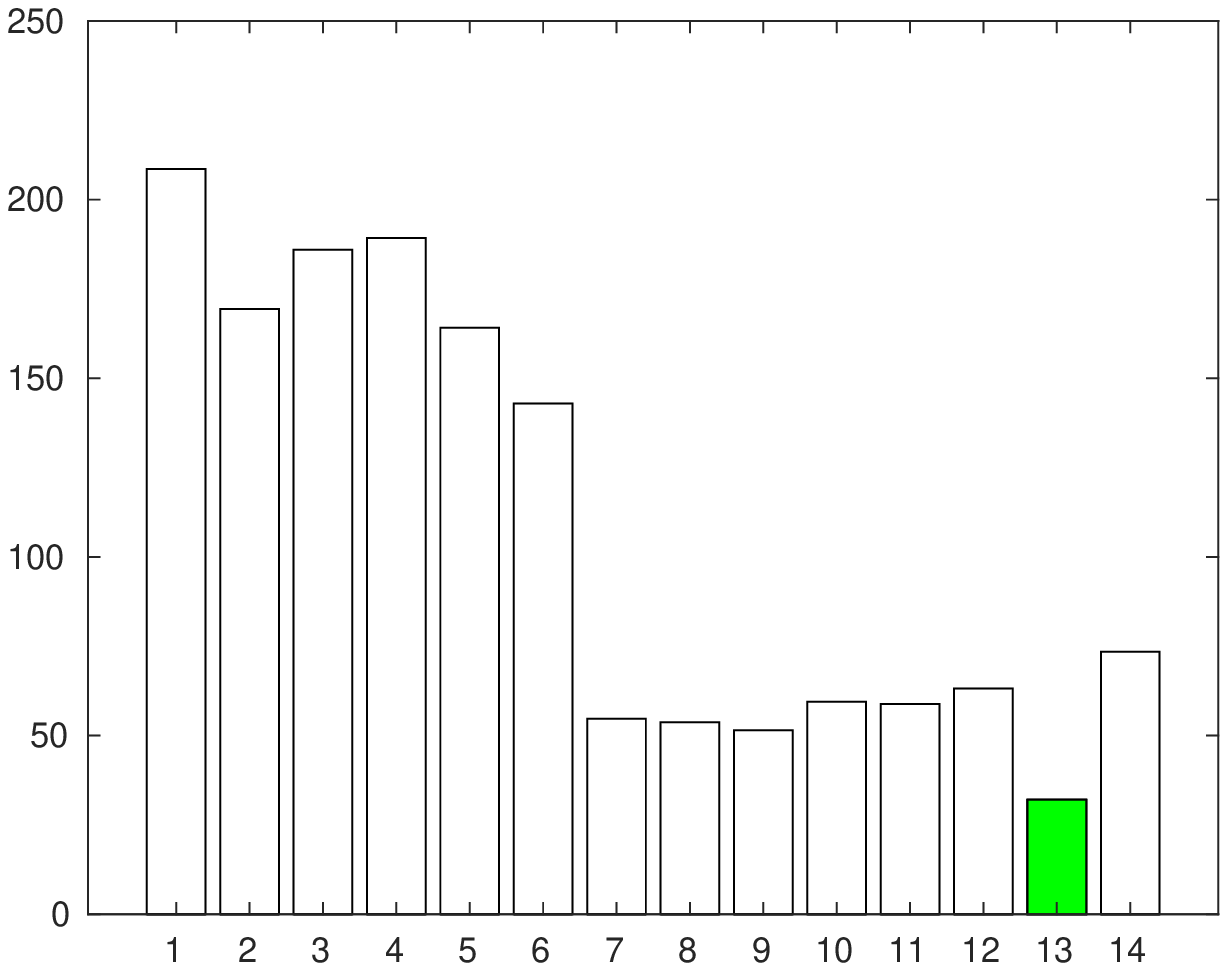} &   \includegraphics[width=75mm]{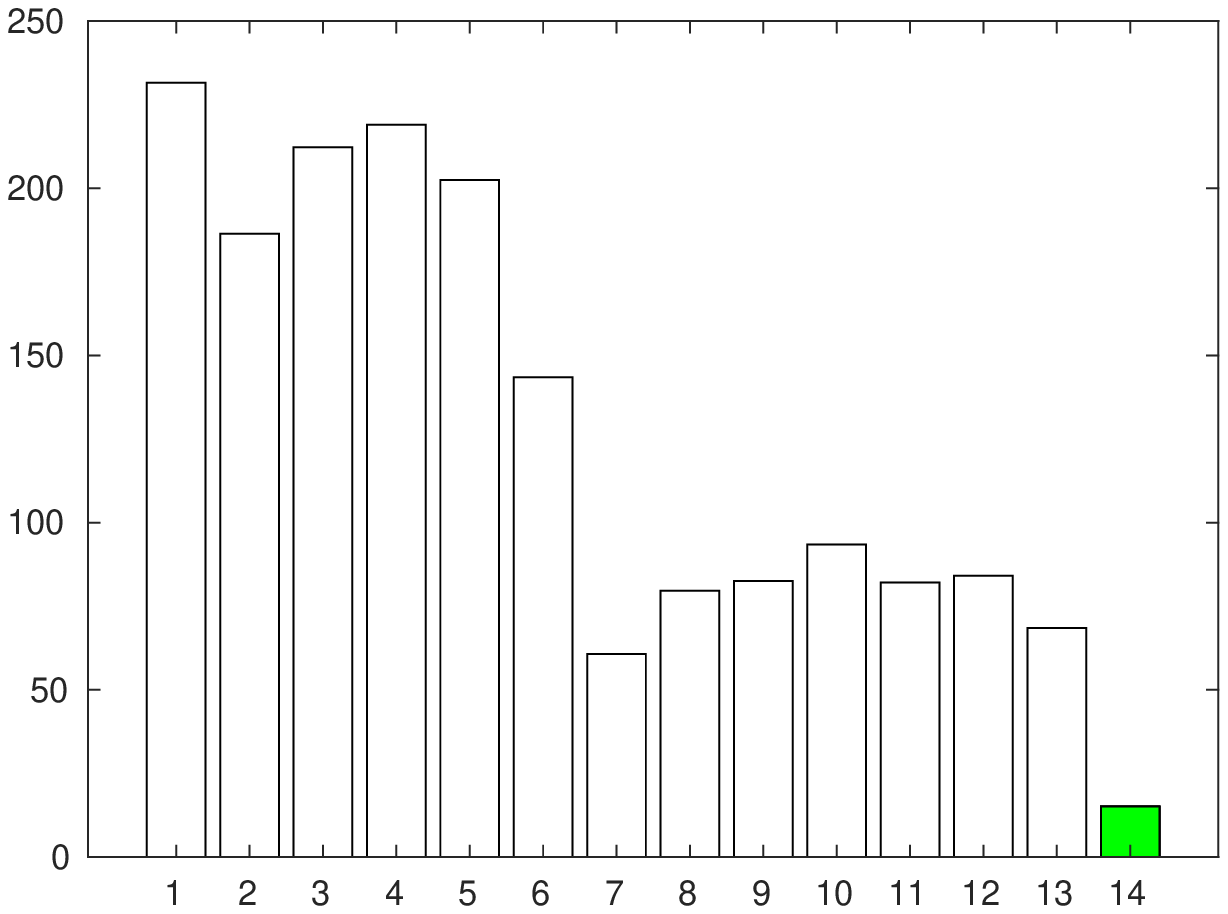} \\
		(13) Disk with a disk and an ellipse inside & (l4) Disk with two ellipses inside
	\end{tabular}
	\label{fig:bar}
	\caption{\textit{Results of identification for all elements of the dictionary in the full-view and no noise.}} 
\end{figure}

\section*{Acknowledgments}
The author gratefully acknowledges Prof. H. Ammari for his guidance. During the preparation of this work, the author was financially supported by a Swiss National Science Foundation grant (number 200021-172483).

\end{document}